\documentclass[a4paper, hidelinks, 11pt]{article}
\usepackage{graphicx} % Include figure files
\usepackage{amsmath, relsize} 
\usepackage{todonotes} 
\usepackage{amsthm} 
\usepackage{amsfonts}
\usepackage{amssymb}
\usepackage{url}
\usepackage{mathrsfs}
\usepackage{amsmath}
\usepackage{amssymb}
\usepackage{amsthm}
\usepackage{bbm}
\usepackage{comment}
\usepackage{subcaption}
\usepackage{tocloft}

\usepackage{enumerate}
\usepackage{setspace}
\usepackage{graphicx,color,hyperref}
\usepackage[margin=0.55in]{geometry}
\parskip \medskipamount
\newtheorem{theorem}{Theorem}[section]
\newtheorem{proposition}[theorem]{Proposition}

\newtheorem{lemma}[theorem]{Lemma}
\theoremstyle{definition}
\newtheorem{definition}[theorem]{Definition}
\newtheorem{remark}[theorem]{Remark}
\newtheorem{example}[theorem]{Example}

\newcommand{\G}     {\mathbb{G}} 
\newcommand{\R}     {\mathbb{R}} 
\newcommand{\Z}     {\mathbb{Z}} 
\newcommand{\N}     {\mathbb{N}} 
\renewcommand{\P}   {\mathbb{P}} 
 
\newcommand{\E}     {\mathbb{E}} 
 
\newcommand{\T}     {\mathbb{T}}

\newcommand{\Acal}   {{\mathcal A }}

\newcommand{\Ccal}   {{\mathcal C }} 
 
\newcommand{\Ecal}   {{\mathcal E }} 
 
\newcommand{\Gcal}   {{\mathcal G }} 
\newcommand{\Hcal}   {{\mathcal H }}

\newcommand{\Lcal}   {{\mathcal L }} 
\newcommand{\Mcal}   {{\mathcal M }} 
\newcommand{\Ncal}   {{\mathcal N }} 
 
\newcommand{\Pcal}   {{\mathcal P }} 
 
\newcommand{\Rcal}   {{\mathcal R }} 
 
\newcommand{\Tcal}   {{\mathcal T }} 
 
\newcommand{\Vcal}   {{\mathcal V }} 
\newcommand{\Wcal}   {{\mathcal W }}

\newcommand{\Zcal}   {{\mathcal Z }}

\title{Macroscopic loops in  the Bose gas,  Spin O(N) and related models }
\author{Alexandra Quitmann\thanks{Weierstrass Institute for Applied Analysis and Stochastics, Berlin, Germany. E-mail: alexandra.quitmann@wias-berlin.de }\and Lorenzo Taggi\thanks{Sapienza Universit\`a di Roma, Dipartimento di Matematica, Roma, Italy. E-mail: lorenzo.taggi@uniroma1.it}}
\date{}   

\numberwithin{equation}{section}

\date{
    \today
}

\begin{document}

\maketitle

\begin{abstract}
We consider a general system of   interacting random loops which includes  several models of interest, such as
 the \textit{Spin O(N) model,} \textit{random lattice permutations},  a version of the
 \textit{interacting Bose gas} in discrete space and of the \textit{loop O(N) model.}  
We consider the system in $\mathbb{Z}^d$,    $d \geq 3$,  and prove  the occurrence of macroscopic loops  whose length is proportional to the volume of the system.
More precisely,  we approximate $\mathbb{Z}^d$ by finite boxes and,  given any two  vertices whose distance is proportional to the diameter of the box,  
we prove that the probability of observing a loop visiting both is uniformly positive.   
Our results hold under general assumptions on the interaction potential,  which may have bounded or unbounded support or introduce hard-core constraints. 
\end{abstract}

\section{Introduction and main results}
We consider a general system of   interacting random loops which includes  several models of interest, such as
 the \textit{Spin O(N) model,} \textit{random lattice permutations},  a version of the
 \textit{interacting Bose gas} in discrete space and of the \textit{loop O(N) model.}  
In our models the loops can be  oriented or unoriented and interact with each other via a potential  which depends on their mutual distance. 
The potential can have bounded or unbounded support and can  allow a bounded  or an arbitrarily large number  of visits  at the vertices.  
We consider the system in $\mathbb{Z}^d$,  with $d>2$, and prove  the occurrence of macroscopic loops, whose length is proportional to the volume of the system. 
In particular,  we approximate $\mathbb{Z}^d$ by finite boxes and,  given any two  vertices whose distance is proportional to the diameter of the box, 
we prove that the probability of observing a loop visiting both is uniformly positive.   
We now discuss some of the models to which our general loop soup reduces or is related to,
after that we provide a formal definition of the model  and state our first main theorem in wide generality.

\paragraph{Spin O(N) model.}
The Spin O(N) model is one of the most important statistical mechanics models. 
In this model the spins take values in a unit sphere of dimension $N-1$,
see Section \ref{sect:spinresults} for the definition.
Special cases of interest are the Ising model (N=1), the XY model (N=2) and the Heisenberg model (N=3).  
The model is interesting for all integer values of $N \geq 1$, in particular its rigorous 
 analysis  is particularly challenging for values of  $N$
 greater than two, in which case important mathematical tools  are missing, 
 for example correlation inequalities.
An important mathematical object in the analysis of such a spin system 
is  the  BFS-random loop model,
which  was introduced and studied by Brydges,  Fr\"ohlich and Spencer in \cite{BFS1, BFS2}
following the work of Symanzik \cite{Symanzik}.
This is a random walk loop soup whose realisations are systems of closed random walk trajectories interacting by local repulsive interactions. 
The general loop soup considered in our paper reduces to the 
BFS-random loop representation for a special choice of the parameters
(as illustrated below).
It is well known from the seminal work of Fr\"ohlich, Simon and Spencer \cite{Frohlich}  that
a phase transition in the Spin O(N) model 
with respect to the variation of an external parameter, the inverse temperature, occurs in dimension $d>2$. 
This phase transition corresponds to the fact that as long as the inverse temperature is a above a certain critical threshold, the  
spin-spin correlations  do not decay to zero with the distance (while they decay exponentially fast with the distance as the inverse temperature is below such threshold, the sharpness of such a phase transition is however known only for $N=1,2$).
When translated into the language of loops through the BFS representation,
the non-decay to zero of the spin-spin correlations implies that the ratio of two partition functions -- with one partition function
corresponding to the weight of configurations with interacting loops and a walk connecting two points, and the other partition corresponding to the weight of configurations with only loops - is uniformly positive with respect to the variation of such two points and to the size of the system.
This fact has no direct consequence for the corresponding  random loop model where only closed   trajectories are present.
In this paper we show that a phase transition occurs also in such a random loop model. 
More precisely,  we prove that,  for any integer $N \geq 2$, for large enough inverse temperatures, the expected length of any loop  is proportional to the volume of the system and,  moreover, 
given any two distant points, the probability of observing a loop connecting both is uniformly positive. 
This also extends a result from \cite{Be-U},
stating that, in the special case   $N=2$,
 a positive fraction of sites are crossed by long loops when
$d \geq 3$ and the inverse temperature is large enough.
Contrary to \cite{Be-U}, our result does not rely on the spin formulation
of the random loop model and is entirely derived from the analysis
of a system of closed random walk trajectories.
Hence, our result holds  for a much larger
class of  random loop models -- which do not necessarily admit a representation as a spin system -- of which Spin O(N) is  just a special case. 
Moreover,  when translated back into the language of spins,  our result about occurrence of macroscopic loops has new implications on the decay of correlations in the Spin O(N) model itself,  see Section \ref{sect:spinresults} for further results.

\paragraph{Interacting Bose gas.}
Providing a rigorous proof of the occurrence of Bose-Einstein condensation  (BEC)
-- a physical phenomenon predicted by Einstein occurring to a certain class of gasses at very low temperatures -- is one of the most important  open problems in rigorous statistical mechanics \cite{LiebBook}.
In 1953  Feynman introduced  what is now referred to as the Feynman-Kac representation \cite{Feynman}.  This representation allows the reformulation of the  Bose gas -- which is defined in the functional analytic framework of rigorous quantum mechanics --  as a probabilistic model of  interacting closed  Brownian trajectories \cite{Ginibre}. 
In this system a phase transition corresponding to the occurrence of macroscopic 
loops as the particle density is above a certain critical threshold and the dimension is greater than two is expected to occur.  This phase transition is considered to be equivalent
to Bose-Einstein condensation  \cite{Suto1, Suto2, UeltschiRelation2}.  Providing a rigorous proof the occurrence of macroscopic loops in such a random loop model is then of great physical interest and an intriguing mathematical problem \textit{per se}.
\begin{figure}
\includegraphics[scale=1.10]{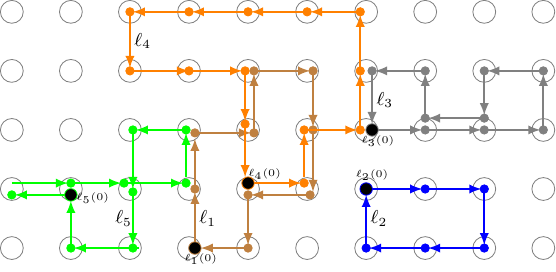}
\centering
\caption{Illustration of a realisation of the random walk loop soup. The larger filled circles represent the starting points of each loop.
In the Bose gas interpretation of the random walk loop soup  each step of each loop hosts a particle and each particle.}
\label{Fig:exampleBEC}
\end{figure}
In recent years increasing effort in the mathematical physics and  probability communities has been made for the solution of this  problem, which nevertheless remains open. 
A first important progress was made in 
\cite{ALSSY, DLS} (see also \cite[Chapter 11]{LiebBook}),
where the occurrence of macroscopic loops 
was proved for the Bose gas under  hard-core local 
interactions under the  so-called \textit{half-filling} condition.
Further important progress has been made in the rigorous analysis of spatial random permutation models  \cite{BU1,Betz0,  BU3, BU4, Bogachev, ElboimPeled}
and in the Bosonic loop soup considered in \cite{D-V}.
In these systems the loops interact through  a potential
 which depends on  the total number of loops of a given length. 
The presence of such  potential makes the model interesting and challenging, 
however the interaction  does not depend on the mutual distance between the loops -- and thus on how they are displaced in space -- and is then
a simplification of the one which is present in the loop representation of  the Bose gas.
A further recent progress has been made in \cite{T},  in this paper the occurrence of a macroscopic (open) loop was proved for the model of lattice permutations,  in which the loops interact at sites via mutual-exclusion.
A further  approach based on large deviations allowed the characterisation of the free energy 
of the Bose gas in $\mathbb{R}^d$  in a certain  region of the phase diagram 
\cite{Adams} and lead to a proof of BEC on the  complete graph \cite{Toth}.

A special case of our general random loop soup corresponds to  the Bose gas  in $\mathbb{Z}^d$ under a minor modification.  The modification consists in the replacement of the continuous time simple random walk connecting two consecutive particles by a single-step simple random walk trajectory (see also Section \ref{sect:extensionsBEC} below for further details and  comments).
Each such step of each loop then  interacts with all the other steps through a potential which depends on their mutual distance.
We consider the system in the grand canonical ensemble,  in which the particle density is controlled by an external parameter, the {chemical potential}.
Our main result is a proof of the occurrence of macroscopic loops as the average particle density is above a certain  (finite) critical threshold. More precisely, we prove that the expected length of any loop is proportional to the volume of the system and that, given any two distant sites, the probability of existence of a loop connecting both is uniformly positive.

\paragraph{Lattice permutations and loop O(N).}
Our general loop soup reduces or is related  to further important statistical mechanics models, for example to the model of lattice permutations and to the loop O(N) model.
Lattice permutations are a special case of our model in which the local time at 
each vertex is allowed to be at most one.
They are relevant for various aspects, for example  they are a generalisation 
of the double dimer model, which attracts interest for many different reasons, see \cite{Dubedat2, Kenyon5, QuitmannTaggi, T}.
The  loop O(N) model is a related model 
in which the local time at each edge (and not at each vertex) is allowed to be at most one
(see \cite{PeledSpinka} for an overview).
Our theorem unfortunately does not cover these models,  since it holds only for models in which the
local time at each vertex is allowed to be a large enough constant. 
However, the qualitative behaviour of our random walk loop soup is not expected to depend on the specific constraint on the local time, hence it is interesting and natural to make a comparison to the above mentioned models.

For lattice permutations,
the only available results about existence of macroscopic loops
involve the special case of fully-packed loops, the so-called double dimer model, 
which corresponds  to the superposition of two independent dimer covers
\cite{Dubedat2, Kenyon5, QuitmannTaggi}.
In the case of non fully-packed loops,  it is known that one single 
 `open loop' forced through the system is macroscopic \cite{QuitmannTaggi}
 when the inverse temperature is large enough in dimensions three and higher. 
For the loop O(N) model,  the occurrence of `macroscopic' loops  has been   proved  on the hexagonal lattice \cite{DuminilCopinParafermionic} along the critical curve
of the phase diagram.
In  this (planar) case the term `macroscopic' refers to the existence of a loop which surrounds a circle of arbitrarily large diameter with uniformly positive probability.
In the higher dimensional case,  instead, 
the notion of `macroscopic' loop
is stronger.
 Indeed, our work shows that  in $\mathbb{Z}^d$ with $d>2$, for 
each random walk loop soup which is covered by our theorem, 
 the expected length of any loop
is proportional to the volume of the box and, moreover, with  uniformly positive probability there exists  a loop connecting any two distant vertices. 
Such a behaviour is not expected to occur in two dimensions.

Let us also stress that random walk loop soups are  intriguing mathematical objects
which  appear  in many other subjects in probability  and mathematical physics,
for example in the framework of quantum spin systems,
and in relation to the theory of scaling limits and SLE curves.
We refer to  \cite{Aizenman3, Aizenman4, UeltschiLoopQuantum}
and \cite{Werner1, Werner2} for some references on the two subjects.

\subsection{Definitions}
Let $\mathbb{T}_L$ be a torus of side length $L$  in $\mathbb{Z}^d$,
whose elements can be identified with the set 
$
\{ x = (x_1, \ldots, x_d) \in \mathbb{Z}^d \,  : \, x_i \in (-\frac{L}{2}, \frac{L}{2}] \mbox{ for each $i=1, \ldots d$ } \}.
$
Let $\mathcal{L}$ be the set of rooted oriented loops, i.e., 
finite ordered sequences of vertices in $\mathbb{T}_L$,
$\ell = \big (\ell(0), \ell(1),   \ldots \ell(k) \big )$,
such that $\ell(i)$ is a nearest-neighbour of $\ell(i-1)$ for each 
$i \in \{1, \ldots, k\}$,   $\ell(k) = \ell(0)$ and $k >1$.
For any such sequence $\ell = \big (\ell(0), \ell(1),   \ldots, \ell(k) \big ) \in \mathcal{L}$,  we denote by $| \ell | : = k$ 
the length of the loop $\ell$. 
We let $\Omega := \cup_{n=0}^ \infty \mathcal{L}^n $ be the configuration space, whose elements are ordered collections of rooted oriented loops.
Given any configuration $\omega \in \Omega$ we denote by $|\omega|$ the \textit{number of loops}, i.e., $|\omega|$ is defined as the  integer $n \in \mathbb{N}_0$ such that $\omega \in \mathcal{L}^n$. 
For any $\omega \in \Omega$, we define by 
$$
n_x(\omega) := \sum\limits_{n=1}^{|\omega|} \sum\limits_{j=0}^{|  \ell_n  | -1} \mathbbm{1}_{
\{ \ell_n(j) = x \} }
$$
the \textit{local time} at $x \in \mathbb{T}_L$.
We now introduce a very general probability measure on $\Omega$,
which depends on several parameters and functions,  after that we will show that
very important models correspond to a special choice of 
such parameters and functions.
We first define a \textit{potential},  $v  :  \mathbb{Z}^d  \rightarrow \mathbb{R}$
satisfying $v(x) = v(-x)$ for each $x \in \mathbb{Z}^d$ and 
 make the interaction periodic under torus translations by introducing the function $v_L : \mathbb{T}_L \times \mathbb{T}_L  \rightarrow \mathbb{R}$, which depends on $v$ and $L$,  and is defined as 
\begin{equation}\label{eq:periodicpotential}
v_L(x,y) :=  \sum_{z \in \mathbb{Z}^d} v(y + L z - x).
\end{equation}
Moreover, we define the  \textit{weight function},
$U : \mathbb{N}_0 \rightarrow \mathbb{R}_0^+$,
which weights the local time at sites and may, for example, 
 suppress configurations with local time above a certain threshold.
We also introduce two parameters $\lambda, N \in \mathbb{R}^+$.
Our measure  assigns to any realisation $\omega \in  \Omega$ with $n \in \mathbb{N}_0$ loops, 
$\omega = \big  ( \ell_1, \ell_2, \ldots \ell_n \big ) \in  \mathcal{L}^n $, 
the weight,
\begin{equation}\label{eq:measureBEC}
\mathcal{P}_{L, U,  v, N, \lambda} (\omega ) : = \frac{1}{\mathcal{Z}_{L,  U, v,  N,  \lambda}}
 \, \, \frac{1}{n!} \, \,   \prod_{ i=1   }^n
\frac{\lambda^{ |\ell_i|}}{|\ell_i|} \, 
 {(\frac{N}{2})}^{n} \prod_{ x \in \mathbb{T}_L} U(n_x(\omega) ) \, 
 \exp \big(-  \mathcal{V}_L( \omega)  \big ),
\end{equation}
 where for any $\omega \in \Omega$,
\begin{equation}\label{eq:definitioninteraction}
\mathcal{V}_L( \omega) : = 
   \sum\limits_{  i = 1 }^{ |\omega|  }
    \sum\limits_{  j = 1 }^{ |\omega|  }
     \sum\limits_{ m=0  }^{ |\ell_i|  -1 }
      \sum\limits_{ n=0  }^{ |\ell_j|  -1  }
      v_L \big (    \ell_i(m),  \ell_j(n)  \big ),
 \end{equation}
and $\mathcal{Z}_{L,  U, v,  N,  \lambda}$ is a normalisation constant,  to which we refer as \textit{partition function}.
We generically refer to the random loop model defined by \eqref{eq:measureBEC} as Random Walk Loop Soup (RWLS).
 According to the previous definition,  any  step of any loop interacts repulsively or attractively (depending on the sign of $v_L(\cdot, \cdot)$) with all the other  steps of all the other loops 
 through the function  $v_L( \cdot , \cdot)$.
 Moreover, the parameter $\lambda \in (0, \infty)$ provides a penalisation or a reward  for the  total length of the loops,   intuitively higher values of $\lambda$    correspond to a greater total loop length.
The parameter $N$ provides a penalisation or a reward  for the total number of loops, intuitively higher values of $N$ correspond to a higher number of loops. 
The weight $1/|\ell_i|$ in \eqref{eq:measureBEC} can be viewed as a normalisation factor for the number of starting points of a given loop, the weight $1/ n!$ can be viewed as a normalisation factor for the number of ways the labels $1$, $2$,  $\ldots$ $n$ can be distributed among the $n$ loops. 
The well-definedness of the measure \eqref{eq:measureBEC} requires some assumptions on $v$ and $U$. 
We say that  the potential $v$ is \textit{tempered} if
\begin{equation}\label{eq:tempered}
 \overline{v} := v(0) + \sum_{x \in \mathbb{Z}^d} v(x) \mathbbm{1}_{   \{ v(x) <0 \}  } \geq 0 \quad \mbox{ and }  \quad   \sum\limits_{x \in \mathbb{Z}^d  }  |v(x) |  < \infty.
\end{equation}
A potential is tempered if it is locally non-attractive ($v(0) \geq 0$),  the  total attractive interactions  (represented by the sum with the indicator in  \eqref{eq:tempered}) are not stronger than such local repulsive interaction,  and if it is summable. 
Such  conditions allow us to prevent that all the loops concentrate in a finite region of the torus with infinite local time.
%We refer to Section \ref{sect:appendixpotentials} for some examples of tempered potentials.
Given the potential $v$, we say that the weight function $U : \mathbb{N}_0 \rightarrow \mathbb{R}_0^+$  is \textit{good} if there exists $M < \infty$ such that, 
\begin{equation}\label{eq:rapidlydecayingdef}
n \, \overline{U}(n+1) \,  \leq  \, M \,  \overline U(n),
\end{equation}
where $\overline{U}(n) = U(n) e^{  -  \overline{v} \,  n^2 }$.
For example,  $U(n) = \mathbbm{1}_{ \{n \leq 10\}  }$ and $U(n) = \frac{1}{n!}$ are good
for any tempered potential $v$,  while  $U(n) = n^2$ is good only if $v$ is not only tempered but also $ \overline{v} > 0$.  It is easy to ensure the well-definedness of \eqref{eq:measureBEC}  for any $ \lambda, N \in \mathbb{R}^+$ if  $v$ is tempered and $U$ is good (see Lemma \ref{lemma: finitepartitionfunction} below).

 \paragraph{Special cases.}
We now discuss the choices of $v$,  $U$,  $\lambda$ and $N$ of our general loop soup which lead to other models of interest.

\textit{{Interacting Bose gas.}}
When $N=2$,  and $U(n) = 1$ for any $n \in \mathbb{N}_0$,  our loop soup is a version of the discrete Bose gas in the grand-canonical ensemble with chemical potential $\log \lambda$   and unit inverse temperature,
the only difference with the Bose gas in continuous space is that  the particles are located in $\mathbb{Z}^d$ rather than in $\mathbb{R}^d$ and that 
  a single-step random walk trajectory rather than a Brownian bridge of time $\beta$ (the inverse temperature) connects two consecutive particles. 
  We refer to  Section \ref{sect:extensionsBEC} below for further details on such an important connection.

%When $v=0$, the measure  (\ref{eq:measureBEC})
%corresponds to the so-called random walk loop soup,
%which has been studied in view of the conformal invariant properties
%of its two-dimensional continuous scaling limit,
%the Brownian walk loop soup,  see 
%for example \cite{LawlerWerner} and references therein.
%In this case  however the partition function introduced in  (\ref{eq:measureBEC}) is infinite %and the model can be defined as a  a Poissonian realization from a $\sigma$-finite
%measure on unrooted loop.

\textit{{Spin O(N) model and BFS representation.}}
When $N \in \mathbb{N}$,  $v(x) = 0$ for any $x \in \mathbb{Z}^d$,
and 
\begin{equation}\label{eq:weightfunctionspin}
U(n) =  \frac{\Gamma(\frac{N}{2})}{ 2^n \Gamma(\frac{N}{2} + n)   },
\end{equation}
our model corresponds to the Brydges, Fr\"ohlich and Spencer 
representation of the Spin O(N) model with inverse temperature $\lambda   \geq 0$ \cite{BFS1}
(note a correction of the original definition in \cite[eq. (6.18)]{BFS2}), see also  \cite{Be-U,  BFS2, L-T-first, L-T-exponential} for further connections between the Spin O(N) model and random loops and for a definition of the spin O(N) model.

  \textit{{Lattice permutations, loop O(N) model and other models}.}
 When  $v(x)= 0$ for any $x \in \mathbb{Z}^d$
 and 
 \begin{equation}\label{eq:weightfunctionR}
 U(n) = \begin{cases}
 1  \quad & \mbox{ if $n \leq R$,}\\
0 \quad & \mbox{ otherwise, }
 \end{cases}
 \end{equation}
 our model is such that the local time at each vertex is upper bounded by $R  \in \mathbb{N}$. When $R= 1$,  our model reduces to random lattice permutations \cite{T},
 which, in turn, reduce to the double dimer model 
 \cite{Kenyon5} when $\lambda = \infty$ \cite{T}.
 The double dimer model  in turn corresponds to the superposition of two independent configurations of the dimer model.
 As we explained above, the  loop O(N) model \cite{PeledSpinka} is not a special case of our general model, but it is closely related to.  
 %To obtain the loop O(N) model one would need to impose the restriction that the 
 %total number of loops crossing each edge in one of the two directions is bounded by one.
To see the connections between our model (where the loops are oriented,  labelled, and have a starting point) and these models  (where the loops are unoriented,   receive no label and have no starting point), one should observe that 
 the terms $\frac{1}{n!}$ and $\frac{1}{|\ell_i|}$ in   \eqref{eq:measureBEC}
are normalisation  factors for the number of possibilities of assignment of $n$ labels to the $n$ loops  and of shifting the starting point of each loop respectively.

\subsection{Main result about the occurrence of macroscopic loops}
 Our main theorem states that,  if $\lambda$ is large enough, 
 the expected length of any loop is of the order of the volume of the torus. 
In particular,  the probability of existence of a loop connecting any pair of sites  having distance proportional to the diameter of the box is uniformly positive.
Such a general result requires a further  assumption on the potential, to which we refer as \textit{separability}. 
This assumption is introduced in Definition \ref{def: separable} below, here we provide 
two main examples of potentials fulfilling such an assumption,  i.e.,
$$v_1(x) =   \alpha  \mathbbm{1}_{\{  x = 0  \}} -  \beta \,  e^{ - \iota | x |_1   }  \mathbbm{1}_{\{  x \neq 0  \}},$$ 
and
$$v_2(x) =   \alpha   \mathbbm{1}_{\{  x = 0  \}} -  \beta  \, |x|_1^{-s} \mathbbm{1}_{\{  x \neq 0  \}},$$
for parameters $\alpha,  \beta,  \iota,  s  \in [0, \infty)$ such that  
$s > d$, where $| \cdot  |_1$ is the $\ell_1$  distance  on the torus.
These two examples correspond to a local repulsive interaction and to a long-range attractive interaction which decays exponentially or polynomially with the distance.  
Such potentials are tempered if $\alpha$ is large enough with respect to $\beta$. 

We now state our first main theorem. 
Let $\Gamma_x$ be the first loop visiting the vertex $x$, namely 
for any $\omega = (\ell_1, \ldots, \ell_{|\omega|}) \in \Omega$, we let 
$w_x(\omega)  := \inf\{ i \in \{1, \ldots, | \omega |  \} \, : \,  x \in \ell_{i}   \}$ be the smallest index of the loops visiting $x$, where $|\omega|$ is the total number of loops in $\omega$. 
We then define for any $\omega = ( \ell_1, \ldots, \ell_{ |\omega| }) \in \Omega$,
\begin{equation} \label{eq:firstloopx}
\Gamma_x (\omega) := \begin{cases}
\ell_{ w_x(\omega)   } & \mbox{ if  $w_x(\omega) < \infty$  } \\
\emptyset & \mbox{ otherwise.} 
\end{cases} 
\end{equation}
Moreover, for any $i \in \mathbb{N}$ we let $\Gamma^{(i)}$ be the ith loop of $\omega = ( \ell_1, \ldots, \ell_{|\omega|}) \in \Omega$,   
\begin{equation}\label{eq:nthloop}
\Gamma^{(i)}(\omega) :=
 \begin{cases}
\ell_{ i  } & \mbox{ if  $i \leq |\omega|$  } \\
\emptyset & \mbox{ otherwise.} 
\end{cases}
\end{equation}
Furthermore, we let $E_{L, U, v, N, \lambda}$ be the expectation with respect to (\ref{eq:measureBEC}).
Finally, we say that the weight function $U : \mathbb{N}_0 \rightarrow \mathbb{R}_0^+$ has \textit{range}  $R$ if 
\begin{equation}\label{eq:definitionrange}
R := \sup \{ n \in \mathbb{N}_0 \, : \, U(n) > 0   \}.
\end{equation}
For example, the range of the weight function is infinite in the case of the interacting Bose gas, $U(n) = 1$ for every $n \in \mathbb{N}_0$, 
and of the Spin O(N) model, (\ref{eq:weightfunctionspin}),
while the range is finite if the weight   function satisfies (\ref{eq:weightfunctionR}).
We also let $\mathbb{T}_L^o$ be the set of sites $x = (x_1, \ldots, x_d) \in \mathbb{T}_L$ such that $x_i \in 2 \mathbb{N}_0+1$ for every coordinate $i \in [d]$ and denote the origin by $o \in \mathbb{T}_L$.
 \begin{theorem}\label{theo:maintheorem}
 Let $d $, $N  \in \mathbb{N}$, be such that $d \geq 3$ and $N \geq 2$,  let $R$ be a large enough integer depending on $d$ and $N$, 
 suppose that $v :  \mathbb{Z}^d \rightarrow \mathbb{R}$ is tempered and separable,
let $U$ be a good weight function with range at least $R$.  
  There exists $\lambda_0 < \infty$  such that,  for any $\lambda > \lambda_0$,  
 the following two properties hold: 
 \begin{enumerate}[(i)]
 \item There exists $c_1\in (0, \infty)$, which does not depend on $L$, such that,
    \begin{equation}\label{eq:twopointBEC1}
\liminf_{ \substack{  L \rightarrow \infty : \\ L \in 2 \mathbb{N} }} \frac{E_{L,   U, v, N, \lambda}  (  |\Gamma_x|  )}{L^d}  > c_1,
 \end{equation}
 for any vertex $x \in \mathbb{T}_L$.
\item There exist $c_2, c_3 \in (0, 1)$, which do not depend on $L$,  such that, for any $L \in 2 \mathbb{N}$,  any $x \in \mathbb{T}_L^o$  such that $d_L(o,x)  \leq c_2 \, L $,
 \begin{equation}\label{eq:twopointBEC2}
\mathcal{P}_{L,   U, v, N, \lambda} \big ( \exists n \in \{1,  \ldots, |\omega| \}  \, : \, o,x  \in \Gamma^{(n)}  \big   )  > c_3,
 \end{equation}
 where $d_L(x,y)$ is the torus graph distance. 
 \end{enumerate} 
 \end{theorem}
On the contrary, if $\lambda$ is sufficiently small 
the model exhibits a quite different behaviour, indeed
  $E_{L, U, v, N, \lambda}  (  |\Gamma_x| )  = O(1)$ in the limit as $L \rightarrow \infty$
and 
the quantity in the left-hand side of 
 (\ref{eq:twopointBEC2}) decays exponentially  with the distance between $x$ and $y$ (with exponential moments uniformly bounded in $L$).
This can be proved  using the cluster expansion  method
(see for example \cite[Chapter 5]{FriedliVelenik}) or  the methods of \cite{Betz2, TaggiShifted}.
   Hence,  the combination of these facts and of our theorem imply the occurrence of a phase transition 
   with respect to the variation of the parameter $\lambda$.
   Under further assumptions on the weight function $U$ the point-wise positivity result, (\ref{eq:twopointBEC2}), can be extended to any vertex  $x \in \mathbb{T}_L$.
   
   Moreover, our result complements a result from \cite{Chayes},
   where a random loop model analogous to ours was considered,  in which any edge is allowed to be crossed by at most one loop and no long-range interactions are present.
   In this paper it was proved that, if $d \geq 2$ and $N$ is a large enough integer, then the loops are `small' for any $\lambda \in [0, \infty)$,
   hence no phase transition with respect to $\lambda$ occurs.
    On the contrary, our result states that, if $d \geq 3$ and 
   $N$ is any integer greater than two,  then a phase transition with respect to $\lambda$ \textit{does} occur as long as the range of the weight function is large enough (depending on $d$ and $N$).

\paragraph{Paper organisation and proof  structure.}
We end this introduction by presenting the organisation of the paper.
In Section \ref{sect:randompathmodel} we introduce the  Random Path Model (RPM),
a random loop model which differs from the RWLS,  \eqref{eq:measureBEC},
and plays an important role in our proofs. 
In Section \ref{sect:equivalence} we prove that these random loop models are \textit{equivalent} if one considers
  averages of functions which do not depend on certain features
 of the configuration, like for example the loop orientation
 or the starting point of the loop. The loop length is an example of such functions.
 From this point of the paper until the last page we work with the RPM
 and exploit the nice properties of such a model in all the proofs, 
 for example reflection positivity, colours,
 the conditional pairing independence.  
In Section \ref{sect:reflectionpositivityandchessboardestimate}
we introduce the first important technique for the analysis
of the random path model, reflection positivity and the chessboard estimate.
In Section \ref{sect:two-pointfunctionanditsproperties}
we relate the two point function to some important observable of
random loop model -- for example we relate the two-point function
defined at two neighbour vertices to the expected number of links
of a certain colours crossing the corresponding edge -- 
and use some probabilistic and combinatorial tools
to provide bounds on the two-point function
of the random path model.
In Section \ref{sect:derivationofthekeyinequality}
we derive the so-called Key Inequality.
In Section \ref{sect:upperboundonthefouriersum}
we use elementary Fourier analysis 
notions for the derivation of the so-called Infrared-bound
from the Key Inequality. 
In Section \ref{sect:proofoftheorem1.1} we 
present the proof of our main theorem. 
Here we use  the Infrared-bound
and our estimates on the two-point function to prove the occurrence of macroscopic loops
in the RPM.
By our equivalence theorem this result 
is then extended to the RWLS, thus concluding the proof of our main theorem,
Theorem \ref{theo:maintheorem}.

\section*{Notation}

\begin{center}
	\begin{tabular}{ l l }
	
	$\R^+$, $ \R^+_0$  & strictly positive  and non-negative real numbers  respectively \\

$\N$,  $\N_0$ &  strictly positive and non-negative integers respectively \\

$[n]$ & set of integers $\{1,2,\dots,n\}$ \\

%	   $ \boldsymbol{e}_i$  & cartesian vector, with $i \in \{1, \ldots d\}$ or   $i \in \{1, \ldots d+1\}$ \\ 

$\mathcal{G} = ( \mathcal{V}, \mathcal{E})$ &  an undirected, simple, finite  graph \\

$e \in \mathcal{E}$ or $\{x,y\} \in \mathcal{E}$ & undirected edges \\

$(x,y) \in \mathcal{E}$ &  edge directed from $x$ to $y$ \\

% $N \in \mathbb{N}_{>0}, \, \, \lambda \in \R^+$& respectively number of colours, and  edge-parameter \\

%$U$ &  {weight function} \\

$m = (m_e)_{e \in \mathcal{E}}$ & link cardinalities, with  $m_e$ corresponding  number of links \textit{on} the edge $e$  \\

$c = (c_e)_{e \in \mathcal{E}}$ & link colourings, with  $c_e : \{1, \ldots, m_e\} \rightarrow \{1, \ldots, N\} $ \\ 

$\pi = (\pi_x)_{x \in \mathcal{V}}$ & pairings, with $\pi_x$ pairing the links touching the vertex $x$ or leaving them unpaired \\

$g=(g_x)_{x \in \Vcal}$ & configuration of ghost pairings \\

%$\mathcal{C}_{\mathcal{G}}(m)$ & the set of colourings for  $m\in\mathcal{M}_{\mathcal{G}}$ 

%$\mathcal{P}_{\mathcal{G}}(m, c)$ & the set of pairing configurations for $m\in\mathcal{M}_{\mathcal{G}}$ and $c\in\mathcal{C}_{\mathcal{G}}(m)$ \\

$\mathcal{W}_{\mathcal{G}}$ &  the set of  configurations of the RPM in $\mathcal{G}$, with $w= (m,c,\pi,g) \in \mathcal{W}_\Gcal$ \\

$\tilde{\Wcal}_\Gcal$ & set of configurations of the RPM with no unpaired links and no ghost pairings \\

$n_x$ & number of pairings at $x$ \\

$u_x$ & number of unpaired endpoints of links at $x$ \\

$\Omega_\Gcal$ & set of configurations of the RWLS in $\mathcal{G}$ \\

$\Lcal_\Gcal$ & set of rooted oriented loops in $\Gcal$ \\

$\Sigma(\Wcal_\Gcal)$,  $\Sigma(\Omega_\Gcal)$, $ \Sigma(\Lcal_\Gcal)$ & set of equivalence classes  \\

 $\delta(l)$ & stretch-factor of $l \in \Lcal_\Gcal$ \\
 
 $J(l)$ & multiplicity of $l \in \Lcal_\Gcal$ \\

$\mathbb{Z}^\ell_{L, U, v,  N, \lambda}$ 
 & partition function of the random path model\\
 
 $\mathbb{Z}_{L, U, v,  N, \lambda}(x,y)$ 
 & directed partition function \\
 
 $\mathbb{G}_{L, U, v,  N, \lambda}(x,y)$ 
 & two-point function \\

%          \end{tabular}

%\end{center}

%\begin{center}
%	\begin{tabular}{ l l }

  $\mathbb{G}_{L, U, v,  N, \lambda}(x)$ 
 & equivalent to $\mathbb{G}_{L, U, v,  N, \lambda}(o,x)$   \\
 
  $\hat{\mathbb{G}}_{L, U, v,  N, \lambda}(k)$ 
 & Fourier transform of  $\mathbb{G}_{L, U, v,  N, \lambda}(x)$  \\
 
 $\Zcal_{L, U, v,  N, \lambda}$  & partition function of the random walk loop soup \\

 $ ( \T_L, \E_L)$ & graph corresponding to the torus $\mathbb{Z}^d / L \mathbb{Z}^d$   \\
 
  $ ( \mathcal{T}_L, \mathcal{E}_L)$ & extended torus, with original and virtual vertices \\
  
      $ \T_L^{ (2)} \subset  \mathcal{T}_L$ & set of virtual vertices  \\
      
         $ \T_L^{*}$  & Fourier dual torus   \\

 $o \in \mathbb{T}_L$,  $o \in \mathcal{T}_L$, $o \in \T_L^*$ & origin \\

    $x \sim y$ & pair of vertices in $\T_L$ which are connected by an edge in $ \E_L$    \\

%   $\mathcal{Z}_{L, N, \rho, U}(\boldsymbol{h})$ 
% & central quantity   \\

$\boldsymbol{v} = (v_x)_{x \in \T_L}$ &  real-valued vector, with  coordinates associated to $\T_L$ \\

$\boldsymbol{h} = (h_x)_{x \in \mathcal{T}_L}$ & vector of real numbers, with  coordinates associated to  $\mathcal{T}_L$ \\

  $\mathcal{Z}_{L, U, v,  N, \lambda}(\boldsymbol{h})$ 
 & partition function with colour changes at $x$ receiving a multiplicative weight $h_x$    \\
 
  $\mathcal{Z}^{(2)}_{L, U, v,  N, \lambda}(\boldsymbol{h})$ 
 & second term of the polynomial expansion   \\

%$\{ A, B, C \} = \{ A, C, B \}$ & unordered tuples (in this %case a 3-tuple) \\
 
%$(A, B, C) \neq (A, C, B)$ & ordered tuples \\
\end{tabular}

\end{center}

Throughout the paper, we denote any positive constant that depends on the model parameters $L, U, v, N, \lambda$ by $c$. The constants $c$ may differ from line to line. 
Any further dependence will be denoted explicitly.

\section{Random path model}
\label{sect:randompathmodel}
In this section we define the Random Path Model (RPM).
Let $\Gcal=(\Vcal,\Ecal)$ be an undirected, simple, finite graph, and let $N \in \N$. We refer to $N$ as the \textit{number of colours}. A realisation of the random path model can be viewed as a collection of %unrooted 
unoriented paths which might be closed or open.
To define a realisation we need to introduce {links}, {colourings} and %(ghost) 
pairings. A glance at Figure \ref{fig:RPM} might be helpful. 
We represent a \textit{link configuration} by
 $m \in  \mathcal{M}_{\mathcal{G}} := \{ m \in \mathbb{N}_0^{\mathcal{E}}: \forall x \in \Vcal, \, \sum_{y \sim x} m_{\{x,y\}} \in 2\N_0 \}$. More specifically
$$m = \big ( m_e \big )_{e \in \mathcal{E}},$$
where $m_{e}\in\N_0$ represents the number of links on the edge $e$.
Intuitively, a link represents a `visit' at the edge from a path.
The links are ordered and receive a label between $1$ and $m_e$. We denote by $(e,p)$ the $p$-th link at $e \in \Ecal$ with $p \in [m_e]$. If a link is on the edge $e = \{x,y\}$, then we  say that \textit{it touches} $x$ and $y$. 

Given a link configuration $m \in \mathcal{M}_{\mathcal{G}}$, a \textit{colouring} $c = (c_e)_{e \in \mathcal{E}}$, with $c_e: [m_e] \to [N]$ is a function which assigns an integer in $[N]$ to each link, which will be called its \textit{colour}.
More precisely, $c_e(p) \in [N]$ is the colour of the $p$-th link on the edge $e \in \mathcal{E}$, with $p \in [m_e]$. A link with colour $i \in [N]$ is called an $i$-link. We let $\Ccal_\Gcal(m)$ be the set of possible colourings $c=(c_e)_{e \in \Ecal}$ for $m$.

Given a link configuration $m \in \mathcal{M}_{\mathcal{G}}$, and a colouring $c \in \mathcal{C}_{\mathcal{G}}(m)$, a pairing $ \pi = (\pi_x)_{x \in \mathcal{V}}$ for $(m,c)$ pairs links touching $x$ in such a way that, if two links are paired, then they have the same colour.  A link touching $x$ can be paired to at most one other link %of same colour 
 touching $x$, and it is not
necessarily the case that all links touching $x$ are paired to another link at $x$. %See Figure \ref{fig:RPM} for an example. 
More formally, $\pi_x$ is a partition of the links touching $x$ into sets of at most two links. If a link touching $x$ is paired at $x$ to no other link touching $x$, then we say that the link is \textit{unpaired at $x$}. Given two links, if there exists a vertex $x$ such that such links are \textit{paired at $x$}, then we say that such links are \textit{paired}.  We remark that, by definition, a link cannot be paired to itself. 
We denote by $\mathcal{P}_{\mathcal{G}}(m, c)$ the set of all such pairings for $m\in\mathcal{M}_\Gcal$, $c \in \mathcal{C}_{\mathcal{G}}(m)$. 
%Note that $ \mathcal{P}_{\mathcal{G}}(m, c)$ generally has many elements, corresponding to the number of ways of pairing links.

We can view a triplet $(m,c,\pi)$ with $m \in \Mcal_\Gcal, c \in \Ccal_\Gcal(m)$ and $\pi \in \Pcal_\Gcal(m,c)$ as a collection of (open or closed) paths which are unrooted and unoriented, see also Figure \ref{fig:RPM}.

Additionally, we also introduce \textit{ghost pairings}. A configuration of ghost pairings is an element $g=(g_x)_{x \in \Vcal} \in \Hcal_\Gcal :=\{0,1,2\}^\Vcal$.
We introduce ghost pairings since 
at some point we will  need to `replace' some removed pairings by  `ghost' pairings  in order to preserve the weight  of the configuration
(later we will introduce a measure which weights the configurations depending
on the sum of the number of pairings and ghost pairings at the vertices),
and they allow the derivation of some useful monotonicity properties, as we explain in  Remark \ref{rem:ghosts} below.

A \textit{configuration}  of the random path model is an element $w = ( m, c, \pi, g)$
such that $m \in \mathcal{M}_{\mathcal{G}}$,
$c \in \mathcal{C}_{\mathcal{G}}(m)$,  
$\pi \in \mathcal{P}_{\mathcal{G}}(m, c)$ and $g \in \Hcal_\Gcal$.
We let $\mathcal{W}_{\mathcal{G}}$ be the set of  such configurations. 

With slight abuse of notation, we will also view, $m,c,\pi,g: \Wcal_\Gcal \to \N_0$ as functions such that for any $w^\prime=(m^\prime,c^\prime, \pi^\prime,g^\prime) \in \Wcal_\Gcal$, $m(w^\prime)=m^\prime$, $c(w^\prime)=c^\prime$, $\pi(w^\prime)=\pi^\prime$ and $g(w^\prime)=g^\prime$. 
For any $x \in \mathcal{V}$ and $i \in [N]$, let $u_x^i: \mathcal{W}_{\mathcal{G}} \rightarrow \mathbb{N}_0$ be the function corresponding to the number of $i$-links touching $x$ which are not paired to any other link touching $x$ and let $u_x: \mathcal{W}_{\mathcal{G}} \rightarrow \mathbb{N}_0$ defined by 
$$
u_x(w):=\sum_{i=1}^N u_x^i(w)
$$ 
be the total number of unpaired links touching $x$.
Let further $n_x^i : \mathcal{W}_{\mathcal{G}} \rightarrow \mathbb{N}_0$ be the function corresponding to the number of pairings of $i$-links at $x$. We then define $n_x: \mathcal{W}_{\mathcal{G}} \rightarrow \mathbb{N}_0$ by 
$$
n_x(w) := \sum_{i=1}^N n_x^i(w)+g_x(w).
$$
as the total number of pairings at $x$, which are ghost or not ghost. Let $m_e^{(i)}: \Wcal_\Gcal \to \N_0$ denote the number of $i$-links on $e$ for any $i \in [N]$ and $e \in \Ecal$, namely for any $w \in \Wcal_\Gcal$, 
$
m_e^{(i)}(w):= \sum_{j=1}^{m_e} \mathbbm{1}_{\{c_e(j)=i\}}(w).
$
%Note that if $g_x(w)=0=u_x(w)$, then $n_x(w)= \frac{1}{2} \sum_{y \sim x} m_{\{x,y\}}(w)$, i.e., the number of links touching $x$ divided by two.

\begin{figure}
  \centering
    \includegraphics[width=0.3\textwidth]{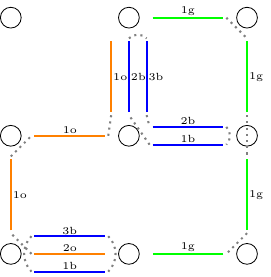}
      \caption{ Representation of a triplet $(m,c,\pi)$ with $m \in \Mcal_\Gcal, c \in \Ccal_\Gcal(m)$ and $\pi \in \Pcal_\Gcal(m,c)$, where $\Gcal$ corresponds to the graph $\{1,2,3\} \times \{1,2,3\}$ with edges connecting nearest neighbours. The vertices are represented by black circles and the edges are not drawn in the figure. The lowest leftmost vertex corresponds to $(1,1)$. On every edge $e$ the links are ordered and receive a label between $1$ and $m_e$.  We assume that $N=3$ and rather than by using numbers we represent the colours using the letters $b,o,g$.  Paired links are connected by a dotted gray line. For example, at $(1,1)$, all links are paired and the pairing is given by $\pi_{(1,1)}=\{\{(e_1,1),(e_1,3)\}, \{(e_1,2),(e_2,1)\}\}$, where $e_1$ and $e_2$ are defined as the edges connecting $(1,1)$ and $(2,1)$ and $(1,1)$ and $(1,2)$, respectively. At $(2,1)$, there exist precisely two unpaired links and the pairing is given by $\pi_{(2,1)}=\{\{(e_1,1),(e_1,3)\}, \{(e_1,2)\},\{(e_3,1)\}\}$, where $e_3$ is the edge connecting $(2,1)$ and $(3,1)$.}
\label{fig:RPM}
\end{figure}

We let 
\begin{equation} \label{eq:randompathtilde}
\tilde{\Wcal}_\Gcal:=\{w \in \Wcal_{\Gcal}: u_x=0=g_x \, \forall x \in \Vcal\}
\end{equation}
be the set of configurations in which there exist no unpaired links and no ghost pairings. Elements of $\tilde{\Wcal}_{\Gcal}$ are simply denoted by triples $w=(m,c,\pi)$.

We now introduce a (non-normalised) measure  on $\mathcal{W}_{\mathcal{G}}$ and a probability measure on $\tilde{\Wcal}_{\Gcal}$.
\begin{definition}\label{def:measure}
Let $N \in \mathbb{N}$, $\lambda \in \R^+$, $U:\N_0 \to \R_0^+$ and $v_\Gcal: \mathcal{V} \times \mathcal{V} \to \mathbb{R}$ be given. We refer to $U$ as \textit{weight function} and to $v_\Gcal$ as \textit{potential}. We define $V:\mathcal{W}_{\mathcal{G}} \to \mathbb{R}$ by
\begin{equation} \label{eq:defV}
\forall w \in \mathcal{W}_{\mathcal{G}}
\quad \quad V(w) := \sum_{x,y \in \mathcal{V}} v_\Gcal(x,y)n_x(w)n_y(w)
\end{equation}
and introduce the (non-normalised) non-negative measure $\mu_{\Gcal,U,v_\Gcal,N,\lambda}$ on $\mathcal{W}_{\mathcal{G}}$,
\begin{equation}\label{eq:RPMmeasure}
\forall w = (m, c,  \pi,g ) \in \mathcal{W}_{\mathcal{G}}
\quad \quad 
\mu_{\Gcal,U,v_\Gcal,N,\lambda}(w) := 
  \bigg(\prod_{e \in \mathcal{E}} \frac{ \lambda^{m_e(w)}}{m_e(w)!} \bigg ) \bigg( \prod_{x \in \mathcal{V}} U(n_x(w))
  \bigg) e^{-V(w)}.
\end{equation}

%Note that $U$ does not assign a weight to unpaired links at $x$ for all $x \in \Vcal$. 

Given a function $f : \mathcal{W}_{\mathcal{G}} \rightarrow \mathbb{C}$, we represent its average by 
$
\mu_{\Gcal,U,v_\Gcal,N,\lambda} \big ( f \big ) :=
\sum\limits_{w \in \mathcal{W}_\Gcal}\mu_{\Gcal,U,v_\Gcal,N,\lambda} (w) f(w).
$

We define the measure $\mu^\ell_{\Gcal,U,v_\Gcal,N,\lambda}$ as the restriction of the measure $\mu_{\Gcal,U,v_\Gcal,N,\lambda}$ to the set of configurations $\tilde{\Wcal}_{\Gcal}$ and define a probability measure on $\tilde{\Wcal}_{\Gcal}$ by
\begin{equation}
\label{eq:RPMprobabilitymeasure}
\forall w =(m,c,\pi) \in \tilde{\Wcal}_{\Gcal}
\quad \quad 
\P_{\Gcal,U,v_\Gcal,N,\lambda}(w) := \frac{\mu^\ell_{\Gcal,U,v_\Gcal,N,\lambda}(w)}{\Z_{\Gcal,U,v_\Gcal,N,\lambda}^\ell},
\end{equation}
where $\Z_{\Gcal,U,v_\Gcal,N,\lambda}^\ell := \mu_{\Gcal,U,v_\Gcal,N,\lambda}(\tilde{\Wcal}_{\Gcal})$  is the partition function.
We denote by $\E_{\Gcal,U,v_\Gcal,N,\lambda}$ the expectation under the measure $\P_{\Gcal,U,v_\Gcal,N,\lambda}$. Sometimes, for a lighter notation, we will omit the sub-scripts. 
\end{definition}

The next lemma provides a sufficient condition for the well-definedness of the measure \eqref{eq:RPMprobabilitymeasure}. 
\begin{lemma} \label{lemma: finitepartitionfunction}
Let $N \in \mathbb{N}$, $\lambda \in \R^+$ and suppose that $v_\Gcal: \Vcal \times \Vcal \to \R$ satisfies $v_\Gcal(x,y)=v_\Gcal(y,x)$ for any $x,y \in \Vcal$ and that
\begin{equation} \label{eq:conditionpotential}
\forall x \in \Vcal \quad \quad \bar{v}_\Gcal(x):=v_\Gcal(x,x) + \sum_{y \in \Vcal} v_\Gcal(x,y) \mathbbm{1}_{\{v_\Gcal(x,y) <0\}} \geq 0.
\end{equation}
Suppose that for $U: \N_0 \to \R_0^+$, there exists $M < \infty$ such that
\begin{equation} \label{eq:rapidlydecayinggeneral}
\forall x \in \Vcal \quad \quad U(n) \, (2n-1)!! \, e^{-\bar{v}_\Gcal(x)n^2} \leq M^n.
\end{equation}
Then, for any $a \in \R_0^+$, it holds that  
\begin{equation} \label{eq:finite}
\mu^\ell\bigg( \prod_{x \in \Vcal} e^{a n_x}\bigg) \leq e^{\lambda e^a M N |\Ecal|}.
\end{equation}
In particular, for the choice $a=0$, \eqref{eq:finite} provides an explicit upper bound on the partition function $\Z^\ell$.
\end{lemma}
Condition \eqref{eq:conditionpotential} implies that the local repulsive interactions prevail on the possibly attractive long-range interactions.
\begin{proof} Let $a \in \R_0^+$. 
To begin, we apply the Cauchy-Schwarz inequality and obtain that for any $w \in \tilde{\Wcal}_{\Gcal}$, 
\begin{equation} \label{eq:upperboundpotential}
\begin{aligned}
-V(w) & \leq -\sum_{x,y \in \Vcal}n_x n_y v_\Gcal(x,y) \mathbbm{1}_{\{v_\Gcal(x,y) <0\}} - \sum_{x \in \Vcal} v_\Gcal(x,x)n_x^2 \\
& \leq  \bigg(-\sum_{x,y \in \Vcal} n_x^2 v_\Gcal(x,y)\mathbbm{1}_{\{v_\Gcal(x,y) <0\}} \bigg)^{\frac{1}{2}}\bigg(-\sum_{x,y \in \Vcal} n_y^2 v_\Gcal(x,y)\mathbbm{1}_{\{v_\Gcal(x,y) <0\}} \bigg)^{\frac{1}{2}} - \sum_{x \in \Vcal} v_\Gcal(x,x)n_x^2\\
%& \leq -\sum_{x \in \Vcal} n_x^2 \bigg(v_\Gcal(x,x)+\sum_{y \in \Vcal} v_\Gcal(x,y) \mathbbm{1}_{\{v_\Gcal(x,y)<0\}}\bigg).
& \leq -\sum_{x \in \Vcal} n_x^2 \, \bar{v}_\Gcal(x).
\end{aligned}
\end{equation}
By assumption \eqref{eq:conditionpotential}, $\bar{v}_\Gcal(x) \geq 0$ for any $x \in \Vcal$. 
In the next calculation, we neglect the constraint that the number of links %of each colour 
touching any vertex is even and obtain that
\begin{equation} \label{eq:finitestep1}
\begin{aligned}
\mu^\ell\bigg( \prod_{x \in \Vcal} e^{a n_x} \bigg) %& = \sum_{w=(m,c,\pi) \in \tilde{\Wcal}_{\Gcal}} \prod_{e \in \Ecal} \frac{\lambda^{m_e}}{m_e!} \prod_{x \in \Vcal} U(n_x(w)) e^{-V(w)} 
& \leq \sum_{m=(m_e)_{e \in \Ecal} \in \N_0^{\Ecal}} \prod_{e \in \Ecal} \frac{(\lambda N)^{m_e}}{m_e!} \prod_{x \in \Vcal} e^{a n_x} \, U\Big(\frac{1}{2} \sum_{y \sim x} m_{\{x,y\}}\Big) \Big(\sum_{y \sim x} m_{\{x,y\}}-1\Big)!! \, e^{-\bar{v}_\Gcal(x) \, \big(\frac{1}{2} \sum_{y \sim x} m_{\{x,y\}}\big)^2} \\
& \leq \sum_{m=(m_e)_{e \in \Ecal} \in \N_0^{\Ecal}} \prod_{e \in \Ecal} \frac{(MN \lambda e^a)^{m_e}}{m_e!} 
= \prod_{e \in \Ecal} \sum_{m_e \in \N_0^{\Ecal}} \frac{(MN\lambda e^a)^{m_e}}{m_e!} 
= e^{\lambda e^a M N |\Ecal|},
\end{aligned}
\end{equation}
where $M$ appears in \eqref{eq:rapidlydecayinggeneral}. Here, in the first step, we used the fact $|\Pcal_\Gcal(m,c)| \leq \prod_{x \in \Vcal} \big(\sum_{y \sim x} m_{\{x,y\}}-1\big)!!$ for any $m \in \Mcal_\Gcal$ and any $c \in \Ccal_\Gcal(m)$. In the second step, we used condition \eqref{eq:rapidlydecayinggeneral} and the fact that $\sum_{x \in \Vcal} n_x = \sum_{e \in \Ecal} m_e$. This concludes the proof.
\end{proof}

%\begin{remark}
%Note that, if $v: \Z^d \to \R$ is tempered and our graph is the torus $\T_L$, then the function $v_L:\T_L \times \T_L \to \R$, which was defined in the introduction, satisfies \eqref{eq:conditionpotential}.
%The reason is that 
%$$
%v_L(x,x)=\sum_{z \in \Z^d} v(Lz) \geq v(0)+\sum_{x \in \Z^d} v(x) \mathbbm{1}_{\{v(x) <0\}}  >0
%$$
%and thus for any $x \in  \T_L$,
%$$
%\begin{aligned}
%v_L(x,x)+ \sum_{y \in \T_L} v_L(x,y) \mathbbm{1}_{\{v_L(x,y) <0\}} 
%& \geq v_L(x,x)+ \sum_{y \neq x} \sum_{z \in \Z^d} v(x-y+Lz) \mathbbm{1}_{\{v(x-y+Lz) <0\}} \\
%& \geq v(0) + \sum_{x \in \Z^d} v(x) \mathbbm{1}_{\{v(x) <0\}}  >0.
%\end{aligned}
%$$
%\end{remark}

From now on, we will always assume that the potential $v_\Gcal:\Vcal \times \Vcal \to \R$ satisfies \eqref{eq:conditionpotential} and that the weight function $U:\N_0 \to \R_0^+$ satisfies \eqref{eq:rapidlydecayinggeneral}. 

\section{Equivalence}
\label{sect:equivalence}
The goal of this section is to state an equivalence property between the random path model, which was defined in the previous section, and the random walk loop soup, which was defined in the introduction.  Most of the paper -- except for the proof of Theorem \ref{theo:maintheorem} -- can be read independently from this section.
To state the main result of this section, Theorem \ref{theo: equivalence} below, we first introduce an equivalence relation on the set of configurations in both models. Throughout the section, we fix an arbitrary undirected, simple, finite graph $\Gcal=(\Vcal, \Ecal)$, parameters $N \in \N$, $\lambda \in \R^+$, and functions $U: \N_0 \to \R_0^+$ and $v_\Gcal: \Vcal \times \Vcal \to \R$.

\paragraph{Rooted oriented loops.} Let $\Lcal_\Gcal$ be the set of rooted oriented loops (in short: r-o-loops) in $\Gcal$, i.e., finite ordered sequences of vertices in $\Vcal$, $\ell=\big(\ell(0),\dots,\ell(k)\big)$ such that $\ell(i)$ is a nearest-neighbour of $\ell(i-1)$ for each $i \in [k]$, $\ell(k)=\ell(0)$ and $k \in 2\N$. For any such  $\ell=\big(\ell(0),\dots,\ell(k)\big) \in \Lcal_\Gcal$, we denote by $|\ell|:=k$ the length of $\ell$. 
We call two r-o-loops \textit{equivalent} if one sequence can be obtained as a time-reversion and/or a cyclic permutation of the other sequence. More precisely, for any $\ell=\big(\ell(0),\dots,\ell(k)\big) \in \Lcal_\Gcal$ with $k \in 2\N$, we denote by $c_m(\ell):=\big(\ell(m),\ell(m+1),\dots,\ell(k),\ell(1),\dots,\ell(m)\big)$ the r-o-loop that is obtained from $\ell$ through a cyclic permutation of length $m \in \{0,\dots,|\ell|-1\}$ and by $r(\ell):=\big(\ell(k),\ell(k-1),\dots,\ell(0)\big)$ we denote the \textit{time-reversal} of $\ell$. We then say that $\ell,\ell^\prime \in \Lcal_\Gcal$ are equivalent if $|\ell|=|\ell^\prime|$ and if there exists $m \in \{0,\dots,|\ell|-1\}$ such that $\ell^\prime =c_m(\ell)$ or $\ell^\prime=c_m(r(\ell))$. We denote the equivalence class of $\ell\in \Lcal_\Gcal$ by $\gamma(\ell)$ and by $\Sigma(\Lcal_\Gcal)$ we denote the set of equivalence classes of r-o-loops.

Given two r-o-loops $\ell,\ell^\prime \in \Lcal_\Gcal$ such that $\ell(0)=\ell^\prime(0)$, we define their concatenation  as
$\ell \oplus \ell^\prime:=\big(\ell(0),\dots,\ell(|\ell|),$ $\ell^\prime(1),\dots, \ell^\prime(|\ell^\prime|)\big).$ 
We define the \textit{multiplicity}  of $\ell$, $J(\ell)$, as the maximal integer $n \in \N$ such that $\ell$ can be written as the $n$-fold concatenation of some r-o-loop, $\tilde{\ell}$,  with itself.  Such a  loop $\tilde{\ell} = \tilde{\ell}(\ell)$ has multiplicity one and will be referred to as the \textit{elementary loop} of $\ell$. 
We call an r-o-loop $\ell$ \textit{stretched} if there exists a cyclic permutation of $\ell$ that is identical to $r(\ell)$, i.e., if there exists $m \in \{0,\dots,|\ell|-1\}$ such that $c_m(\ell)=r(\ell)$. Otherwise the r-o-loop is called \textit{unstretched}, see Figure \ref{fig:rootedorientedloops}. %It is important to note that $l$ is stretched if and only if $l^\prime(l)$ is stretched.
For any $\ell \in \Lcal_\Gcal$, we define the \textit{stretch-factor} $\delta(\ell)$ by
$$
\delta(\ell):=\begin{cases} 1  & \text{if } \ell \text{ is stretched,} \\
2 & \text{if } \ell \text{ is unstretched.} \end{cases}
$$
Note that, for any pair of equivalent r-o-loops $\ell,\ell^\prime \in \Lcal_\Gcal$, it holds that $J(\ell)=J(\ell^\prime)$ and $\delta(\ell)=\delta(\ell^\prime)$.  Thus, by slight abuse of notation, we also use the notations $\delta(\gamma)$ and  $J(\gamma)$ for the equivalence classes $\gamma \in \Sigma(\Lcal_\Gcal)$. Further, we denote by $\alpha(\gamma)$ the length of any r-o-loop in $\gamma$ and by $|\gamma|$ we denote the cardinality of $\gamma$. 

\begin{figure}
\centering
\begin{subfigure}{.25\textwidth}
  \centering
  \includegraphics[width=\textwidth]{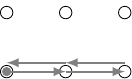}
  \caption{}
  %\label{fig:right}
\end{subfigure}
\hspace{3em}
\begin{subfigure}{.25\textwidth}
  \centering
  \includegraphics[width=\textwidth]{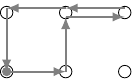}
  \caption{}
  %\label{fig:left}
\end{subfigure}
\hspace{3em}
\begin{subfigure}{.25\textwidth}
  \centering
  \includegraphics[width=\textwidth]{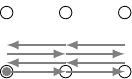}
  \caption{}
  %\label{fig:left}
\end{subfigure}
\caption{Three rooted oriented loops. The roots are denoted by filled dots. (a) The loop has multiplicity one and is stretched. (b) The loop has multiplicity one and is unstretched. (c) The loop is stretched and has multiplicity two. It is the $2$-fold concatenation of the loop in (a).} 
    \label{fig:rootedorientedloops}
\end{figure}

%For many rooted oriented loops $l \in \Lcal_\Gcal$, the cardinality of its equivalence class $\gamma(l)$ is given by $2\alpha(l)$ since each step of the loop can be a possible root in either of the two directions. However, a high multiplicity and being stretched reduce the cardinality of $\gamma(l)$ as it is shown in the following lemma.

The next lemma provides an exact formula for the cardinality of any equivalence class $\gamma \in \Sigma(\Lcal_\Gcal)$. This is an auxiliary lemma for the equivalence relation that is treated later in this section.
\begin{lemma}
\label{lemma: rootedorientedloopcardinality} 
For any $\gamma \in \Sigma(\Lcal_\Gcal)$, we have that
\begin{equation} \label{eq:rootedorientedloopcardinality}
\big|\gamma \big| = \frac{\delta(\gamma) \, \alpha(\gamma)}{J(\gamma)}.
\end{equation}
\end{lemma}

%\begin{example}
%Consider a loop $l \in \Lcal$ that is unstretched (i.e. no matter which vertex of the sequence $l$ you choose as root, the sequence of vertices visited following the loop in the two directions is different) and that satisfies $J(l)=1$ (i.e. each part of the loop is distinguishable), then
%\begin{equation}
%|\Lcal^{\text{i,c,r}}(l)|=2\,|l|\prod_{e \in \E_L} m_e!.
%\end{equation}
%This will typically be the case for very long loops in a large box.
%\end{example}

\begin{proof}
To begin, we fix an arbitrary r-o-loop $\ell \in \gamma$ and write $\gamma=\gamma(\ell)=A \cup B$,
where 
$$
A:= \big\{c_m(\ell): \, m \in \{0,\dots,|\ell|-1\} \big\} \text{ and } B:= \big\{c_m(r(\ell)): \, m \in \{0,\dots,|\ell|-1\} \big\}.
$$
We consider the set $A$ and we define $k:=|\ell|/ J(\ell)$. By definition of the multiplicity, $\ell$ is the $J(\ell)$-fold concatenation of its elementary loop $\tilde{\ell}=\tilde{\ell}(\ell)$ of length $k$ and of multiplicity one with itself. This implies that for any $m>k$, $c_m(\ell)=c_{m-k}(\ell)$ and that for any $m, m^\prime \in \{0,\dots,k-1\}$ such that $m \neq m^\prime$, $c_m(\ell) \neq c_{m^\prime}(\ell)$. It follows that $|A|=k$. 
Since length and multiplicity of $\ell$ and $r(\ell)$ each are identical, the same considerations as above apply to the set $B$ implying that $|B|=|A|=k$.

We will now see that the sets $A$ and $B$ are either identical or disjoint depending on whether $l$ is stretched or unstretched. We first consider the case that $\ell$ is stretched. This means that there exists $\tilde{m} \in \{0,\dots,|\ell|-1\}$ such that $c_{\tilde{m}}(\ell)=r(\ell)$ and thus $r(\ell) \in A$. In particular, any cyclic permutation of $r(\ell)$ is a cyclic permutation of $\ell$ since for any $m \in \{0,\dots,|\ell|-1\}$, we have that $c_m(r(\ell))=c_m(c_{\tilde{m}}(\ell))=c_{(m + \tilde{m}) \bmod{|\ell|}}(\ell)$, implying that $A=B$. We thus obtain \eqref{eq:rootedorientedloopcardinality} with $\delta=1$ for $\ell$ stretched.

Suppose now that $\ell$ is unstretched and assume that $A \cap B \neq \emptyset$. Then, there exist $m, m^\prime \in \{0,\dots,|\ell|-1\}$ such that $c_m(\ell)=c_{m^\prime}(r(\ell))$ implying that $c_{(m-m^\prime) \bmod{|\ell|}}(\ell)=r(\ell)$. This contradicts the fact that $\ell$ is unstretched and thus it must hold that $A \cap B = \emptyset$. We obtain \eqref{eq:rootedorientedloopcardinality} with $\delta=2$ for $\ell$ unstretched.
This concludes the proof.
\end{proof}

\subsection{Equivalence classes in the random walk loop soup}
In this section we introduce an equivalence relation on the set of configurations of the random walk loop soup that was defined in the introduction. We first generalize the definition of the model on a general undirected, simple, finite graph $\Gcal=(\Vcal,\Ecal)$ and then define the equivalence relation.  We denote by  $\Lcal_\Gcal$  the set of rooted oriented loops in $\Gcal$. We let $\Omega_\Gcal := \cup_{n=0}^ \infty \mathcal{L}_\Gcal^n $ be the configuration space. For a potential $v_\Gcal: \Vcal \times \Vcal \to \R$, a weight function $U:\N_0 \to \R_0^+$ and parameters $\lambda, N \in \R^+$, we define the non-normalized measure
\begin{equation} \label{eq:loopsoupgeneralmeasure}
\nu_{\Gcal, U,v_\Gcal, N, \lambda}(\omega) :=\frac{1}{n!} \, \,   \prod_{ i=1   }^n
\frac{\lambda^{ |\ell_i|}}{|\ell_i|} \, 
 {(\frac{N}{2})}^{n} \prod_{ x \in \Vcal} U(n_x(\omega) ) \, 
 \exp \big(-  \mathcal{V}( \omega)  \big ),
\end{equation}
where 
\begin{equation} \label{eq:definitioninteractiongeneral}
\mathcal{V}( \omega) : = 
   \sum\limits_{  i, j  = 1  }^{n}
     \sum\limits_{ m=0  }^{ |\ell_i|  -1 }
      \sum\limits_{ n=0  }^{ |\ell_j|  -1  }
      v_\Gcal \big (    \ell_i(m),  \ell_j(n)  \big ),
\end{equation}
for any $\omega \in  \Omega_\Gcal$ such that 
$\omega = \big  ( \ell_1, \ell_2, \ldots \ell_n \big ) \in  \mathcal{L}_\Gcal^n $ for some $n \in \mathbb{N}_0$. We refer to the constant $\Zcal_{\Gcal, U,v_\Gcal, N, \lambda}:=\nu_{\Gcal, U,v_\Gcal, N, \lambda}(\Omega_\Gcal)$ as partition function. We define a probability measure on $\Omega_\Gcal$ by 
\begin{equation} \label{eq:generalloopsoupprob}
\forall \omega \in \Omega_\Gcal \quad \quad \Pcal_{\Gcal, U,v_\Gcal, N, \lambda}(\omega):=\frac{\nu_{\Gcal, U,v_\Gcal, N, \lambda}(\omega)}{\Zcal_{\Gcal, U,v_\Gcal, N, \lambda}}.
\end{equation}
We always assume that the potential $v_\Gcal$ satisfies \eqref{eq:conditionpotential} and that the weight function $U$ satisfies \eqref{eq:rapidlydecayinggeneral} such that the measure in \eqref{eq:loopsoupgeneralmeasure} has finite mass. We denote by $E_{\Gcal, U,v_\Gcal, N, \lambda}$ the expectation with respect to the measure \eqref{eq:generalloopsoupprob}. Sometimes, for a lighter notation, we will omit the sub-scripts. 

We now introduce an equivalence relation on $\Omega_\Gcal$. We call two configurations in $\Omega_\Gcal$ equivalent if one configuration can be obtained from the other one by changing the orientation and/or the root of the r-o-loops and/or by changing the labels of the r-o-loops. This is formalized in the next definition. We denote by $S_n$ the group of permutations of $n$ elements for $n \in \N$. 
\begin{definition} 
\label{def:equivalenceBEC}
We call two configurations $\omega, \omega^\prime \in \Omega_\Gcal$ \textit{equivalent} if there exists $n \in \N$ such that $\omega=(\ell_1,\dots,\ell_n)$ and $\omega^\prime=(\ell_1^\prime,\dots,\ell_n^\prime)$ and if there exists a permutation $\pi \in S_n$ such that $\ell_{\pi(i)} \in \gamma(\ell_i^\prime)$ for all $i \in [n]$. 
We denote by $\rho(\omega)$ the equivalence class of $\omega \in \Omega_\Gcal$ and by $\Sigma(\Omega_\Gcal)$ we denote the set of equivalence classes of $\Omega_\Gcal$.
\end{definition}
Note that for any two equivalent configurations $\omega,\omega^\prime \in \Omega_\Gcal$, it holds that $\nu(\omega)=\nu(\omega^\prime)$.

\begin{definition} \label{def:rolindependentfct}
We call a function $f:\Omega_\Gcal \to \R$ that is constant on each equivalence class \textit{root-orientation-label-independent (in short: r-o-l-independent)}. With slight abuse of notation, we then let $f(\rho)$ be the evaluation of $f$ at any configuration in $\rho \in \Sigma(\Omega_\Gcal)$.
\end{definition}

For any $\gamma \in \Sigma(\Lcal_\Gcal)$, we introduce the function $k_\gamma^1: \Omega_\Gcal \to \N_0$, which is defined as follows: For any $\omega \in \Omega_\Gcal$ such that $\omega=(\ell_1,\dots,\ell_n) \in \Lcal_\Gcal^n$ for some $n \in \N_0$, we set
$
k_\gamma^1(\omega)= |\{i \in [n]: \ell_i \in \gamma\}|.
$
Hence, $k_\gamma^1$ counts the number of r-o-loops in a configuration $\omega \in \Omega_\Gcal$ that are an element of the equivalence class $\gamma$. 
For any $\omega \in \Omega_\Gcal$, we define the sets $\Xi_1(\omega):=\{(\gamma,k_\gamma^1(\omega)): \gamma \in \Sigma(\Lcal_\Gcal)\}$ 
and 
\begin{equation}\label{eq:defxi1}
\Xi_1^*(\omega):=\{(\gamma,k_\gamma^1(\omega)) \in \Xi_1(\omega): \, k_\gamma^1(\omega)>0\}.
\end{equation}
Note that $\Xi_1^*(\omega)$ is always finite. The next lemma follows immediately from our definitions. 

\begin{lemma} \label{lemma:equivalenceRWLS}
Two configurations $\omega, \omega^\prime \in \Omega_\Gcal$ are equivalent if and only if $\Xi_1^*(\omega)=\Xi_1^*(\omega^\prime)$.
\end{lemma}
In particular, it follows that the function $\Xi_1^*$ is r-o-l-independent. The next proposition provides a formula for the cardinality of any equivalence class $\rho \in \Sigma(\Omega_\Gcal)$. Recall from the beginning of this section that $\delta(\gamma)$, $J(\gamma)$ and $\alpha(\gamma)$ denote the stretch-factor, the multiplicity and the length of $\gamma \in \Sigma(\Lcal_\Gcal)$.
\begin{proposition}
Suppose that $\omega \in \Omega_\Gcal$ is such that $\omega \in \Lcal_\Gcal^k$ for some $k \in \N$ and suppose that $\Xi_1^*(\omega)=\{(\gamma_1,k_1),\dots,(\gamma_p,k_p)\}$ for some $p \in [k]$, $\gamma_1,\dots,\gamma_p \in \Sigma(\Lcal_\Gcal)$ and $k_1,\dots,k_p \in \N$. Then,
\begin{equation} \label{eq: cardinalityequivalenceBEC}
\big| \rho(\omega) \big| = \binom{k}{k_1,\dots,k_p} \prod_{j=1}^p \bigg(\frac{\delta(\gamma_j) \,\alpha(\gamma_j)}{J(\gamma_j)}\bigg)^{k_j}.
\end{equation}
\end{proposition}
\begin{proof} Let $p \in \N$, $\gamma_1,\dots,\gamma_p \in \Sigma(\Lcal_\Gcal)$,  $k_1,\dots,k_p \in \N$ and $\omega \in \Omega_\Gcal$ be fixed such that $\Xi_1^*(\omega)=\{(\gamma_1,k_1),\dots,(\gamma_p,k_p)\}$. Using Lemma \ref{lemma:equivalenceRWLS}, we have that
$$
\begin{aligned}
\big|\rho(\omega)\big|  &= \big| \{\omega^\prime \in \Omega_\Gcal: \Xi_1^*(\omega^\prime)=\Xi_1^*(\omega)\}\big| \\
& = \big|\{(\ell_1^\prime,\dots,\ell_k^\prime) \in \Lcal_\Gcal^k: \# \{i \in [k]: \ell_i^\prime \in \gamma_j\}=k_j \, \forall j \in [p] \}\big| \\
& = \big|\{(\gamma_1^\prime,\dots,\gamma_k^\prime) \in \Sigma(\Lcal_\Gcal)^k: \# \{i \in [k]: \gamma_i^\prime=\gamma_j\}=k_j \, \forall j \in [p] \}\big| \, \prod_{j=1}^p |\gamma_j|^{k_j} \\
%& = \bigg(\prod_{j=1}^p |\gamma_j|^{k_j}\bigg) \binom{k}{k_1,\dots,k_p} \bigg(\prod_{j=1}^p |\gamma_j|^{k_j}\bigg) \\
& = \binom{k}{k_1,\dots,k_p} \prod_{j=1}^p \bigg(\frac{\delta(\gamma_j) \, \alpha(\gamma_j)}{J(\gamma_j)}\bigg)^{k_j},
\end{aligned}
$$
where we used \eqref{eq:rootedorientedloopcardinality} in the last step. 
This concludes the proof.
\end{proof}

\subsection{Equivalence classes in the random path model} \label{subsec:ERPM}
We now introduce an equivalence relation on the set $\tilde{\Wcal}_{\Gcal}$, which is defined in \eqref{eq:randompathtilde}. 
Roughly speaking, we call two configurations in $\tilde{\Wcal}_{\Gcal}$ equivalent if one configuration can be obtained from the other one by changing the positions of the links on the edges while keeping their colours and pairings fixed, see Figure \ref{fig:loopreduction} for an example. 

Given $m \in \Mcal_\Gcal$, we call a sequence of permutations $\kappa=(\kappa_e)_{e \in \Ecal}$ a \textit{permutation for} $m$ if $\kappa_e \in S_{m_e}$ for each $e \in \Ecal$, where $S_n$ is the group of permutations of $n$ integers.

Consider now $m \in \Mcal_\Gcal$, $c \in \Ccal_\Gcal(m)$ and $\pi=(\pi_x)_{x \in \Vcal} \in \Pcal_\Gcal(m,c)$ such that $(m,c,\pi) \in \tilde{\Wcal}_\Gcal$ and let $\kappa=(\kappa_e)_{e \in \Ecal}$ be a permutation for $m$. For $\pi_x=\big\{\{(e_1,p_{e_1}),(e_2,p_{e_2})\},\dots, \{(e_{l-1},p_{e_{l-1}}), (e_l,p_{e_l})\}\big\}$, we define 
$$
\pi_x^\kappa:= \big\{\{(e_1,\kappa_{e_1}(p_{e_1})),(e_2,\kappa_{e_2}(p_{e_2}))\},\dots, \{(e_{l-1},\kappa_{e_{l-1}}(p_{e_{l-1}})), (e_l,\kappa_{e_l}(p_{e_l}))\}\big\}.
$$
Furthermore,  for any $w = (m,c,\pi) \in \tilde{\Wcal}_{\Gcal}$ and any permutation $\kappa$ for $m$ we let $w^{\kappa} = (m^\prime,c^\prime, \pi^\prime)$ 
be the configuration such that $m=m^\prime$, 
$c_e^\prime(p)=c_e(\kappa_e^{-1}(p))$ for all $e \in \Ecal$,
and $\pi_x^\prime=\pi_x^\kappa$ for all $x \in \Vcal$.
See also Figure \ref{fig:loopreduction}.
\begin{figure}
\centering
\begin{subfigure}{.45\textwidth}
  \centering
  \includegraphics[width=\textwidth]{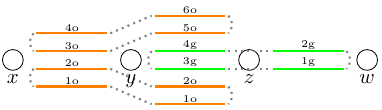}
\end{subfigure}
\hspace{2em}
\begin{subfigure}{.45\textwidth}
  \centering
  \includegraphics[width=\textwidth]{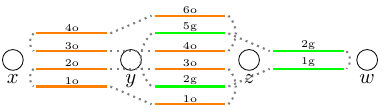}
\end{subfigure}
\caption{Two equivalent configurations $w, w^\prime \in \tilde{\Wcal}_\Gcal$, where $\Gcal$ corresponds to the graph $\{x,y,z,w\}$ with edges connecting nearest neighbour vertices. The number of colours is $N=2$ ($1:=$ g, $2:=$ o). The configuration $w^\prime$ on the right can be obtained from the configuration $w$ on the left through the permutation $\kappa$ for $m$ defined  by $\kappa_{\{x,y\}}(n)=n$ for $n \in [4]$, $\kappa_{\{y,z\}}(1)=1$, $\kappa_{\{y,z\}}(2)=3$, $\kappa_{\{y,z\}}(3)=2$, $\kappa_{\{y,z\}}(4)=5$, $\kappa_{\{y,z\}}(5)=4$, $\kappa_{\{y,z\}}(6)=6$ and $\kappa_{\{z,w\}}(n)=n$ for $n \in [2]$. %We have that $\Xi_1^*(w)=\Xi_1^*(w^\prime)=\big\{(\gamma((x,y,z,y,x)),2,2), (\gamma((y,z,w,z,y)),1,1)\big\}$.
} 
    \label{fig:loopreduction}
\end{figure}
\begin{definition}
\label{def:equivalenceRPM}
We call two configurations $w=(m,c,\pi), w^\prime=(m^\prime,c^\prime, \pi^\prime) \in \tilde{\Wcal}_{\Gcal}$ \textit{equivalent} if there exists a permutation $\kappa=(\kappa_e)_{e \in \Ecal}$ for $m$, such that $w^{\prime} = w^{\kappa}$.
We denote by $\sigma(w)$ the equivalence class of $w \in \tilde{\Wcal}_\Gcal$ and we denote by $\Sigma(\tilde{\Wcal}_\Gcal)$ the set of equivalence classes of $\tilde{\Wcal}_\Gcal$.
\end{definition}
%Note that the permutation $\kappa$ does not have to be unique. 
Note that for any two equivalent configurations $w,w^\prime \in \tilde{\Wcal}_{\Gcal}$, it holds that $\mu(w)=\mu(w^\prime)$. 
With slight abuse of notation, if a function $f: \tilde{\Wcal}_\Gcal \to \R$ is constant on each equivalence class,  
then we let $f(\sigma)$ be the evaluation of $f$ at any configuration in $\sigma \in \Sigma(\tilde{\Wcal}_\Gcal)$.

Our goal is now to compute the cardinality of an equivalence class $\sigma \in \Sigma(\tilde{\Wcal}_\Gcal)$ and our result is presented in Proposition \ref{prop:RPM} below.
 
Given a configuration $w=(m,c,\pi) \in \tilde{\Wcal}_\Gcal$ and a permutation for $m$, $\kappa$, the configuration $w^{\kappa}$ is by definition equivalent to $w$. Clearly,  there exist $\prod_{e \in \mathcal{E}} m_e!$ distinct permutations for $m$. However
distinct permutations for $m$,  $\kappa$ and $\kappa^\prime$, may lead to the same configuration,  namely it may be the case that $w^{\kappa} = w^{\kappa^\prime}$. Hence,  in order to determine the number of configurations which are equivalent to $w$, 
we need to control the overcounting.
For this, we now introduce the notion of \textit{cycles}.

\paragraph{Cycles.}
Given $m \in \Mcal_\Gcal$, a \textit{rooted oriented linked loop} (in short: \textit{r-o-l-loop}) for $m$ is an ordered sequence of nearest-neighbour vertices and pairwise distinct links,
$$L=(x_0,(\{x_0,x_1\},p_1),x_1,(\{x_1,x_2\},p_2),\dots,(\{x_{k-1},x_{k}\},p_{k}),x_{k}\big),
$$ where $p_j \in \{1,\dots,m_{\{x_{j-1},x_j\}}\}$ for each $j \in [k]$ and such that $x_0=x_{k}$ for some $k \in 2\N$. The vertex $x_0$ is called the \textit{root} of $L$, the link $(\{x_0,x_1\},p_1)$ is called the \textit{starting link} of $L$ 
and the length of $L$ is defined by $|L|:=k$. 
We let $\Rcal_m$ be the set of all r-o-l-loops for $m$. Two r-o-l-loops are said to be \textit{equivalent} if one sequence can be obtained as a time-inversion and/or a cyclic permutation of the other sequence. More precisely, we denote by $c_n(L):=(x_n,(\{x_n,x_{n+1}\},p_{n+1}), \dots, x_k, (\{x_0,x_1\},p_1),x_1, \dots,(\{x_{n-1},x_{n}\},p_{n}),x_n)$ the r-o-l-loop that is obtained from $L$ through a cyclic permutation of length $n \in \{0,\dots,|L|-1\}$ and by $r(L):= (x_k,(\{x_{k-1},x_{k}\},p_{k}),x_{k-1},(\{x_{k-1},x_{k-2}\},p_{k-1}),\dots,(\{x_{0},x_{1}\},p_{1}),x_{0}\big)$ we denote the time-reversal of $L$. We then say that $L,L^\prime \in \Rcal_m$ are equivalent if $|L|=|L^\prime|$ and if there exists $n \in \{0,\dots,|L|-1\}$ such that $L^\prime=c_n(L)$ or $L^\prime=c_n(r(L))$. We denote by $\chi(L) \subset \Rcal_m$ the equivalence class of $L \in \Rcal_m$ and we denote by $\Sigma(\Rcal_m)$ the set of %all 
equivalence classes of r-o-l-loops. The equivalence class of an r-o-l-loop will simply be referred to as \textit{cycle}, which can be thought of as an unoriented loop with no root. We say that a link $(e,p)$ with $p \in [m_e]$ and $e \in \Ecal$ or a vertex $x \in \Vcal$ is  \textit{contained} in $\chi \in \Sigma(\Rcal_m)$ if $(e,p)$ or $x$ occurs in the sequence $L$ for every $L \in \chi$. We then write $(e,p)\in \chi$ and $x \in \chi$.
A set of cycles for $m$, $\{\chi_1,\dots,\chi_k\} \subset \Sigma(\Rcal_m)$ with $k \in \N$, is called an \textit{ensemble of cycles} for $m$ if every link is contained in precisely one cycle of the set, i.e. if for every $(e,p)$ with $p \in [m_e], e \in \Ecal$, there exists precisely one $i \in [k]$ such that $(e,p) \in \chi_i$. We denote by $\Ecal_m$ the set of ensembles of cycles for $m$.

For any $w=(m,c,\pi) \in \tilde{\Wcal}_\Gcal$, we can uniquely construct an ensemble of cycles 
\begin{equation} \label{eq:ensemblecycles}
\zeta(w):=\{\chi_1(w),\dots,\chi_{k(w)}(w)\} \in \Ecal_m
\end{equation}
as follows: We take any link $(\{x,y\},p_{\{x,y\}})$ and choose a point $z_0 \in \{x,y\}$. Step-by-step, we construct a r-o-l loop $L_1$ for $m$ with root $z_0$ and with starting link $(\{x,y\},p_{\{x,y\}})$ by choosing $z_1 \in \{x,y\} \setminus \{z_0\}$ as the next vertex and by choosing the link which is paired at $z_1$ to $(\{x,y\},p_{\{x,y\}})$ as the next link in the sequence. We continue until we are back at the link $(\{x,y\},p_{\{x,y\}})$. We define the cycle $\chi_1(w) \in \Sigma(\Rcal_m)$ as the equivalence class of $L_1$. For the next cycle, we choose a link that has not been selected yet and proceed as before. We continue until all links have been selected. We call a cycle $\chi \in \zeta(w)$ an $i$-\textit{cycle} with $i \in [N]$ if every link that is contained in $\chi$ has colour $i$. %We define $k$ as the smallest integer such that any link is explored by $\Gamma_1,\dots, \Gamma_k$. 

%Vice versa, given a loop ensemble for $(m,c)$ we can uniquely construct a pairing configuration $\pi \in \tilde{\Pcal}_\Gcal(m,c)$ by pairing any two successive links of a loop, i.e. each loop ensemble for $(m,c)$ uniquely defines a configuration $w=(m,c,\pi) \in \tilde{\Wcal}_\Gcal$. In particular, we have a bijection between the two sets $\{\pi \in \tilde{\Pcal}_\Gcal(m,c)\}$ and $\Ecal(m,c)$.

%\begin{remark}
%\label{rem:equiavelenceRPM1}
%From Definition \ref{def:equivalenceRPM} it follows that two configurations $w=(m,c,\pi), w^\prime=(m,c^\prime,\pi^\prime) \in \Wcal_{\Gcal}$ are equivalent if and only if there exists a permutation $\kappa$ for $m$ such that $\Lcal(w)= \Lcal(w^\prime)^\kappa$.
%\end{remark}

Recall that $\Lcal_\Gcal$ denotes the set of r-o-loops, as defined at the beginning of this section. For $m \in \Mcal_\Gcal$, we introduce the map 
$$\vartheta: \Rcal_m \to \Lcal_\Gcal,$$ which acts by projecting an r-o-l-loop $L=(x_0,(\{x_0,x_1\},p_1),x_1,\dots,(\{x_{k-1},x_{k}\},p_{k}),x_{k}\big) \in \Rcal_m$ onto the corresponding r-o-loop $\vartheta(L):=(x_0,x_1,\dots,x_k) \in \Lcal_\Gcal$ by \lq forgetting\rq{} about the links in the sequence. It is important to note that for any two equivalent r-o-l-loops $L,L^\prime \in \Rcal_m$ also the r-o-loops $\vartheta(L), \vartheta(L^\prime) \in \Lcal_\Gcal$ are equivalent. By slight abuse of notation, we thus also use the function $\vartheta$ to map an equivalence class $\chi \in \Sigma(\Rcal_m)$ to its corresponding equivalence class $\vartheta(\chi) \in \Sigma(\Lcal_\Gcal)$.

For any $\gamma \in \Sigma(\Lcal_\Gcal)$ and any $i \in [N]$, we introduce the function $k_{\gamma,i}^2: \tilde{\Wcal}_\Gcal \to \N_0$, where $k_{\gamma,i}^2(w)$ is defined as the number of $i$-cycles $\chi \in \zeta(w)$ that satisfy  $\vartheta(\chi) = \gamma$ for any $w \in \tilde{\Wcal}_\Gcal$. We further define $k_\gamma^2(w):= \sum_{i=1}^N k_{\gamma,i}^2(w)$. 
For any $w \in \tilde{\Wcal}_\Gcal$ we define the set $\Xi_2(w):= \big\{(\gamma,k_{\gamma}^2(w)) \, : \, \gamma \in \Sigma(\Lcal_\Gcal)\big\}$. We further define 
\begin{equation}\label{eq:defxi2}
\Xi_2^*(w):= \big\{(\gamma,k_{\gamma}^2(w)) \in \Xi_2(w) \, : \, k_{\gamma}^2(w) >0 \big\}.
\end{equation}
Note that $\Xi_2^*(w)$ is always finite and that $\Xi_2^*$ is many-to-one.
It follows immediately from our definitions that the following lemma holds true.

%In the next lemma we use the following notation: Given $m \in \Mcal_\Gcal$, an r-o-l-loop $L \in \Rcal_m$ and a permutation $\kappa$ for $m$, we let $L^\kappa \in \Rcal_m$ be the r-o-l-loop that is obtained from $L$ by replacing any link $(e,p)$ that occurs in the sequence $L$ by the link $(e,\kappa_e(p))$. Further, for a cycle $\chi \in \Sigma(\Rcal_m)$, we set $\chi^\kappa:=\{L^\kappa: \, L \in \chi\}$.

\begin{lemma} \label{lemma:equivalenceequivalenceRPM}
Two configurations $w,w^\prime \in \tilde{\Wcal}_\Gcal$ are equivalent if and only if $\Xi_2^*(w)=\Xi_2^*(w^\prime)$ and $k_{\gamma,i}^2(w)=k_{\gamma,i}^2(w^\prime)$ for any $\gamma \in \Sigma(\Lcal_\Gcal)$ and $i \in [N]$.
\end{lemma}

The next lemma is a preparatory lemma for Proposition \ref{prop:RPM} below. 

\begin{lemma} \label{lemma:bijection}
 Suppose that $w=(m,c,\pi) \in \tilde{\Wcal}_\Gcal$ is such that $\Xi_2^*(w)=\big\{(\gamma_1,k_1),\dots, (\gamma_p,k_p)\big\}$ for some $p \in \N$, $\gamma_1,\dots, \gamma_p \in \Sigma(\Lcal_\Gcal)$ and $k_1,\dots,k_p \in \N$ and suppose that $k_{\gamma_j,i}^2(w)=m_{j,i}$ for some $m_{j,i} \in \N_0$ with $i \in [N]$ such that $\sum_{i=1}^Nm_{j,i}=k_j$ for any $j \in [p]$. 
It holds that
\begin{equation} \label{eq:preparationRPM}
|\sigma(w)| = \prod_{j=1}^p \binom{k_j}{m_{j,1},\dots,m_{j,N}} \Big|\big\{\{\chi_1,\dots\chi_k\} \in \Ecal_m: \, |\{l \in [k]: \, \vartheta(\chi_l)=\gamma_j\}|=k_j \, \forall j \in [k]\big\}\Big|,
\end{equation}
where $k:= \sum_{j=1}^p k_j$. 
\end{lemma}
\begin{proof}
Suppose that $\Xi_2^*(w)=\big\{(\gamma_1,k_1),\dots, (\gamma_p,k_p)\big\}$ and $k_{\gamma_j,i}^2(w)=m_{j,i}$ for all $i \in [N]$, $j \in [p]$, as in the statement of the lemma. 
We set 
$$
\tilde{\Ecal}_m=\tilde{\Ecal}_m(\gamma_1,\dots,\gamma_p) := \big\{\{\chi_1,\dots\chi_k\} \in \Ecal_m: \, |\{l \in [k]: \, \vartheta(\chi_l)=\gamma_j\}|=k_j \, \forall j \in [k]\big\} \subset \Ecal_m.
$$ 
We denote by $\tilde{\zeta}: \sigma(w) \to \tilde{\Ecal}_m$ the restriction of the map $\zeta$, which was defined in \eqref{eq:ensemblecycles}, onto $\sigma(w)$, i.e.,  $\tilde{\zeta}(w^\prime):=\zeta(w^\prime) \in \tilde{\Ecal}_m$ for all $w^\prime \in \sigma(w)$. By Lemma \ref{lemma:equivalenceequivalenceRPM} the map $\tilde{\zeta}$ is well-defined and, for any $\{\chi_1,\dots\chi_k\} \in \tilde{\Ecal}_m$, $\tilde{\zeta}^{-1}(\{\chi_1,\dots\chi_k\})$ corresponds to the set of configurations $w^\prime=(m, c^\prime, \pi^\prime) \in \sigma(w)$, where $c^\prime \in \Ccal_\Gcal(m)$ is a colouring function that has the following two properties: 
\textbf{(1):} $c^\prime$ assigns the same colour to any two links that are contained in the same cycle;
\textbf{(2):} $ |\{l \in [k]: \, \vartheta(\chi_l)=\gamma_j \text{ and } c^\prime_e(p)=i \, \forall \, (e,p) \in \chi_l\}|=m_{j,i}$ for any $j \in [p]$ and $i \in [N]$. The pairing configuration $\pi^\prime \in \Pcal_\Gcal(m,c^\prime)$ is then uniquely defined by $\{\chi_1,\dots\chi_k\}$.   
From these considerations, we deduce that,
\begin{equation} \label{eq:preimagecycles}
|\tilde{\zeta}^{-1}(\{\chi_1,\dots\chi_k\})|= \prod_{j=1}^p \binom{k_j}{m_{j,1},\dots,m_{j,N}},
\end{equation}
where the binomial coefficients correspond to the number of ways of choosing which of the $k_j$ cycles $\chi \in \{\chi_1,\dots,\chi_k\}$ that satisfy $\vartheta(\chi)=\gamma_j$ are assigned colour $i$ for any $i \in [N]$. Since the right hand-side of \eqref{eq:preimagecycles} does not depend on the choice of $\{\chi_1,\dots,\chi_k\} \in \tilde{\Ecal}_m$, we obtain \eqref{eq:preparationRPM} and the proof is concluded.
\end{proof}

The next proposition provides a formula for the  cardinality of any equivalence class $\sigma \in \Sigma(\tilde{\Wcal}_{\Gcal})$. Recall from the beginning of this section that $\delta(\gamma)$ and $J(\gamma)$ denote the stretch-factor and the multiplicity of $\gamma \in \Sigma(\Lcal_\Gcal)$.

\begin{proposition} \label{prop:RPM}
For $w \in \tilde{\Wcal}_{\Gcal}$ satisfying the same assumptions as in  Lemma \ref{lemma:bijection} we have that,
\begin{equation}\label{eq:cardinalityequivalenceclassRPM}
|\sigma(w)|=\prod_{e \in \Ecal} m_e! \, \prod_{j=1}^p \frac{1}{m_{j,1}! \cdots m_{j,N}!} \bigg(\frac{\delta(\gamma_j)}{2 J(\gamma_j)}\bigg)^{k_j}.
\end{equation}
\end{proposition}
\begin{proof}
Suppose that $w \in \tilde{\Wcal}_{\Gcal}$ satisfies the assumptions of the proposition. 
In the following calculation, we use that any cycle $\chi \in \Sigma(\Rcal_m)$ such that $\vartheta(\chi)=\gamma \in \Lcal_\Gcal$ is the equivalence class of precisely $2\alpha(\gamma)$ r-o-l-loops. For $j \in [p]$ we set $\upsilon_j:= \binom{k_j}{m_{j,1},\dots,m_{j,N}}$ and we obtain from Lemma \ref{lemma:bijection} that 
$$
\begin{aligned}
|\sigma(w)| &= \prod_{j=1}^p \upsilon_j \, \Big| \big\{\{\chi_1,\dots\chi_k\} \in \Ecal_m: \, |\{n \in [k]: \, \vartheta(\chi_n)=\gamma_j\}|=k_j \, \forall j \in [k]\big\} \Big| \\ 
&= \prod_{j=1}^p \upsilon_j \bigg(\frac{1}{2\alpha(\gamma_j)}\bigg)^{k_j} \\
& \qquad \times \Big| \big\{\{L_1,\dots,L_k\} \subset \Rcal_m : \{\chi(L_1),\dots,\chi(L_k)\} \in \Ecal_m: \, |\{n \in [k]: \, \vartheta(L_n)=\gamma_j\}|=k_j \, \forall j \in [k]\big\} \Big| \\ 
& = \prod_{j=1}^p \upsilon_j \bigg(\frac{1}{2\alpha(\gamma_j)}\bigg)^{k_j} \, \frac{1}{k!} \\
& \qquad \times  \Big| \big\{(L_1,\dots,L_k) \in \Rcal_m^k: \, \{\chi(L_1),\dots,\chi(L_k)\} \in \Ecal_m: \, |\{n \in [k]: \, \vartheta(L_n)=\gamma_j\}|=k_j \, \forall j \in [k]\big\} \Big| \\ 
& = \prod_{j=1}^p \upsilon_j \bigg(\frac{1}{2\alpha(\gamma_j)}\bigg)^{k_j} \, \frac{1}{k!} \,  \prod_{e \in \Ecal} m_e! \, \Big| \big\{(\ell_1,\dots,\ell_k) \in \Sigma(\Lcal_\Gcal)^k: \, |\{n \in [k]: \, \ell_n \in \gamma_j\}|=k_j \, \forall j \in [k]\big\} \Big| \\
& = \prod_{j=1}^p \binom{k_j}{m_{j,1},\dots,m_{j,N}} \bigg(\frac{1}{2\alpha(\gamma_j)}\bigg)^{k_j} |\gamma_j|^{k_j} \, \frac{1}{k!} \binom{k}{k_1,\dots,k_p} \, \prod_{e \in \Ecal} m_e!\\ 
%& = \prod_{j=1}^p \frac{1}{m_{j,1}! \cdots  m_{j,N}!} \bigg(\frac{1}{2\alpha(\gamma_j)}\bigg)^{k_j} |\gamma_j|^{k_j} \, \prod_{e \in \Ecal} m_e!  \\ 
& =\prod_{e \in \Ecal} m_e! \, \prod_{j=1}^p \frac{1}{m_{j,1}! \cdots m_{j,N}!} \bigg(\frac{\delta(\gamma_j)}{2 J(\gamma_j)}\bigg)^{k_j},
\end{aligned}
$$
where in the last step we used \eqref{eq:rootedorientedloopcardinality} for the cardinality of $\gamma_j$ for each $j \in [p]$. This concludes the proof of the proposition.
\end{proof}

\subsection{Relation between RPM and RWLS} \label{sec:relationbetweenmodels}
We now state an equivalence relation between the RPM and the RWLS.
We first associate through a map,  $\Phi$, to each equivalence class in  $\Sigma(\Omega_\Gcal)$ a set of equivalence classes in $\Sigma(\tilde{\Wcal}_\Gcal)$, see also Figure \ref{fig:equivalenceboth}. 
After that,  we use this map to associate to any function
taking values in 
$\Omega_\Gcal$
which does `not depend' on `certain features' of the configurations in $\Omega_\Gcal$
a function taking values in $\tilde{\Wcal}_\Gcal$, which also does not depend on certain features of the configurations in $\tilde{\Wcal}_\Gcal$.
Our equivalence theorem, Theorem \ref{theo: equivalence} below,  states that the average with respect to the RWLS measure of the function taking values in $\Omega_\Gcal$ equals the average with respect to the RPM
of the corresponding function taking values in $\tilde{\Wcal}_\Gcal$.

To begin, recall the definitions of the maps $\Xi_1^*$ and $\Xi_2^*$ in (\ref{eq:defxi1}) and  (\ref{eq:defxi2}). For any $\rho \in \Sigma(\Omega_\Gcal)$, we define the set,
$$
\Phi(\rho):=(\Xi_2^*)^{-1}\big(\Xi_1^*(\rho)\big) \subset \Sigma(\tilde{\Wcal}_\Gcal).
$$
The map $\Phi$ is well-defined since $\Xi_1^*(\Omega_\Gcal)=\Xi_2^*(\tilde{\Wcal}_\Gcal)$. Let $\rho \in \Sigma(\Omega_\Gcal)$ be such that $\Xi_1^*(\rho)=\big\{(\gamma_1,k_1),\dots,(\gamma_p,k_p)\big\}$ for some $p \in \N$, $\gamma_1,\dots,\gamma_p \in \Sigma(\Lcal_\Gcal)$ and $k_1,\dots,k_p \in \N$. It follows from the definition of $\Phi$ that any $\sigma \in \Phi(\rho) \subset \Sigma(\tilde{\Wcal}_\Gcal)$ satisfies the following two properties:
\begin{enumerate}[(i)]
  \setlength{\itemsep}{0pt}
  \setlength{\parskip}{1pt}
    \setlength{\topsep}{0pt}
\item $\Xi_2^*(\sigma)=\Xi_1^*(\rho)$,
\item There exist $m_{j,i} \in \N_0$ such that $\sum_{i \in [N]} m_{j,i}=k_j$ for all $j \in [p]$ and such that $k_{\gamma_j,i}^2(\sigma)=m_{j,i}$ for any $j \in [p]$ and $i \in [N]$.
\end{enumerate}

\begin{definition}
We call a function $f: \tilde{\Wcal}_{\Gcal} \to \R$ \textit{colour-label-independent} (in short: \textit{c-l-independent}) if $f(w)=f(w^\prime)$ for any $w,w^\prime \in \tilde{\Wcal_\Gcal}$ such that $\Xi_2^*(w)=\Xi_2^*(w^\prime)$.
\end{definition}
Stated differently, a function which is c-l-independent does not depend on the colour of the cycles nor on the label of the links of the cycle. 
In particular, any c-l-independent function $f: \tilde{\Wcal}_{\Gcal} \to \R$ is constant on $\Phi(\rho)$ for each $\rho \in \Omega_\Gcal$. %With slight abuse of notation, we then let $f(\Phi(\rho))$ be the evaluation of $f$ at any configuration that is contained in an equivalence class of $\Phi(\rho)$.
We will now give examples of functions that are c-l-independent and of functions that are not c-l-independent. Recall from Section \ref{sect:randompathmodel} that $n_x^i(w)$ and $n_x(w)$ denote the number of pairings of $i$-links and the total number of pairings at $x$ for any $i \in [N]$, $x \in \Vcal$ and $w \in \tilde{\Wcal}_\Gcal$.

\begin{example} \label{ex:functions}
Consider the functions $f_i : \tilde{\Wcal}_\Gcal \to \R$, $i=1, \ldots ,6$ below,
which are defined for any $w \in \tilde{\Wcal}_\Gcal$ and $A \subset \mathcal{V}$ as,
\vspace{-0.3cm}
\begin{itemize}
\itemsep0em 
\item $f_1(w) := \mathbbm{1}_{\{\text{There exist two links on } \{o,\boldsymbol{e}_1\} \text{ that are paired together at both its endpoints}\}}(w)$ 
\item $ f_2(w) := n_o(w) $
\item $ f_3(w) := \mathbbm{1}_{\{\text{There exists a cycle that touches every vertex } x \in A\}}(w)$
\item $f_4(w) := \mathbbm{1}_{\{\text{There exist two links on } \{o,\boldsymbol{e}_1\} \text{ that have colour } 1 \}}(w)$,
\item $f_5(w) := n_o^1(w)$,
\item $f_6(w):= \mathbbm{1}_{\{\text{The link } (\{o,\boldsymbol{e}_1\},1) \text{ and the link } (\{o,\boldsymbol{e}_1\},2) \text{ are paired together at both its endpoints}\}}(w)$.
\end{itemize}
The first three functions are c-l-independent, while the last three functions are not.
\end{example}

Recall from Definition \ref{def:rolindependentfct} that we call a function $f: \Omega_\Gcal \to \R$ r-o-l-independent if it is constant on each equivalence class $\rho \in \Sigma(\Omega_\Gcal)$.
With the next definition we associate to any r-o-l independent function 
a c-l independent function.

\begin{definition}
Suppose that $f: \Omega_\Gcal \to \R$ is r-o-l-independent and that $\tilde{f}: \tilde{\Wcal}_{\Gcal} \to \R$ is c-l-independent. We write $f \sim \tilde{f}$ if $f(\rho)=\tilde{f}(\sigma)$ for any $\rho \in \Sigma(\Omega_{\Gcal})$ and any $\sigma \in \Phi(\rho)$.
\end{definition}
For instance, the r-o-l independent function $g: \Omega_\Gcal \to \R$ defined by $g(w):=\mathbbm{1}_{\{(o, \boldsymbol{e}_1,o) \in \omega \text{ or } (\boldsymbol{e}_1,o, \boldsymbol{e}_1) \in \omega\}}(\omega)$ for $\omega \in \Omega_\Gcal$ satisfies $g \sim f_1$, where $f_1: \tilde{\Wcal}_\Gcal \to \R$ is defined in Example \ref{ex:functions} above. Another important example is provided in the next lemma,
to which we will refer in the proof our main theorem, Theorem \ref{theo:maintheorem}.
For $x,y \in \Vcal$ we introduce the functions $\Ncal_{x,y}: \Omega_\Gcal \to \R$ and $\tilde{\Ncal}_{x,y}: \tilde{\Wcal}_\Gcal \to \R$, which count the number of loops in a configuration that visit $x$ and $y$. More precisely, for any $\omega \in \Omega_\Gcal$ such that $\omega=(\ell_1,\dots,\ell_n) \in \Lcal_\Gcal^n$ for some $n \in \N_0$, we set 
\begin{equation} \label{eq:Ncal}
\Ncal_{x,y}(\omega):=\sum_{j=1}^{n} \mathbbm{1}_{\{x,y \in \ell_j\}}(\omega)
\end{equation}
and for any $w \in \tilde{\Wcal}_\Gcal$, we set 
\begin{equation} \label{eq:tildeNcal}
\tilde{\Ncal}_{x,y}(w):= \sum_{\chi \in \zeta(w)} \mathbbm{1}_{\{x,y \in \chi\}}(w).
\end{equation}
It is easy to see that $\Ncal_{x,y}$ is r-o-l-independent, that $\tilde{\Ncal}_{x,y}$ is c-l-independent.  
\begin{lemma} \label{lemma:example}
For any $x,y \in \Vcal$, it holds that
\begin{equation} \label{eq:example}
\Ncal_{x,y} \sim \tilde{\Ncal}_{x,y}.
\end{equation}
\end{lemma}
\begin{proof} Consider $\rho \in \Sigma(\Omega_\Gcal)$ and $\sigma \in \Phi(\rho)$. It holds that 
$$
\Ncal_{x,y}(\rho)=\sum_{(\gamma,k) \in \Xi_1^*(\rho)} k \, \mathbbm{1}_{\{x \in \ell \, \forall \, \ell \in \gamma\}}= \sum_{(\gamma,k) \in \Xi_2^*(\sigma)} k \, \mathbbm{1}_{\{x \in \ell \, \forall \, \ell \in \gamma\}}=\tilde{\Ncal}_{x,y}(\sigma),
$$
where in the second step we used that $\Xi_1^*(\rho)=\Xi_2^*(\sigma)$. It thus holds that $\Ncal_{x,y} \sim \tilde{\Ncal}_{x,y}$. 
\end{proof}

We now have all ingredients to state and prove the equivalence theorem. In \eqref{eq:equivalencethm} below the expectation in the left hand-side was defined in \eqref{eq:generalloopsoupprob} and refers to the  random walk loop soup. The expectation in the right hand-side was defined in \eqref{eq:RPMprobabilitymeasure} and refers to the  random path model. The identity \eqref{eq:equivalencethm} states that the two models are equivalent if one looks at functions which do not depend on certain features of the configurations; the colour and the label of the links for the random path model, and the root, orientation and label of the 
rooted oriented loops for the random walk loop soup. 

\begin{theorem} \label{theo: equivalence}
For any $\rho \in \Sigma(\Omega_\Gcal)$ it holds that 
\begin{equation} \label{eq:comparisonweights}
\nu(\rho)=\mu^\ell(\Phi(\rho)).
\end{equation}
In particular, if $f: \Omega_\Gcal \to \R$ is r-o-l-independent, $\tilde{f}: \tilde{\Wcal}_\Gcal \to \R$ is c-l-independent and   $f \sim \tilde{f}$, then
\begin{equation} \label{eq:equivalencethm}
E_{\Gcal,U,v_\Gcal,N,\lambda} (f)=\E_{\Gcal,U,v_\Gcal,N,\lambda}(\tilde{f}).
\end{equation}
\end{theorem}

\begin{proof}
To begin, we fix $p \in \N$, $k_1,\dots,k_p \in \N$ and $\gamma_1,\dots,\gamma_p \in \Sigma(\Lcal_\Gcal)$ and we consider $\rho \in \Sigma(\Omega_\Gcal)$ such that $\Xi_1^*(\rho)=\big\{(\gamma_1,k_1),\dots,(\gamma_p,k_p)\big\}$. We will now calculate $\nu(\rho)$. Note that for any $\omega, \omega^\prime \in \rho$, it holds that $\nu(\omega)=\nu(\omega^\prime)$. We can thus fix a configuration $\omega=(\ell_1,\dots,\ell_k)  \in \rho$, where $k=\sum_{j=1}^pk_j$, and write
\begin{equation} \label{eq:measurerho}
\nu(\rho) = |\rho| \, \nu(\omega).
\end{equation}
Given the set $\Xi_1^*(\omega)=\Xi_1^*(\rho)$, we rewrite the sum of the interaction terms \eqref{eq:definitioninteractiongeneral} of $\nu(\omega)$ as 
\begin{equation} \label{eq:rewriteinteractionterm}
\begin{aligned}
\mathcal{V}( \omega) %& = \sum_{i,j=1}^k \sum_{n=0}^{|\ell_i|-1} \sum_{m=0}^{|\ell_j|-1} v_\Gcal(\ell_i(n),\ell_j(m)) %= \sum_{x,y \in \Vcal} v_\Gcal(x,y) \sum_{i=1}^k \sum_{n=0}^{|\ell_i|-1} \mathbbm{1}_{\{\ell_i(n)=x\}} \sum_{j=1}^k \sum_{m=0}^{|\ell_j|-1} \mathbbm{1}_{\{\ell_j(m)=y\}} \\
& = \sum_{x,y \in \Vcal} v_\Gcal(x,y) n_x(\boldsymbol{k},\boldsymbol{\gamma}) n_y(\boldsymbol{k},\boldsymbol{\gamma}),
\end{aligned}
\end{equation}
where $n_x(\boldsymbol{k},\boldsymbol{\gamma}):=\sum_{j=1}^p k_j n_x(\gamma_j)$ for $x \in \Vcal$. Plugging \eqref{eq: cardinalityequivalenceBEC} and \eqref{eq:rewriteinteractionterm} in \eqref{eq:measurerho} we obtain that
\begin{equation} \label{eq:weightequivalenceclassBEC}
\nu(\rho)= \prod_{j=1}^p \frac{1}{k_j!} \bigg(\frac{N\lambda^{\alpha(\gamma_j)}\delta(\gamma_j)}{2J(\gamma_j)}\bigg)^{k_j}  \prod_{x \in \Vcal}\bigg( U(n_x(\boldsymbol{k},\boldsymbol{\gamma}))\prod_{y \in \Vcal} e^{-v_\Gcal(x,y) n_x(\boldsymbol{k},\boldsymbol{\gamma}) n_y(\boldsymbol{k},\boldsymbol{\gamma})}\bigg).
\end{equation}

\begin{figure}
\centering
\begin{subfigure}{.3\textwidth}
  \centering
  \includegraphics[width=\textwidth]{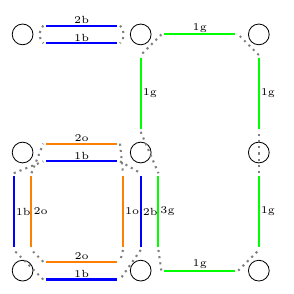}
  \caption{}
  \label{fig:sub1}
\end{subfigure}
\hspace{6em}
\begin{subfigure}{.3\textwidth}
  \centering
  \includegraphics[width=\textwidth]{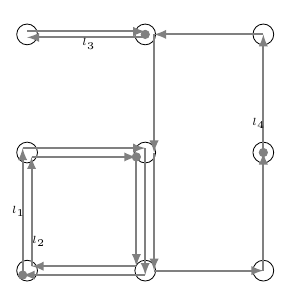}
  \caption{}
  \label{fig:sub2}
\end{subfigure}
\caption{(a) A configuration $w=(m,c,\pi) \in \tilde{\Wcal}_{\Gcal}$. (b) A configuration $\omega=(\ell_1,\ell_2,\ell_3,\ell_4) \in \Omega_\Gcal$. The roots of the rooted oriented loops in $\omega$ are represented by small gray filled circles. It holds that $\sigma(w) \in \Phi(\rho(\omega))$ and it is easy to see that $\Xi_1^*(\omega)=\Xi_2^*(w)$.}
\label{fig:equivalenceboth}
\end{figure}

Consider now $\sigma \in \Phi(\rho)$ and let $m_{j,i} \in \N_0$ be such that $\sum_{i \in [N]} m_{j,i}=k_j$ for any $j \in [p]$ and such that $k_{\gamma_j,i}^2(\sigma)=m_{j,i}$ for any $j \in [p]$ and $i \in [N]$. We note that for any $w,w^\prime \in \sigma$, it holds that $\mu(\sigma)=\mu(\sigma^\prime)$. We can thus fix a configuration $w \in \sigma$ and write
\begin{equation} \label{eq:measuresigma}
\mu^\ell(\sigma)=|\sigma| \, \mu^\ell(w).
\end{equation}
Using \eqref{eq:measuresigma} and  \eqref{eq:cardinalityequivalenceclassRPM} we obtain that
\begin{equation}
\label{eq:weightequivalenceclassRPM}
\mu^\ell(\sigma)= \prod_{j=1}^p \prod_{i=1}^N \frac{1}{m_{j,i}!} \bigg(\frac{\lambda^{\alpha(\gamma_j)} \delta(\gamma_j)}{2J(\gamma_j)}\bigg)^{k_j} \prod_{x \in \Vcal} \bigg(U(n_x(\boldsymbol{k},\boldsymbol{\gamma})) \prod_{y \in \Vcal} e^{-v_\Gcal(x,y)n_x(\boldsymbol{k},\boldsymbol{\gamma})n_y(\boldsymbol{k},\boldsymbol{\gamma})}\bigg).
\end{equation}
Applying \eqref{eq:weightequivalenceclassRPM}, we then obtain that 
$$ 
\begin{aligned}
& \mu^\ell(\Phi(\rho))= \sum_{\sigma^\prime \in \Phi(\rho)}\mu^\ell(\sigma^\prime) \\
& = \prod_{j=1}^p \bigg(\frac{\lambda^{\alpha(\gamma_j)} \delta(\gamma_j)}{2J(\gamma_j)}\bigg)^{k_j} \prod_{x \in \Vcal} \bigg(U(n_x(\boldsymbol{k},\boldsymbol{\gamma})) \prod_{y \in \Vcal} e^{-v_\Gcal(x,y)n_x(\boldsymbol{k},\boldsymbol{\gamma})n_y(\boldsymbol{k},\boldsymbol{\gamma})}\bigg) \bigg(\prod_{j=1}^p \sum_{\substack{m_{j,1},\dots,m_{j,N} \in \N_0: \\ m_{j,1}+\dots+m_{j,N}=k_j}} \frac{1}{m_{j,1}! \cdots m_{j,N}!}\bigg) \\
& = \prod_{j=1}^p \frac{1}{k_j!}\bigg(\frac{\lambda^{\alpha(\gamma_j)} N\delta(\gamma_j)}{2J(\gamma_j)}\bigg)^{k_j} \prod_{x \in \Vcal} \bigg(U(n_x(\boldsymbol{k},\boldsymbol{\gamma})) \prod_{y \in \Vcal} e^{-v_\Gcal(x,y)n_x(\boldsymbol{k},\boldsymbol{\gamma})n_y(\boldsymbol{k},\boldsymbol{\gamma})}\bigg) = \nu(\rho),
\end{aligned}
$$
where we used \eqref{eq:weightequivalenceclassBEC} in the last step. This proves \eqref{eq:comparisonweights}. Let now $f: \Omega_{\Gcal} \to \R$ be r-o-l-independent and $\tilde{f}: \tilde{\Wcal}_\Gcal \to \R$ be c-l-independent such that $f \sim \tilde{f}$. Applying \eqref{eq:comparisonweights}, we have that
\begin{equation} \label{eq:identicalmeasure}
\mu^\ell\big(\tilde{f}\big)=\sum_{\rho \in \Sigma(\Omega_\Gcal)} \mu^\ell\big(\tilde{f}\mathbbm{1}_{\{\sigma \in \Phi(\rho)\}}\big)=\sum_{\rho \in \Sigma(\Omega_\Gcal)}f(\rho) \nu(\rho)=\nu\big(f\big).
\end{equation} 
In particular, the partition functions of both models, corresponding respectively to $\nu(\Omega_\Gcal)$ and $\mu^\ell(\tilde{\Wcal}_\Gcal)$, are identical. % (choosing $f=1=\tilde{f}$). 
This implies that \eqref{eq:identicalmeasure} also holds true for the normalized measures giving \eqref{eq:equivalencethm} and thus concludes the proof of the theorem.
\end{proof}
In the proof of our main theorem, Theorem \ref{theo:maintheorem}, such an equivalence theorem together with Lemma \ref{lemma:example} will be used to extend to the RWLS
our result about occurrence of macroscopic loops, which will be first be proved for the RPM. From now on the paper will deal with the RPM.

\section{Reflection positivity and Chessboard estimate}
\label{sect:reflectionpositivityandchessboardestimate}
We now introduce reflection positivity,  and its immediate consequence, the chessboard estimate. 
The novelty of this section with respect to \cite{T} is that  the technique of reflection positivity is extended to the measure \eqref{eq:RPMmeasure}, in which long-range interactions  are present. 
Consider the $d$-dimensional torus $(\T_L, \E_L)$ with $d,L \in \N$, i.e., the graph with vertex set $\T_L=\{ x = (x_1, \ldots, x_d) \in \mathbb{Z}^d \,  : \, x_i \in (-\frac{L}{2}, \frac{L}{2}] \text{ for each } i \in [d] \}$ and with edges connecting nearest neighbour vertices. Similar to  \cite{T},  the derivation of the Key Inequality in Section \ref{sect:derivationofthekeyinequality} below requires the introduction of a new graph.
\paragraph{Extended torus, virtual  and original vertices.}
We  now view $(\T_L, \E_L)$ as the sub-graph of a larger graph embedded in $\mathbb{R}^{d+1}$, which is denoted by 
$(\mathcal{T}_L, \mathcal{E}_L)$ and  is  referred to as \textit{extended torus}. The extended torus is obtained from the $d$-dimensional torus  by duplicating the vertex-set and by adding an edge between every vertex in $\T_L$ and its copy.
More precisely,   we define the vertex set of the extended torus as,
$$
\mathcal{T}_L : =  \big
\{ \,  (x_1, \ldots, x_{d+1} )  \in \mathbb{Z}^{d+1} \, : \,  x_i \in (- \frac{L}{2} , \frac{L}{2}  ] \,  \mbox{ for every } i \in [d], \,  \mbox{ and }  \,  x_{d+1} \in \{1,2\} \,   \big   \},
$$
where
$\T_L  = \{   (x_1,  \ldots, x_{d+1} )\in \mathcal{T}_L \, \, : \, \, x_{d+1} = 1    \}
 \subset \mathcal{T}_L$,  and $\T_L^{(2)} : = \mathcal{T}_L \setminus \T_L$.  
We define the edge-set,
$$
\mathcal{E}_L := \, \, \, \E_L \cup   \big \{ \{x,y\} \, 
\subset \mathbb{Z}^{d+1} : \, \,  x \in \T_L,  \, y = x + (0, \ldots,0, 1) \,        \big \}.
$$
This defines the extended torus $(\mathcal{T}_L, \mathcal{E}_L)$.
We will refer to  the vertices in $\T_L \subset \mathcal{T}_L$ as \textit{original} and to the vertices in  $\T_L^{(2)} \subset \mathcal{T}_L$ as \textit{virtual}. 
From now on we will replace the sub-script $\Gcal$ in all the quantities which were defined above by $\Tcal_L$ or $\T_L$, when considering the random path model in the extended or in the original torus. 
We  keep referring to $o$, corresponding to the vertex $(0, \ldots, 0) \in  \T_L \subset \mathcal{T}_L \subset \mathbb{Z}^{d+1}$, as the origin.
\subsection{Reflection positivity}
\label{sect:RP}
We now introduce reflection positivity, partially following  \cite[Section 4.2]{T}. %From now on the underlying graph $\Gcal$ will be a torus of side length $L$, $(\T_L,\E_L)$ in dimension $d \geq 2$ with edges connecting nearest-neighbour vertices. 
%The function $v_\Gcal: \T_L \times \T_L \to \R$ will always be given by $v_\Gcal(x,y)= \sum_{z \in \Z^d} v(y+Lz-x)$ for some potential $v: \Z^d \to \R$. We will write $\mu_{L,U,v,N,\lambda}$ and $\P_{L,U,v,N,\lambda}$ instead of $\mu_{\Gcal,U,v_\Gcal,N,\lambda}$ and $\P_{\Gcal,U,v_\Gcal,N,\lambda}$. For lighter notation we will often omit the sub-scripts.

\paragraph{Domains and restrictions.}
We introduce the notion of domain and restriction. 
A function $f : \mathcal{W}_{\Tcal_L} \rightarrow \mathbb{C}$ has \textit{domain} $D \subset \Tcal_L$ if, for any
pair of configurations $w = (m, c, \pi, g), w^{\prime} = (m^{\prime}, c^{\prime}, \pi^{\prime}, g^{\prime}) \in \mathcal{W}_{\Tcal_L}$ such that 
$$
\quad \forall e \in \mathcal{E}_L :  \,e\cap D\neq \emptyset, \quad \forall z \in D, \quad m_e = m_e^{\prime} \quad c_e = c_e^{\prime}, \quad \pi_z = \pi_z^{\prime}, \quad g_z=g_z^\prime
$$
one has that 
$f(w) = f(w^{\prime})$.
Moreover, for any  $w=(m, c, \pi,g) \in \mathcal{W}_{\Tcal_L}$ we define the \emph{restriction} of $w$ to $D \subset \Tcal_L$, $w_D=(m_D, c_D, \pi_D, g_D) \in \Wcal_{\Tcal_L}$,
%with $m_D \in \Mcal_{\Tcal_L}$, $c_D \in \mathcal{C}_{\Tcal_L}(m_D)$, $\pi_D\in \mathcal{P}_{\Tcal_L}(m_D, c_D)$, $g_D \in \Hcal_{\Tcal_L}$
 by 
  \vspace{-0.1cm}
 \begin{enumerate}[i)]
   \setlength{\itemsep}{0pt}
    \setlength{\topsep}{0pt}
  \vspace{-0.1cm}
 \item $(m_D)_e = m_e$ for any edge $e \in \mathcal{E}_L$ which has at least one end-point in $D$ and $(m_D)_e =0$ otherwise,
  \vspace{-0.1cm}
  \item $(c_D)_e = c_e$ for any edge $e$ which has at least one end-point in $D$ and $(c_D)_e= \emptyset$ otherwise,
 \vspace{-0.1cm}
  \item $(\pi_D)_x = \pi_x$ for any $x \in D$, and for $x\in  \Tcal_L \setminus D$ we set $(\pi_D)_x$ as the pairing which leaves all links touching $x$ unpaired (if any)
   \vspace{-0.1cm}
 \item $(g_D)_x = g_x$ for any $x \in D$, and for $x\in  \Tcal_L \setminus D$ we set $(g_D)_x=0$.
   \vspace{-0.1cm}
\end{enumerate}

\paragraph{Reflection through edges.}
Recall that the graph $(\mathcal{T}_L, \mathcal{E}_L)$ is embedded in $\mathbb{R}^{d+1}$.
We say that the plane $R$ is  through the  edges of $(\mathcal{T}_L, \mathcal{E}_L)$ if it is orthogonal to one of the cartesian vectors $\boldsymbol{e_i}$ for $i \in [d]$ %(and not $i = d+1$)
and if it  intersects the midpoint of
$L^{d-1}$ edges of the graph $(\mathcal{T}_L,\Ecal_L)$,  i.e.
$R = \{z \in \mathbb{R}^{d+1} \, \, : \, \, z \cdot \boldsymbol{e_i} = u      \}$, for some $u$ such that  $u - 1/2 \in \mathbb{Z} \cap (-\frac{L}{2},\frac{L}{2}]$
and $i \in [d]$. 
Given such 
a plane $R$, we denote by 
$\Theta : \mathcal{T}_L \rightarrow \mathcal{T}_L$  the reflection operator which reflects the vertices of  $\,\mathcal{T}_L$ with respect to $R$, i.e.
for any $x = (x_1, x_2, \ldots, x_{d+1}) \in \mathcal{T}_L$, 
\begin{equation}
\Theta(x)_k : = 
\begin{cases}
x_k & \mbox{ if } k \neq i, \\ 
2u - x_k   \mod  L & \mbox{ if } k = i.
\end{cases}
\end{equation}
Let $\mathcal{T}_L^+, \mathcal{T}_L^- \subset \mathcal{T}_L$ be the corresponding partition of the extended torus into two disjoint halves
such that $\Theta(\mathcal{T}_L^\pm) = \mathcal{T}_L^{\mp}$. 
Let $\Ecal^{+}_L, \Ecal^{-}_L \subset \Ecal_L$, be the set of edges $\{x,y\}$ with at least one of $x,y$ in $\mathcal{T}_L^+$ respectively $\mathcal{T}_L^-$. Moreover, let 
$\Ecal_L^R  := \Ecal^{+}_L \cap \Ecal^{-}_L$. Note that this set contains $2 L^{d-1}$ edges, half of them intersecting the plane $R$, and all of them  belonging to $\E_L$.
Further, let $\Theta  : \mathcal{W}_{\Tcal_L} \rightarrow \mathcal{W}_{\Tcal_L}$ denote the reflection operator reflecting the configuration $w=(m, c, \pi, g)$ with respect to $R$ (we commit an abuse of notation by using the same letter). More precisely we define $\Theta w=(\Theta m, \Theta c, \Theta \pi, \Theta g)$, where $(\Theta m)_{\{x,y\}}=m_{\{\Theta x,\Theta y\}}$, 
$(\Theta c)_{\{x,y\}}=c_{\{\Theta x,\Theta y\}}$, 
 $(\Theta \pi)_x=\pi_{\Theta x}$, and  $(\Theta g)_x= g_{\Theta x}$.
Given a function $f :\mathcal{W}_{\Tcal_L} \to \mathbb{C}$, we also use the letter $\Theta$ to denote the reflection operator $\Theta$ which acts on $f$ as $\Theta f(w) : = \overline{f(\Theta w)}$.
We denote by $\Acal^\pm$ the set of functions with domain $\mathcal{T}_L^\pm$ and denote by $\mathcal{W}_{\Tcal_L}^\pm$ the set of configurations $w\in\mathcal{W}_{\Tcal_L}$ that are obtained as a restriction of some $w^\prime \in \mathcal{W}_{\Tcal_L}$ to $\mathcal{T}_L^\pm$.

We remark  that, although the graph $(\mathcal{T}_L, \mathcal{E}_L)$ is embedded in $\mathbb{R}^{d+1}$, we will only consider reflections with respect to reflection planes which are orthogonal to one of the cartesian vectors $\boldsymbol{e}_i$ for $i \in [d]$ (and not $i=d+1$).

Recall from the introduction that for a potential $v:\Z^d \to \R$, the function $v_L: \T_L \times \T_L \to \R$ is defined by $v_L(x,y):= \sum_{z \in \mathbb{Z}^d} v(y+Lz-x)$ for any $x,y \in \T_L$. From now on, we will view $v_L$ as a function on $\Tcal_L \times \Tcal_L$ by assuming that $v_L(x+\boldsymbol{e}_{d+1},y +\boldsymbol{e}_{d+1})=v_L(x+\boldsymbol{e}_{d+1},y)=v_L(x,y)$ for any $x,y \in \T_L$.

\begin{definition} \label{def: separable}
The function $v_L: \Tcal_L \times \Tcal_L \to \mathbb{R}$ is called $i$-\emph{separable}, $i \in [d]$, if there exists an integer $k \in \mathbb{N}$, a measure space $(T, \mathcal{T}, \nu)$, where $\nu$ is a finite, non-negative measure, and two bounded measurable functions 
$$
\alpha, \tilde{\alpha}: T \times \mathbb{T}_L \to \mathbb{C}^k,
$$
such that for any, $x,y \in \mathbb{T}_L$ with $x_i > 0 > y_i$, we can write
\begin{equation} \label{eq: interaction}
-v_L(x,y)= \int_{T} d \nu(t) \, \alpha(t,x) \cdot \tilde{\alpha}(t,y),
\end{equation}
where $\alpha(t,x) \cdot \tilde{\alpha}(t,y)$ denotes the inner product of two vectors in $\mathbb{C}^k$ and such that it holds
\begin{equation} \label{eq: conjugate}
\forall x=(x_1,\dots,x_d) \in \mathbb{T}_L \quad \quad \tilde{\alpha}(t, x) = \overline{\alpha(t, x_1, \dots, x_{i-1}, -x_i, x_{i+1}, \dots, x_d)} \quad \nu- \text{almost surely.}
\end{equation}
We say that $v$ is \emph{separable} if $v_L$ is $i$-separable for any $i \in [d]$.
\end{definition}

Recall also the definition in \eqref{eq:tempered} of a tempered potential $v: \Z^d \to \R$. Two examples of potentials $v$ that are tempered and separable are given in Lemma \ref{lemma:examplesseparable} in the appendix.

The next  theorem introduces an important tool. The theorem states that the random path model is reflection positive. 

\begin{theorem}[\textbf{Reflection positivity}]\label{theo:RP1}
Consider the torus $(\mathcal{T}_{L},\Ecal_{L})$ for $L\in2\N$. Let $R$ be a reflection plane through edges, which is orthogonal to one of the cartesian vectors $\boldsymbol{e}_i$, $i \in \{1, \ldots, d\}$, let $\Theta$ be the corresponding reflection operator. Consider the random path model with  $N \in \mathbb{N}$,  $\lambda \in \mathbb{R}^+$, and weight function $U:\N_0 \to \R^+$. Let $v: \Z^d \to \mathbb{R}$ be tempered and separable. 
For any pair of functions
$f, g \in \mathcal{A}^+$, we have that,
\begin{enumerate}
\item[(1)] $\mu_{\Tcal_L,U,v,N,\lambda} (f\Theta g)=\overline{\mu_{\Tcal_L,U,v,N,\lambda} (g \Theta f)}$,
\item[(2)] $\mu_{\Tcal_L,U,v,N,\lambda} (f\Theta f)\geq 0$.
\end{enumerate}
From this we obtain that,
\begin{equation}\label{eq:RPCS}
Re\big(\mu_{\Tcal_L,U,v,N,\lambda} \big (   f \,  \Theta g  \big )\big)
\leq 
\mu_{\Tcal_L,U,v,N,\lambda} \big (   f  \,  \Theta f  \big )^{\frac{1}{2}}
\, \, 
\mu_{\Tcal_L,U,v,N,\lambda} \big (   g \,  \Theta g  \big )^{\frac{1}{2}}.
\end{equation}
\end{theorem}

\begin{proof}
The proof of the case where $v$ equals the zero function, denoted $v=0$, is analogous to the proof of Theorem 4.3 in \cite{T} and will not be stated here. The only difference is that here we deal with complex functions and that, in addition to links, colourings and pairings, also ghost pairings are present. 
%and unpaired links receive a weight
%It is easy to see that the presence of ghost pairings does not change the proof.
%The proof of \cite[Theorem 4.3]{T} can be modified to the existence of ghost pairings by defining the projection $P_R: \Wcal_\Gcal \to \Wcal_\Gcal^R$ such that, for any $w=(m,c,\pi,g) \in \Wcal_\Gcal$, $P_R(w):=(m^R,c^R,\pi^R,g^R)$ with $m_e^R=m_e \mathbbm{1}_{\{e \in \Ecal_L^R\}}$, $c_e^R =c_e$ if $e \in \Ecal_L^R$ and $c_e^R=\emptyset$ otherwise, all links are unpaired at every vertex and $g_x^R=0$ for all $x \in \Vcal$. Then, as in the proof of \cite[Theorem 4.3]{T} we use that, given a triplet of configurations $w^\prime \in \Wcal_\Gcal^R$, $w_1 \in \Wcal_\Gcal^+$, $w_2 \in \Wcal_\Gcal^-$ such that $P_R(w_1)=P_R(w_2)=w^\prime$, there exists a unique configuration $w \in \Wcal_\Gcal$ such that
%$$
%w_{\Tcal_L^+} =w_1, \quad w_{\Tcal_L^-}=w_2, \quad P_R(w)=w^\prime.
%$$
%Further note that the function $F_x:=U(n_x)$ is invariant under reflections, i.e. $\Theta F_x=F_{\Theta x}$. 

We now present the proof in case   $v: \Z^d \to \mathbb{R}$ is a non-zero,  tempered and separable potential. It holds that
\begin{equation} \label{eq:measurezerononzero}
\mu_{\Tcal_L,U,v,N,\lambda}(\omega) = \mu(\omega) e^{-V(\omega)},
\end{equation}
where we use the notation $\mu=\mu_{\Tcal_L,U,0,N,\lambda}$ with $v=0$ denoting the zero function. 
By translation invariance of the random path measure we can assume without loss of generality that the reflection plane $R$ passes through the origin. 
We can decompose the terms of the function $V$, which is defined in \eqref{eq:defV}, into three groups, those involving sites lying entirely in $\Tcal_L^+$ or $\Tcal_L^-$ and those involving sites that lie in both halves of the torus:
\begin{equation} \label{eq:decomposition}
\begin{aligned}
V(\omega) %& = \sum_{x,y \in \T_L}  v_L(x,y)\tilde{n}_x(\omega)\tilde{n}_y(\omega) \\
& = \underbrace{\sum_{x,y \in \T_L^+} v_L(x,y) \tilde{n}_x(\omega)\tilde{n}_y(\omega)}_{=:V^+} + \underbrace{\sum_{x,y \in \T_L^-}  v_L(x,y)\tilde{n}_x(\omega)\tilde{n}_y(\omega)}_{= \Theta (V^+)} + 2 \underbrace{ \sum_{x \in \T_L^+, y \in \T_L^-}  v_L(x,y)\tilde{n}_x(\omega)\tilde{n}_y(\omega)}_{=:V_R},
\end{aligned}
\end{equation}
where $\tilde{n}_z(w):=n_z(w)+n_{z + \boldsymbol{e}_{d+1}}(w)$ for any $z \in \T_L$ and $w \in \Wcal$. 
We note that $V(\theta w)=V(w)$ and $\mu_{\Tcal_L,U,v,N,\lambda}(\theta w)=\mu_{\Tcal_L,U,v,N,\lambda}(w)$ for any $w \in \Wcal_{\Tcal_L}$. The proof of \textit{(1)} is thus the same as in \cite{T}. We now prove \textit{(2)}. Using \eqref{eq: interaction} we obtain for each $l \in \mathbb{N}$ that,
\begin{equation} \label{eq: factorcalculation}
\begin{aligned} 
(-V_R)^l & = \sum_{\substack{x^1, \dots, x^l \in \T_L^+ \\ y^1, \dots, y^l \in \T_L^-}} (-v_L(x^1,y^1))\tilde{n}_{x^1}\tilde{n}_{y^1} \ldots (-v_L(x^l,y^l)) \tilde{n}_{x^l} \tilde{n}_{y^l} \\
%& = \sum_{\substack{x^1, \dots, x^l \in \T_L^+ \\ y^1, \dots, y^l \in \T_L^-}} \int_T d \nu(t^1) \alpha(t^1,x^1) \cdot \tilde{\alpha}(t^1,y^1) \tilde{n}_{x^1} \tilde{n}_{y^1} \ldots \int_T d \nu(t^l) \alpha(t^l,x^l) \cdot \tilde{\alpha}(t^l,y^l) \tilde{n}_{x^l} \tilde{n}_{y^l} \\
%& = \sum_{\substack{x^1, \dots, x^l \in \T_L^+ \\ y^1, \dots, y^l \in \T_L^-}} \int_T d \nu(t^1) \Big(\sum_{j_1=1}^k \alpha_{j_1}(t^1,x^1) \tilde{\alpha}_{j_1}(t^1,y^1)\Big) \tilde{n}_{x^1} \tilde{n}_{y^1} \ldots \int_T d \nu(t^l) \Big(\sum_{j_l=1}^k \alpha_{j_l}(t^l,x^l) \tilde{\alpha}_{j_l}(t^l,y^l)\Big) \tilde{n}_{x^l} \tilde{n}_{y^l} \\
& = \int_T d\nu(t^1) \ldots \int_T d\nu(t^l)  \sum_{j_1, \dots, j_l \in \{1, \dots, k \}} \bigg( \sum_{x^1, \dots, x^l \in \T_L^+} \alpha_{j_1}(t^1,x^1) \cdots \alpha_{j_l}(t^l,x^l) \tilde{n}_{x^1} \cdots \tilde{n}_{x^l} \bigg) \\
& \qquad \qquad \qquad \qquad \qquad \qquad \qquad \qquad \qquad \qquad \qquad \times \bigg( \sum_{y^1,\dots, y^l \in \T_L^-} \tilde{\alpha}_{j_1}(t^1,y^1) \cdots \tilde{\alpha}_{j_l}(t^l,y^l) \tilde{n}_{y^1} \cdots \tilde{n}_{y^l} \bigg)\\
& = \int_T d\nu(t^1) \cdots \int_T d\nu(t^l)  \sum_{j_1, \dots, j_l \in \{1, \dots, k \}} F_{t^1, \dots, t^l, j_1, \dots, j_l} \Theta F_{t^1, \dots, t^l, j_1, \dots, j_l},
\end{aligned}
\end{equation}
where $F_{t^1, \dots, t^l,j_1, \dots, j_l}: \mathcal{W} \to \mathbb{C}$ is defined by
$$
F_{t^1, \dots, t^l,j_1, \dots, j_l}(w):=\sum_{x^1,\dots, x^l \in \T_L^+} \alpha_{j_1}(t^1,x^1) \cdots \alpha_{j_l}(t^l,x^l) \tilde{n}_{x^1}(w) \cdots \tilde{n}_{x^l}(w).
$$
Using \eqref{eq:measurezerononzero}, \eqref{eq:decomposition}, \eqref{eq: factorcalculation}  and  a Taylor expansion, we obtain that
$$
\begin{aligned}
& \mu_{\Tcal_L,U,v,N,\lambda}\big(f \Theta f\big)  
= \mu(f e^{-V^+} \Theta (fe^{-V^+}) e^{-2 V_R}) 
= \sum_{l \geq 0} \frac{2^l}{l!} \, \mu(f e^{-V^+} \Theta (fe^{-V^+}) (-V_R)^l) \\ 
& \quad = \sum_{l \geq 0} \frac{2^l}{l!} \int_T d\nu(t^1) \cdots \int_T d\nu(t^l)  \sum_{j_1, \dots, j_l \in \{1, \dots, k \}} \underbrace{\mu(f e^{-V^+} F_{t^1, \dots t^l, j_1, \dots, j_l} \Theta (f e^{-V^+} F_{t^1, \dots t^l, j_1, \dots, j_l}))}_{\geq 0} \geq 0,
\end{aligned}
$$
where in the last step we used that the measure $\mu$ with $v=0$ is reflection positive. 
Thus, the theorem is proven.
\end{proof}

\subsection{Chessboard estimate}
\label{sect:Chessboardscheme}
We now introduce the notion of support.
Contrary to the notion of domain, 
which was introduced in Section \ref{sect:RP},
the notion of support is defined only for subsets of the original torus.
We say that the function
 $f : \mathcal{W}_{\Tcal_L} \rightarrow \mathbb{R}$ has \textbf{\textit{support}} in 
 $D \subset \T_L$ if it has domain in $D \cup   D^{(2)}$,
 where $D^{(2)}$ is defined as the set of sites which are `on the top' of those in $D$, 
$$
  D^{(2)} : = 
  \{ z \in \mathbb{T}_L^{(2)} \, \, : z-\boldsymbol{e}_{d+1} \in D \}.
 $$
  Fix an arbitrary site $t \in \T_L$ and let  $t_0=o$, $t_1$, $\ldots$, $t_k = t$
be a self-avoiding nearest-neighbour path from $o$ to $t$,
and for any $i \in \{1, \ldots, k\}$, let $\Theta_i$ be the reflection with respect to the plane going through the edge $\{  t_{i-1}, t_{i}  \}$.
Let $f$ be a function having support in $\{o\}$ and define
$$
f^{ [t]} : = \Theta_k  \circ \Theta_{k-1} \, \ldots \, \circ \Theta_1 \, ( f ).
$$
Observe that the function $f^{ [t]}$ does not depend on the chosen path (see also 
Figure $6$ in \cite{T}). 
The next proposition provides an important tool that can be  applied to upper bound the measure of a product of local functions.
\begin{proposition}[Chessboard estimate]\label{prop:chessboardabstract}
Let $f = (f_t)_{t \in \T_L}$ be real-valued functions
with support $\{o\}$ which are either all bounded or all non-negative.
Under the same assumptions as in Theorem \ref{theo:RP1}, 
we have that
\begin{equation} \label{eq:chessboardabstract}
\mu_{\Tcal_L,U,v,N,\lambda} \Big(   \prod_{t \in \T_L} f^{[t]}_t  \Big   ) \, \leq \, 
\bigg (
\,  \, \prod_{t \in \T_L} \, \, 
   \mu_{\Tcal_L,U,v,N,\lambda}  \Big (   \prod_{s \in \T_L} f_t^{[s]}  \Big )   \, \, \, \,  \bigg )^{\frac{1}{|\T_L|}}.
\end{equation}
\end{proposition}
For the proof of Proposition \ref{prop:chessboardabstract} we refer to the original paper \cite{FrohlichLieb} or to the overviews  \cite[Theorem 5.8]{Biskup} or \cite[Theorem 10.11]{FriedliVelenik}.

The following remark states that the chessboard estimate can also be applied to the expectation $\E_{\T_L}$ for the RPM in the original torus with only closed paths being present (recall the definition of the expectation $\E$ provided right after equation 
\eqref{eq:RPMprobabilitymeasure}). 

\begin{remark} \label{remark:chessboard}
For any $F: \Wcal_{\T_L} \to \R$, we let $f_F : \Wcal_{\Tcal_L} \to \R$ be a function 
which extends $F$ to $\mathcal{W}_{\mathcal{T}_L}$ in such a way that 
 $f_F(w):= F(w_{\T_L})$ for any $w \in \mathcal{W}_{\Tcal_L}$,
where we recall from Section \ref{sect:RP} that 
$w_{{\T_L}}$ denotes the restriction of 
 $w \in \Wcal_{\Tcal_L}$ to $\mathbb{T}_L$ and can be viewed as a configuration in $\Wcal_{\T_L}$.  
We then have from our construction of the extended torus that, 
$$
\E_{\T_L, U, v, N, \lambda}[F]= \frac{\mu_{\Tcal_L, U, v, N, \lambda}\big(f_F \, \mathbbm{1}_{\tilde{\Wcal}_{\Tcal_L}^o} \big)
}{ \mathbb{Z}^{\ell}_{\mathbb{T}_L} },
$$
where
\begin{equation} \label{eq:extendedtorusclosedpaths}
\tilde{\Wcal}_{\Tcal_L}^o:=\{w \in \tilde{\Wcal}_{\Tcal_L}: \, n_z =0 \, \forall z \in \T_L^{(2)}\},
\end{equation}
corresponds to the set of configurations $w \in \tilde{\Wcal}_{\Tcal_L}$ in which any closed path lies entirely in the original torus (recall also that $\tilde{\Wcal}_{\Tcal_L}$, defined in \eqref{eq:randompathtilde}, is the set of configurations with no unpaired link and no ghost pairing).
From this relation, from the fact that $$\mathbbm{1}_{\tilde{\Wcal}^o_{\mathcal{T}_L}} = \prod_{x \in \mathbb{T}_L} \mathbbm{1}_{\{n_{x + \boldsymbol{e}_{d+1}} = 0 \}} \mathbbm{1}_{\{u_x=u_{x + \boldsymbol{e}_{d+1}} = 0 \}} \mathbbm{1}_{\{g_x=g_{x + \boldsymbol{e}_{d+1}} = 0 \}}$$
and from Proposition \ref{prop:chessboardabstract} we deduce that  \eqref{eq:chessboardabstract} also holds true with $\mu_{\Tcal_L}$ replaced by $\E_{\T_L}$. 
\end{remark}

The next lemma provides an upper bound on the local time
for the RPM with only closed paths being present.
\begin{lemma} \label{lemma:Chessboardapplicationno}
Under the same assumptions as in Theorem \ref{theo:RP1}, there exists $c \in (0,\infty)$ such that for any $A \subset \T_L$ and any vector $(a_x)_{x \in A} \in \N_0^{A}$, it holds for any $L \in 2\N$ that  
\begin{equation} \label{eq:upperboundnumberpairings}
\E_{\T_L, U, v, N, \lambda}\bigg(\prod_{x \in A} e^{a_x n_x}\bigg) \leq e^{c \lambda N d \sum_{x \in A} e^{a_x}}.
\end{equation}
\end{lemma}
\begin{proof} 
Consider $A \subset \T_L$ and $(a_x)_{x \in \T_L} \in \N_0^{\T_L}$. 
From \eqref{eq:chessboardabstract} and from Remark \ref{remark:chessboard} we  obtain  that
$$
\E_{\T_L, U, v, N, \lambda}\bigg(\prod_{x \in A} e^{a_x n_x}\bigg)  
\leq \bigg(\prod_{x \in A} \E_{\T_L, U, v, N, \lambda}\big(\prod_{y \in \T_L} e^{a_x n_y}\big)\bigg)^{\frac{1}{|\T_L|}} \leq e^{c \lambda N d \sum_{x \in A} e^{a_x}},
$$
where in the last step we applied \eqref{eq:finite}. This concludes the proof. %We used $\Z^\ell \geq 1$.
\end{proof}

\section{Two-point function and its properties} \label{sect:two-pointfunctionanditsproperties}
In this section we define the two-point function and study its properties. This function is an important object in our analysis and corresponds to a fraction between partition functions. We will later, in Proposition \ref{lemma:expectednumberloopsrelation} below, relate the two-point function to the expected number of closed paths touching two vertices
in the random path model with only closed paths.  
In the whole section, we consider the RPM on the (original) torus $(\T_L,\E_L)$. 
Recall from Section \ref{sect:randompathmodel} that $u_x^i(w)$ denotes the number of $i$-links touching $x$ that are unpaired at $x$ and that $g_x(w)$ denotes the number of ghost pairings at $x$ for any $x \in \T_L$, $i \in [N]$ and $w \in \Wcal_{\T_L}$. 
Recall from \eqref{eq:definitionrange} that we denote by $R \in \N_0 \cup \{\infty\}$ the range of $U:\N_0 \to \R_0^+$.

\begin{definition}
For $x,y \in \T_L$ with $x \neq y$, we define the sets
\begin{equation} \label{eq:Wcalxy}
\begin{aligned}
\Wcal_{x,y}:= \Big\{w \in \Wcal_{\T_L} : \quad  & u_z=g_z=0 \, \forall z \in \T_L \setminus \{x,y\}, \\ & u_z^1=1=u_z^2, \, u_z^i=0 \, \forall i \in \{3,\dots,N\}, \, g_z=2 \, \forall z \in \{x,y\} \Big\}
\end{aligned}
\end{equation}
and
\begin{equation} \label{eq:Wcalxx}
\begin{aligned}
\Wcal_{x,x}:= \Big\{w \in \Wcal_{\T_L} : \quad & u_z=g_z=0 \, \forall z \in \T_L \setminus \{x\}, \\ & u_x^1=2 \text{ or } u_x^2=2 \text{ or } u_x=0 \text{ and } u_x^i=0 \, \forall i \in \{3,\dots,N\}, \, g_x=4 \Big\}.
\end{aligned}
\end{equation} 
Moreover, we define the \textit{directed partition function}, 
\begin{equation}
\label{eq: definition_partition_function}
\Z_{\T_L,U,v,N,\lambda}(x,y) 
:= \begin{cases} \mu_{\T_L,U,v,N,\lambda}(\Wcal_{x,y}) & \text{ if } x \neq y, \\
\mu_{\T_L,U,v,N,\lambda}(2^{\mathbbm{1}_{\{u_x^1=2\}}+\mathbbm{1}_{\{u_x^2=2\}}} \, \mathbbm{1}_{\Wcal_{x,x}}) & \text{ if } x = y,
\end{cases}
\end{equation}
and  the \textit{two-point function},
\begin{equation}
\label{eq: twopointfunction}
\G_{\T_L,U,v,N,\lambda}(x,y):= \lambda^2 \, \frac{\Z_{\T_L,U,v,N,\lambda}(x,y)}{\Z_{\T_L,U,v,N,\lambda}^\ell }.
\end{equation}
\end{definition}
Stated differently, any configuration $w \in \Wcal_{x,y}$ with $x \neq y$ is such that there exist precisely two open paths connecting $x$ and $y$, one of them having colour $1$ and one of them having colour $2$, see also Figure \ref{fig:twoopenpaths}. In Section \ref{sec:relationtwopointfct} we will define a map which changes the colour of one of the open paths and merges them afterwards. 
This will allow us to relate the two-point function to the expected number of closed paths connecting $x$ and $y$. It is important to have two open paths of \textit{different} colour. Indeed, imposing two unpaired links of same colour at $x$ and $y$ does not necessarily imply a connection between $x$ and $y$. 

From now on we fix the dimension $d \in \N$, the number of colours $N \in \N_{\geq 2}$, the size of the system $L \in  2\N$, the parameter $\lambda \in \R^+$,
a tempered and separable potential $v: \Z^d \to \R$ and a good weight function $U: \N_0 \to \R_0^+$. 
For a lighter notation we will omit the dependence from  such quantities  when appropriate.

A first important property of the two-point function is monotonicity.
\begin{proposition}\label{prop:monotonicity}
Let $x \in \mathbb{T}_L$ be an arbitrary vertex such that $n = x \cdot \boldsymbol{e}_i$ is odd and $n \geq 3$ for some $i \in [d]$. Then,
\begin{equation}\label{eq:monotonicity}
\mathbb{G}(o,x)  \leq\mathbb{G}(o, n \boldsymbol{e}_i) \leq \mathbb{G}(o,(n-2)  \boldsymbol{e}_i).
\end{equation}
\end{proposition}
\begin{proof}
Note that, by definition, for any pair of distinct vertices $x, y \in \mathbb{T}_L$, 
the two-point function can be expressed as 
\begin{equation}\label{eq:twopoint}
\mathbb{G}(x,y) = 
\frac{
\mu  \big (  \mathbbm{1}_{\{ g_x=2 \}}  
 \mathbbm{1}_{\{ u^1_x = u^2_x = 1, u_x=2 \}} 
 \mathbbm{1}_{\{ g_y=2 \}}     \mathbbm{1}_{\{ u^1_y=u^2_y = 2, u_y = 2 \}}      \prod_{z \neq x,y} \mathbbm{1}_{\{ u_z = g_z=0 \}}     \big ) }{\mu(\tilde{ \mathcal{W}}_{\mathbb{T}_L }  )  },
\end{equation}
 namely as the expectation with respect to a reflection positive measure of a product of 
 functions, with the domain of each function corresponding to a vertex of the torus. 
 It follows then from the general theorem \cite[Theorem 2.1]{L-T-quantum} that
 a function satisfying such a property satisfies (\ref{eq:monotonicity}). 
\end{proof}
\begin{remark}
\label{rem:ghosts}
The introduction of the ghost pairings 
 is particularly useful to express the two point function as the average of a product of functions,  with the domain of each function corresponding to a vertex of the torus. 
 This allows us to apply \cite[Theorem 2.1]{L-T-quantum} and to deduce the monotonicity property.  
\end{remark}

\begin{figure}
  \centering
\begin{subfigure}{.4\textwidth}
    \includegraphics[width=\textwidth]{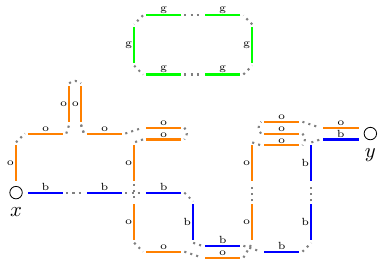}
    \caption{}
\end{subfigure}
\hspace{4em}
\begin{subfigure}{.4\textwidth}
    \includegraphics[width=\textwidth]{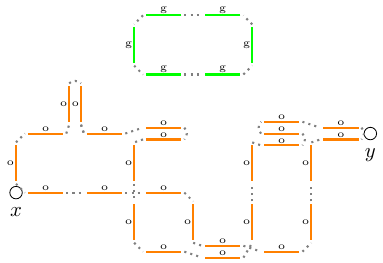}
    \caption{}
\end{subfigure}
\caption{(a) A configuration $w \in  \Wcal_{x,y}$ such that $u_x^1=u_x^2=u_y^1=u_y^2=1$. There exist precisely two open paths connecting $x$ and $y$, one of colour $1:=$ o and one of colour $2:=$ b. There also exists a closed path of colour $3:=$ g. (b) The configuration $F_1(w) \in \tilde{\Wcal}$. The two open paths in $w$ have been merged to a closed path containing $x$ and $y$.}
\label{fig:twoopenpaths}
\end{figure}

\subsection{Reformulation of the two-point function} \label{sec:relationtwopointfct}
In this section we provide a reformulation of the two-point function. 
In Section \ref{sec:lowerboundG(0,e1)} this reformulation will be used to derive a lower bound for $\G(o,\boldsymbol{e}_1)$ and in Section \ref{sec:twopointfctloops} it will be used to derive the relation between the two-point function and the expected number of closed paths touching two vertices.  Recall from \eqref{eq:randompathtilde} that $\tilde{\Wcal}$ denotes the set of configurations of the RPM with only closed paths being present.
Recall also that $\zeta(w)$ denotes the ensemble of cycles of $w \in \tilde{\Wcal}$
and was defined in   \eqref{eq:ensemblecycles}. 
%We note that closed paths were referred to as cycles in Section \ref{sect:equivalence}. 
We let $\zeta^1(w) \subset \zeta(w)$ be the set of $1$-cycles of $\zeta(w)$ and for any $\chi \in \zeta(w)$ and $x \in \T_L$, we denote by $n_x^{(\chi)}$ the number of visits of the cycle $\chi$ at $x$, namely
$$
n_x^{(\chi)} := \frac{1}{2} \, \sum_{y \sim x} \sum_{p=1}^{m_{\{x,y\}}} \mathbbm{1}_{\big\{(\{x,y\},p) \in \chi\big\}}.
$$
In other words $n_x^{(\chi)}$ is the number of links contained in $\chi$ touching $x$ divided by two. 
The next lemma states that the two-point function  $\G(x,y)$ can be expressed as the expectation of a function that involves the number of visits of $1$-cycles at $x$ and $y$. 
Recall the definition of $\mathbb{P}$ and of its expectation $\mathbb{E}$, which were provided in \eqref{eq:RPMprobabilitymeasure}.
\begin{lemma} \label{lemma:relationtwopointfct}
For $x,y \in \T_L$ with $x \neq y$ we have that 
\begin{equation}
\label{eq: twopointfunction2}
\begin{aligned}
\G(x,y) = 2 \, \lambda^2 \, \E\Big(\sum\limits_{\chi \in \zeta^1(w)} n_x^{(\chi)} n_y^{(\chi)} \, \frac{U(n_x+1)}{U(n_x)} \, \frac{U(n_y+1)}{U(n_y)} \, e^{-\tilde{v}_{x,y}(w)} \Big),
\end{aligned}
\end{equation}
where 
\begin{equation}
\tilde{v}_{x,y}(w):=2 \, \big(v_L(o,o)+v_L(x,y)+\sum_{u \in \T_L} (v_L(u,x)+v_L(u,y))n_u(w)\big).
\end{equation}
\end{lemma}
In \eqref{eq: twopointfunction2}, we use the convention that $\frac{0}{0}=1$. Note that, since $U$ is good,  it is  the case that $U(n)=0$ implies that $U(n+1)=0$. 
\begin{proof}
Consider $x,y \in \T_L$, $x \neq y$. Recall from \eqref{eq:randompathtilde} that $\tilde{\Wcal}$ denotes the set of configurations $w \in \Wcal$ in which there exist no unpaired links nor ghost pairings. 
We introduce a map $F_1: \Wcal_{x,y} \to \tilde{\Wcal}$ which acts by changing the colour of all links that are contained in the open path of colour $2$ to colour $1$, by then pairing at $x$ and $y$ the two $1$-links that have an unpaired endpoint and by removing all ghost pairings in the configuration, see also Figure \ref{fig:twoopenpaths}. 
For any $w \in \tilde{\Wcal}$, $F_1^{-1}(w)$ then corresponds to the set of configurations that are obtained from $w$ by adding two ghost pairings at $x$ and $y$, by choosing an arbitrary $1$-cycle $\chi \in \zeta^1(w)$, 
by selecting two arbitrary pairs of links belonging to such a cycle
which are paired  at $x$ and $y$ respectively and unpairing them (thus obtaining two open paths with end-points $x$ and $y$ both),
by changing the colour of the links belonging to one of these two open paths to $2$.
From these considerations we deduce that for any $w \in \tilde{\Wcal}$ we have that
\begin{equation} \label{eq:preimage}
\big| \{w^\prime \in \Wcal_{x,y}: F_1(w^\prime)=w\}\big| = 2 \sum\limits_{\chi \in \zeta^1(w)} n_x^{(\chi)} n_y^{(\chi)}.
\end{equation}

Further, for any $w \in \tilde{\Wcal}$, for any $w^\prime \in \Wcal_{x,y}$ such that $F_1(w^\prime)=w$ and for any $z \in \T_L$,
$$
n_z(w^\prime)=n_z(w)+\mathbbm{1}_{\{z \in \{x,y\}\}}
$$
since at $x$ and $y$ precisely two ghost pairings are removed, and precisely one pairing (of colour $1$) is added. 
In particular,
$$
V(w^\prime)=\sum_{u,z \in \T_L} v(u,z) (n_u(w)+\mathbbm{1}_{\{u \in \{x,y\}\}}) \, (n_z(w)+\mathbbm{1}_{\{z \in \{x,y\}\}}) = V(w)+\tilde{v}_{x,y}(w)
$$
implying that
\begin{equation} \label{eq:weight}
\mu(w^\prime)= \mu(w) \, \frac{U(n_x(w)+1)}{U(n_x(w))}\, \frac{U(n_y(w)+1)}{U(n_y(w))} \, e^{-\tilde{v}_{x,y}(w)}.
\end{equation}
From \eqref{eq:preimage} and \eqref{eq:weight} we obtain that,
\begin{equation} \label{eq:eq1}
\begin{aligned}
\Z(x,y) & = \sum_{w \in \tilde{\Wcal}} \mu\big(\Wcal_{x,y} \cap \{F_1(w^\prime)=w\}\big) \\
& = 2\, \sum_{w \in \tilde{\Wcal}} \mu(w) \sum\limits_{\chi \in \zeta^1(w)} n_x^{(\chi)} n_y^{(\chi)}\, \frac{U(n_x(w)+1)}{U(n_x(w))} \, \frac{U(n_y(w)+1)}{U(n_y(w))} \, e^{-\tilde{v}_{x,y}(w)}.
\end{aligned}
\end{equation}
Multiplying both sides of \eqref{eq:eq1} with $\frac{\lambda^2}{\Z^\ell}$ gives \eqref{eq: twopointfunction2} and the proof is concluded.
\end{proof}

The next lemma provides a further reformulation of the two-point function evaluated at two neighbour vertices,  this value plays an important role in our method.  
%The lemma states that $\G(x,y)$ is equal to the expected product of the number of $1$-links and $2$-links on the edge $\{x,y\}$ for any $\{x,y\} \in \E_L$. 
\begin{lemma}
It holds that 
\begin{equation} \label{eq:twopointneighbour}
\G(o,\boldsymbol{e}_1) = \E\big(m_{\{o,\boldsymbol{e}_1\}}^{(1)} m^{(2)}_{\{o,\boldsymbol{e}_1\}}\big).
\end{equation}
\end{lemma}
\begin{proof}
To begin, we introduce a map $F_2: \Wcal_{o,\boldsymbol{e}_1} \to \tilde{\Wcal}$ which acts by inserting precisely one $1$-link and one $2$-link on the edge $\{o,\boldsymbol{e}_1\}$ such that the $1$-link has the highest and the $2$-link has the second highest position, by pairing at $o$ and $\boldsymbol{e}_1$ the inserted links with the link of same colour that has an unpaired endpoint and by removing all ghost pairings in the configuration. 
For any $w=(m,c,\pi) \in \tilde{\Wcal}$ we have that
\begin{equation} \label{eq:F42}
\big| \{w^\prime \in \Wcal_{o,\boldsymbol{e}_1}: F_2(w^\prime)=w\}\big| = \mathbbm{1}_{\big\{c_{\{o,\boldsymbol{e}_1\}}(m_{\{o,\boldsymbol{e}_1\}})=1, \, c_{\{o,\boldsymbol{e}_1\}}(m_{\{o,\boldsymbol{e}_1\}}-1)=2\big\}}(w),
\end{equation}
where we recall from Section \ref{sect:randompathmodel} that $c_e(p)$ denotes the colour of the $p$-th link on the edge $e$. 
For any $w \in \tilde{\Wcal}$ and for any $w^\prime \in \Wcal_{o,\boldsymbol{e}_1}$ such that $F_2(w^\prime)=w$, it holds that
\begin{equation} \label{eq:F43}
\mu(w^\prime)= \frac{1}{\lambda^2} \, \mu(w) \, \frac{m_{\{o,\boldsymbol{e}_1\}}(w)!}{(m_{\{o,\boldsymbol{e}_1\}}(w)-2)!}.
\end{equation}
We define $\Mcal_{L,N} \subset \N_0^{\E_L \times [N]}$ as the set of elements $\boldsymbol{m}=(m_e^1,\dots,m_e^N)_{e \in \E_L} \in \N_0^{\E_L \times [N]}$ such that $\sum_{y \sim x} m_{\{x,y\}}^i \in 2\N_0$ for all $x \in \T_L$ and $i \in [N]$. Recall from Section \ref{sect:randompathmodel} that $m_e^{(i)}(w)$ denotes the number of $i$-links on $e$ for any $e \in \E_L$, $i \in [N]$, and $w \in \tilde{\Wcal}$. Note that 
\begin{equation} \label{eq:disjointunion}
\tilde{\Wcal} = \bigcup_{\boldsymbol{\tilde{m}} \in \Mcal_{L,N}} \{w \in \tilde{\Wcal} : m_e^{(i)}(w)=\tilde{m}_e^i \, \forall i \in [N] \, \forall e \in \E_L\},
\end{equation}
where the union in \eqref{eq:disjointunion} is disjoint.
From \eqref{eq:F42},  \eqref{eq:F43} and \eqref{eq:disjointunion}, we obtain that
\begin{equation} \label{eq:eq2}
\begin{aligned}
&\G(o,\boldsymbol{e}_1) = \frac{\lambda^2}{\Z^\ell } \, \sum_{w \in \tilde{\Wcal}} \mu\big(\Wcal_{o,\boldsymbol{e}_1} \cap \{F_2(w^\prime)=w\}\big) \\
%& = \sum_{\sigma \in \tilde{\Wcal}} \P(\sigma) \frac{m_{\{o,\boldsymbol{e}_1\}}(\sigma)!}{(m_{\{o,\boldsymbol{e}_1\}}(\sigma)-2)!} \mathbbm{1}_{\big\{c_{\{o,\boldsymbol{e}_1\}}(m_{\{o,\boldsymbol{e}_1\}})=1, \, c_{\{o,\boldsymbol{e}_1\}}(m_{\{o,\boldsymbol{e}_1\}}-1)=2\big\}}(\sigma) \\
& = \E\Big( \frac{m_{\{o,\boldsymbol{e}_1\}}!}{(m_{\{o,\boldsymbol{e}_1\}}-2)!} \mathbbm{1}_{\big\{c_{\{o,\boldsymbol{e}_1\}}(m_{\{o,\boldsymbol{e}_1\}})=1, \, c_{\{o,\boldsymbol{e}_1\}}(m_{\{o,\boldsymbol{e}_1\}}-1)=2\big\}}\Big) \\
& = \sum_{\tilde{\boldsymbol{m}} \in \Mcal_{L,N}} \E\Big( \frac{m_{\{o,\boldsymbol{e}_1\}}!}{(m_{\{o,\boldsymbol{e}_1\}}-2)!} \mathbbm{1}_{\{c_{\{o,\boldsymbol{e}_1\}}(m_{\{o,\boldsymbol{e}_1\}})=1, \, c_{\{o,\boldsymbol{e}_1\}}(m_{\{x,y\}}-1)=2\}} \mathbbm{1}_{\{m_e^{(i)}=\tilde{m}_e^i \, \forall e \in \E_L \forall i \in [N]\}}\Big) \\
& = \sum_{\tilde{\boldsymbol{m}} \in \Mcal_{L,N}} \frac{\tilde{m}_{\{o,\boldsymbol{e}_1\}}!}{(\tilde{m}_{\{o,\boldsymbol{e}_1\}}-2)!} \, \P\big(m_e^{(i)}=\tilde{m}_e^i \, \forall e \in \E_L \forall i \in [N]\big) \\
& \qquad \qquad \qquad \qquad \times \P\big( c_{\{o,\boldsymbol{e}_1\}}(m_{\{o,\boldsymbol{e}_1\}})=1, \, c_{\{o,\boldsymbol{e}_1\}}(m_{\{o,\boldsymbol{e}_1\}}-1)=2 \, | \, m_e^{(i)}=\tilde{m}_e^i \, \forall e \in \E_L \forall i \in [N] \big),
\end{aligned}
\end{equation}
where $\tilde{m}_{\{o,\boldsymbol{e}_1\}}:= \sum_{i=1}^N \tilde{m}_{\{o,\boldsymbol{e}_1\}}^i$. 
For any $\tilde{\boldsymbol{m}} \in \Mcal_{L,N}$, we have that 
\begin{equation} \label{eq:conditionalprob}
\begin{aligned}
& \P\big( c_{\{o,\boldsymbol{e}_1\}}(m_{\{o,\boldsymbol{e}_1\}})=1, \, c_{\{o,\boldsymbol{e}_1\}}(m_{\{o,\boldsymbol{e}_1\}}-1)=2 \, | \, m_e^{(i)}=\tilde{m}_e^i \, \forall e \in \E_L \forall i \in [N]\big) \\
& = \frac{\binom{\tilde{m}_{\{o,\boldsymbol{e}_1\}}-2}{\tilde{m}^1_{\{o,\boldsymbol{e}_1\}}-1,\tilde{m}^2_{\{o,\boldsymbol{e}_1\}}-1,\tilde{m}^3_{\{o,\boldsymbol{e}_1\}},\dots,\tilde{m}^N_{\{o,\boldsymbol{e}_1\}}}}{\binom{\tilde{m}_{\{o,\boldsymbol{e}_1\}}}{\tilde{m}^1_{\{o,\boldsymbol{e}_1\}},\tilde{m}^2_{\{o,\boldsymbol{e}_1\}},\tilde{m}^3_{\{o,\boldsymbol{e}_1\}},\dots,\tilde{m}^N_{\{o,\boldsymbol{e}_1\}}}} = \frac{(\tilde{m}_{\{o,\boldsymbol{e}_1\}}-2)!}{\tilde{m}_{\{o,\boldsymbol{e}_1\}}!} \, \tilde{m}_{\{o,\boldsymbol{e}_1\}}^1 \tilde{m}_{\{o,\boldsymbol{e}_1\}}^2.
\end{aligned}
\end{equation}
For the first equation we used that $\mu(w)=\mu(w^\prime)$ for any $w,w^\prime \in \{w \in \tilde{\Wcal}: \, m_e^{(i)}(w)=\tilde{m}_e^i \, \forall i \in [N] \, \forall e \in \E_L\}$. The fraction between the two multinomial coefficients thus  corresponds to the different numbers of colourings. 
Plugging \eqref{eq:conditionalprob} into \eqref{eq:eq2} and using \eqref{eq:disjointunion}, we obtain \eqref{eq:twopointneighbour} and the proof is concluded.  
\end{proof}

\subsection{Lower bound for $\G(o,\boldsymbol{e}_1)$} \label{sec:lowerboundG(0,e1)}
In this section we show that the quantity
$\mathbb{E}(  m^{ (1)}_{ \{ o, \boldsymbol{e}_1 \}} m^{ (2)}_{ \{o, \boldsymbol{e}_2 \}})$,
which was proved to be equal to the term 
 $\G(o,\boldsymbol{e}_1)$
of the two-point function,  gets arbitrarily large as the range of the weight function, $R$, and $\lambda$ are large.
In particular, we show that $\mathbb{E}(  m^{ (1)}_{ \{ o, \boldsymbol{e}_1 \}} m^{ (2)}_{ \{o, \boldsymbol{e}_2 \}})$ gets larger than   $\mathbb{E}(  m^{ (1)}_{ \{o, \boldsymbol{e}_1\} })$
uniformly in $L$  as $R$ and $\lambda$ are large. 
Our estimate is presented in  Proposition \ref{proposition: uniformpositivity} below. 
We first state two preparatory lemmas. 
Recall that $n_x$ denotes the local time at $x$.
The next lemma states that the local time approaches the range of the weight function as $\lambda$ goes to infinity uniformly in the size of the box.

\begin{lemma} \label{lemma:pairingcardinalityset} 
For any $k <R$, where $R$ is the range of $U$, it holds that,
\begin{equation}
\label{eq: vertexcardinality}
\lim_{\lambda \to \infty} \limsup_{\substack{L \to \infty \\ L \text{ even}}} \, \P\big(n_o \leq k \big) =0.
\end{equation}
% uniformly in L!
\end{lemma}

\begin{proof}
Consider $k <R$ and take $k^\prime >k$ such that $U(k^\prime)>0$. Applying the chessboard estimate, Proposition \ref{prop:chessboardabstract} and Remark \ref{remark:chessboard}, we obtain that
\begin{equation} \label{eq:upperbound1}
\begin{aligned}
\P\big(n_o \leq k \big) & \leq \P\big(\forall\, x\in \T_L ,\, n_x \leq k \big)^{\frac{1}{|\T_L|}} 
%& = \Bigg(\frac{\mu\big(\forall\, x\in \T_L ,\, n_x \leq k \big)}{\mu\big(\Wcal\big)}\Bigg)^{\frac{1}{|\T_L|}} \\
 \leq \Bigg(\frac{\mu^\ell\big(\forall\, x\in \T_L ,\, n_x \leq k \big)}{\mu^\ell\big(\forall\, x\in \T_L ,\, n_x = k^\prime \big)}\Bigg)^{\frac{1}{|\T_L|}}.
%& \leq \Bigg(\frac{\big(\sup_{n \leq k} F(n)(2n-1)!!^N\big)^{|\T_L|} \lambda^{k |\T_L|}(2k+1)^{Nd|\T_L|}}{\lambda^{(k+1) |\T_L|}F(k+1)^{|\T_L|}/(2(k+1))!^{|T_L|/2}}\Bigg)^{\frac{1}{|\T_L|}} \\ % Im Nenner die Konfiguration gefixt, bei der wir auf L^d/2 (jede zweite Kante parallel zu e1) Kanten genau 2k links von Farbe 1 haben, die (2k+1)^NL^d im Zähler kommen von der Anzahl an Konfigurationen.
%& \leq \frac{C(k,N)}{\lambda},
\end{aligned}
\end{equation}
Using similar calculations as in \eqref{eq:finitestep1}, we obtain for the numerator that
\begin{equation} \label{eq:upperbound2}
\begin{aligned}
\mu^\ell\big(\forall\, x\in \T_L ,\, n_x \leq k\big) 
& \leq \sum\limits_{m=(m_e)_{e \in \E_L} \in \N_0^{\E_L}}
\prod_{e\in \E_L}\frac{(cN\lambda)^{m_e}}{m_e!} \mathbbm{1}_{\big\{\sum\limits_{y \sim x} m_{\{x,y\}} \leq 2k \, \forall x \in \T_L \big\}}  \\
& \leq (cN\lambda)^{k|\T_L|} \,(2k+1)^{d |\T_L|}.
\end{aligned}
\end{equation}
We used that there exist no more than $(2k+1)^{d |\T_L|}$ link configurations $m=(m_e)_{e \in \E_L} \in \N_0^{\E_L}$ that satisfy  $\sum_{y \sim x} m_{\{x,y\}} \leq 2k$ for any $x \in \T_L$. 

For the denominator we choose as lower bound the weight of a single configuration and obtain that
\begin{equation}
\label{eq: lowerbound_nominator}
\mu^\ell(\forall\, x\in \T_L ,\, n_x =k^\prime) \geq \frac{\lambda^{k^\prime |\T_L|}}{(2k^\prime)!^{|\T_L|/2}} \, U(k^\prime)^{|\T_L|}. %e^{(k^\prime)^2 \sum_{x,y \in \T_L} -v(x,y)}.
\end{equation}
The right hand-side of \eqref{eq: lowerbound_nominator} corresponds to the weight of a configuration where the $2k^\prime$ links touching a vertex are on precisely one edge each. 
From \eqref{eq:upperbound1}, \eqref{eq:upperbound2} and \eqref{eq: lowerbound_nominator} we deduce that,
$$
\P\big(n_o \leq k \big) \leq \lambda^{k-k^\prime} c(d,k,k^\prime,N,U,v),%e^{-(k^\prime)^2 \sum_{x,y \in \T_L} -v(x,y)},
$$
where the constant $c(d,k,k^\prime,N,U,v)<\infty$ does not depend on $\lambda$ and $L$. This proves \eqref{eq: vertexcardinality}.
\end{proof}

The next lemma states that, as long as the number of links on $\{o, \boldsymbol{e}_1\}$ is sufficiently large compared to the local time at the origin, the probability that $\{o, \boldsymbol{e}_1\}$ is crossed by at least two distinct closed paths is also large.
This is a necessary step for proving that $\mathbb{E}(  m^{ (1)}_{ \{o, \boldsymbol{e}_1 \} } m^{ (2)}_{ \{ o, \boldsymbol{e}_1 \} })$ is large.

\begin{lemma} \label{lemma: sameloop}
For any $k \in \N_{\geq 2}$, $0 < \varepsilon <2$ and any weight function $U$ with range $R \geq \frac{k}{\varepsilon}$ it holds that,
\begin{align}
\P\big(\text{all links on } \{o,\boldsymbol{e}_1\} \text{ belong to the same closed path} \, | \, m_{\{o,\boldsymbol{e}_1\}} \geq n_o \varepsilon \geq k \big) \leq e^{-\varepsilon (\frac{1}{4}-\frac{1}{2k})}.
\end{align}
\end{lemma}
\begin{proof}
Consider $k \in \N_{\geq 2}$ and $0<\varepsilon <2$. Using that for any configuration $w \in \tilde{\Wcal}$ all links on $\{o,\boldsymbol{e}_1\}$ belong to the same closed path only if they have the same colour, we can write 
\begin{equation}
\label{eq: linkssamecolour}
\begin{aligned}
& \P\big(\text{all links on } \{o,\boldsymbol{e}_1\} \text{ belong to the same closed path} \, 
\big | \, m_{\{o,\boldsymbol{e}_1\}} \geq n_o \varepsilon \geq k \big) \\
& = \frac{\mathbb{P}\big(\text{all links on } \{o,\boldsymbol{e}_1\} \text{ belong to the same closed path}, \, m_{\{o,\boldsymbol{e}_1\}} \geq n_o \varepsilon \geq k \big)}{\mathbb{P}\big(m_{\{o,\boldsymbol{e}_1\}} \geq n_o \varepsilon \geq k \big)} \\
& \leq \frac{\sum_{i=1}^N \mathbb{P}\big(\text{all links on } \{o,\boldsymbol{e}_1\} \text{ belong to the same closed path}, \, m_{\{o,\boldsymbol{e}_1\}}^{(i)} \geq n_o \varepsilon \geq k, \, m_{\{o,\boldsymbol{e}_1\}}^{(j)} =0 \, \forall j \neq i \big)}{\sum_{i=1}^N \mathbb{P}\big(m_{\{o,\boldsymbol{e}_1\}}^{(i)} \geq n_o \varepsilon \geq k, \, m_{\{o,\boldsymbol{e}_1\}}^{(j)} =0 \, \forall j \neq i \big)} \\
& = \frac{\mathbb{P}\big(\text{all links on } \{o,\boldsymbol{e}_1\} \text{ belong to the same closed path}, \, m_{\{o,\boldsymbol{e}_1\}}^{(1)} \geq n_o \varepsilon \geq k, \, m_{\{o,\boldsymbol{e}_1\}}^{(j)} =0 \, \forall j \neq 1 \big)}{\mathbb{P}\big(m_{\{o,\boldsymbol{e}_1\}}^{(1)} \geq n_o \varepsilon \geq k, \, m_{\{o,\boldsymbol{e}_1\}}^{(j)} =0 \, \forall j \neq 1 \big)}. 
\end{aligned}
\end{equation}
We now provide an upper bound for the numerator of the last term in \eqref{eq: linkssamecolour}. The idea is to condition on the link and colouring configuration on every edge and on the pairing configuration outside the origin. Given $m \in \Mcal$ and $c \in \Ccal(m)$, we denote by $\Pcal^o(m,c) \subset \Pcal(m,c)$ the set of pairing configurations $\pi=(\pi_x)_{x \in \T_L} \in \Pcal(m,c)$ which are such that each link is paired at both its endpoints, except for the $1$-links touching the origin, which are left unpaired at the origin. We denote by $\Acal_{L,\varepsilon,k}$ the set of triples $(m,c,\pi)$ with $m \in \Mcal$, $c \in \Ccal(m)$ and $\pi \in \Pcal^o(m,c)$, which are such that $m_{\{o,\boldsymbol{e}_1\}} \geq n_o \varepsilon \geq k$, all links on $\{o,\boldsymbol{e}_1\}$ have colour $1$ and $U\big(\frac{1}{2} \sum_{y \sim x} m{\{x,y\}}\big)>0$ for all $x \in \T_L$. 
\begin{figure}
  \centering
    \includegraphics[width=0.3\textwidth]{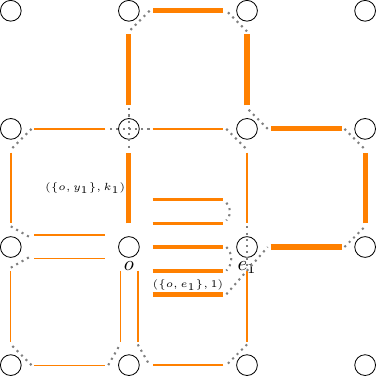}
      \caption{Part of a link and colouring configuration, in which only the $1$-links are depicted. All pairings outside of the origin are fixed. There exist precisely five open paths starting and ending at the origin. The links belonging to the five paths are drawn in different widths.}
\label{fig:Openpaths}
\end{figure}
Using the conditional probability formula, we have that
\begin{equation} \label{eq:conditioned}
\begin{aligned}
& \mathbb{P}(\text{all links on } \{o,\boldsymbol{e}_1\} \text{ belong to the same closed path}, \, m_{\{o,\boldsymbol{e}_1\}}^{(1)} \geq n_o \varepsilon \geq k, \, m_{\{o,\boldsymbol{e}_1\}}^{(j)} =0 \, \forall j \neq 1 \big) \\
& = \sum_{(\tilde{m}, \tilde{c},\tilde{\pi}) \in \Acal_{L,\varepsilon,k}} \mathbb{P}(\text{precisely one closed path on } \{o,\boldsymbol{e}_1\}\, \big | \, m_e(w)=\tilde{m}_e, c_e(w)=\tilde{c}_e \, \forall e \in \E_L, \pi_x(w) = \tilde{\pi}_x \, \forall x \neq o) \\
& \qquad \qquad \qquad \qquad \qquad  \times \mathbb{P}( m_e(w)=\tilde{m}_e, c_e(w)=\tilde{c}_e \, \forall e \in \E_L, \pi_x(w) = \tilde{\pi}_x  \, \forall x \neq o) \\
& = \sum_{(\tilde{m}, \tilde{c},\tilde{\pi}) \in \Acal_{L,\varepsilon,k}} \frac{|P(\tilde{m},\tilde{c},\tilde{\pi})|}{|\{\pi \in \Pcal(\tilde{m},\tilde{c}): \, \pi_x =\tilde{\pi}_x \, \forall x \neq o\}|} \,  \mathbb{P}( m_e(w)=\tilde{m}_e, c_e(w)=\tilde{c}_e \, \forall e \in \E_L, \pi_x(w) = \tilde{\pi}_x  \, \forall x \neq o),
\end{aligned}
\end{equation}
where $|P(\tilde{m},\tilde{c},\tilde{\pi})|$ denotes the number of pairing configurations $\pi \in \Pcal(\tilde{m},\tilde{c})$ which are such that $\pi_x=\tilde{\pi}_x$ for all $x \in \T_L \setminus \{o\}$ and such that all $1$-links on the edge $\{o,\boldsymbol{e}_1\}$ belong to the same closed path. In the last step, we used that $\mathbb{P}(w)=\mathbb{P}(w^\prime)$ for any two configurations $w, w^\prime \in \tilde{\Wcal}$ such that $m_e(w)=m_e(w^\prime)$ and $c_e(w)=c_e(w^\prime)$ for all $e \in \E_L$. 

We now derive an upper bound for $|P(\tilde{m},\tilde{c},\tilde{\pi})|$.  
The pairings of the $1$-links outside of the origin are fixed in such a way that we have a collection of open paths of colour $1$ which start and end at the origin, meaning that precisely two links of each such path have an unpaired endpoint at the origin, see also Figure \ref{fig:Openpaths}. The idea is to pair the paths step-by-step such that we obtain a closed path that contains all links on the edge $\{o,\boldsymbol{e}_1\}$. We begin with the path containing the first link on $\{o,\boldsymbol{e}_1\}$,  $(\{o,\boldsymbol{e}_1\},1)$,  and we denote by $(\{o,y_1\},k_1)$ with $y_1 \sim o$, and $ 1 \leq k_1 \leq \tilde{m}_{\{o,y_1\}}$ the other link of this path which is unpaired at the origin. 
At the origin, we then pair the link $(\{o,\boldsymbol{e}_1\},1)$ to a link $(\{o, y_2\}, k_2)$ which is not $(\{o,y_1\},k_1)$,  since we do not want to close the path. Hence, for the first link on the edge $\{o,\boldsymbol{e}_1\}$, we have $(2 \tilde{n}_o^1-2)$ pairing possibilities, where 
$\tilde{n}_o^1:= \frac{1}{2} \sum_{y \sim o} \sum_{p \in [\tilde{m}_{\{o,y\}}]} \mathbbm{1}_{\{\tilde{c}_{\{o,y\}}(p)=1\}}$ is the number of $1$-links touching $o$ divided by two.
We denote by $(\{o,y_3\},k_3)$ the link corresponding to the same path as $(\{o,y_2\},k_2)$. This link may neither be paired to $(\{o,y_1\}, k_1)$ at the origin implying that we have $(2 \tilde{n}_o^1 -4)$ pairing possibilities for the link $(\{o,y_2\}, k_2)$. We proceed in this manner until we have paired all links on the edge $\{o,\boldsymbol{e}_1\}$. More precisely, the link $(\{o,y_1\}, k_1)$ is first paired at the origin after all links on $\{o,\boldsymbol{e}_1\}$ have been explored. We need at least $\alpha:=\lceil (\tilde{m}_{\{o,\boldsymbol{e}_1\}}-2)/2 \rceil$ pairings at the origin until all 1-links on $\{o,\boldsymbol{e}_1\}$ belong to  a closed path. An upper bound for the number of pairing possibilities of links of colour $1$ at the origin is thus given by 
\begin{equation}
|P(\tilde{m},\tilde{c},\tilde{\pi})| \leq (2\tilde{n}_o^1-2) (2\tilde{n}_o^1-4) \cdots (2\tilde{n}_o^1 - 2\alpha) (2\tilde{n}_o^1 - 2\alpha-1)!!.
\end{equation}
We compare the upper bound for $|P(\tilde{m},\tilde{c},\tilde{\pi})|$ with the number $(2\tilde{n}_o^1-1)!!$ of pairing possibilities of the $1$-links  without the restriction that all links on the edge $\{o,\boldsymbol{e}_1\}$ must belong to the same closed path. We consider the case $\tilde{m}_{\{o,\boldsymbol{e}_1\}}$ even, which is the worst case. Using the estimate $1-x \leq e^{-x}$ for any $x \in \R$, we obtain that
\begin{equation}
\label{eq: comparisonbound}
\begin{aligned}
\frac{|P(\tilde{m},\tilde{c},\tilde{\pi})|}{(2\tilde{n}_o^1-1)!!} & =
\frac{(2\tilde{n}_o^1-2) (2\tilde{n}_o^1-4) \cdots (2\tilde{n}_o^1-\tilde{m}_{\{o,e_1\}}^1+2)}{(2\tilde{n}_o^1-1) (2\tilde{n}_o^1-3) \cdots (2\tilde{n}_o^1-\tilde{m}_{\{o,e_1\}}^1+3)} 
 \leq \exp\bigg(-\sum_{l=1}^{(\tilde{m}_{\{o,\boldsymbol{e}_1\}}-2)/2} \frac{1}{2\tilde{n}_o^1-(2l-1)}\bigg) \\
%& \leq \exp\bigg(- \frac{1}{2\tilde{n}_o^1}\frac{\tilde{m}_{\{o,\boldsymbol{e}_1\}}-2}{2}\bigg) \\
& \leq \exp\bigg(- \frac{1}{2\tilde{n}_o}\frac{\tilde{m}_{\{o,\boldsymbol{e}_1\}}-2}{2}\bigg)
% & = \exp\bigg(- \Big(\frac{\tilde{m}_{\{o,\boldsymbol{e}_1\}}}{4\tilde{n}_o}-\frac{\varepsilon}{2\varepsilon n_0}\Big)\bigg) \\
 \leq \exp\Big(-\frac{\varepsilon}{4}+\frac{\varepsilon}{2k}\Big),
\end{aligned}
\end{equation}
where in the last step we used that $\tilde{m}_{\{o,\boldsymbol{e}_1\}} \geq \varepsilon n_0 \geq k$. 
The proof is concluded by plugging \eqref{eq: comparisonbound} in \eqref{eq:conditioned} and by using \eqref{eq: linkssamecolour}.
\end{proof}

The next proposition states that $\G(o,\boldsymbol{e}_1)$ gets arbitrarily large as long as the range of the weight function and the parameter $\lambda$ are sufficiently large. In particular, $\G(o,\boldsymbol{e}_1)$ gets much larger than the expected number of links
of a given colour on $\{o,\boldsymbol{e}_1\}$.
\begin{proposition} \label{proposition: uniformpositivity}
Let $d \in \N$, $N \in \N_{\geq 2}$ and $v:\Z^d \to \R$ be tempered and separable. For any $M,a \in (0,\infty)$, there exists $R \in (0,\infty)$ such that for any good weight function $U$ with range at least $R$, there exists $\lambda_0>0$ such that for any $\lambda \geq \lambda_0$ and any $L \in 2\N$, %\textcolor{red}{$\lambda_0$ depends on $U$}
\begin{equation} \label{eq:lowerboundG(o,e1)}
\E\big(m_{\{o,\boldsymbol{e}_1\}}^{(1)} m_{\{o,\boldsymbol{e}_1\}}^{(2)}\big) -a \, \E\big(m_{\{o,\boldsymbol{e}_1\}}^{(2)}\big) \geq M.
\end{equation}
\end{proposition}
\begin{proof} Let $d \in \N$, $N \in \N_{\geq 2}$ and $v:\Z^d \to \R$ be tempered and separable and consider $M, a \in (0,\infty)$. 
In the following calculations, $U:\N_0 \to \R_0^+$ will be an arbitrary good weight function and we will later specify its range such that all the events on which we condition have a strictly positive probability.  Consider $k \in \N$ such that $k \geq a$. We have that  
\begin{equation} \label{eq:G(0,e1)}
\begin{aligned}
& \quad \E\big(m_{\{o,\boldsymbol{e}_1\}}^{(1)} m_{\{o,\boldsymbol{e}_1\}}^{(2)}\big) -a \, \E\big(m_{\{o,\boldsymbol{e}_1\}}^{(2)}\big) %= \E\big(m_{\{o,\boldsymbol{e}_1\}}^{(2)} (m_{\{o,\boldsymbol{e}_1\}}^{(1)}-a)\big) \\
= \frac{1}{N-1} \, \E\bigg(\sum_{j=2}^N m_{\{o,\boldsymbol{e}_1\}}^{(j)} (m_{\{o,\boldsymbol{e}_1\}}^{(1)}-a)\bigg) \\
& \geq \frac{k-a}{N-1} \, \P\big(m_{\{o,\boldsymbol{e}_1\}}^{(1)} \geq k, \, \sum_{i=2}^N m_{\{o,\boldsymbol{e}_1\}}^{(i)} \geq 1 \big) \\
& \geq \frac{k-a}{N-1} \, \P\big(m_{\{o,\boldsymbol{e}_1\}}^{(1)} \geq k, \, \sum_{i=2}^N m_{\{o,\boldsymbol{e}_1\}}^{(i)} \geq 1 \, | \,  m_{\{o,\boldsymbol{e}_1\}} \geq \frac{n_o}{d} \geq kN, \text{ not all links on } \{o, \boldsymbol{e}_1\} \text{ of same colour }\big) \\
& \qquad \qquad \qquad \qquad \times \P(m_{\{o,\boldsymbol{e}_1\}} \geq \frac{n_o}{d} \geq kN, \text{ not all links on } \{o, \boldsymbol{e}_1\} \text{ of same colour }) \\
%& \geq \frac{k-a}{N(N-1)} \, \P(m_{\{o,\boldsymbol{e}_1\}} \geq \frac{n_o}{d} \geq kN, \text{ not all links on } \{o, \boldsymbol{e}_1\} \text{ of same colour })\\
& \geq \frac{k-a}{N(N-1)} \, \P(m_{\{o,\boldsymbol{e}_1\}} \geq \frac{n_o}{d} \geq kN) \, \P(\text{not all links on } \{o, \boldsymbol{e}_1\} \text{ of same colour } | \, m_{\{o,\boldsymbol{e}_1\}} \geq \frac{n_o}{d} \geq kN).
\end{aligned}
\end{equation}
We provide a lower bound for the first probability appearing in \eqref{eq:G(0,e1)}. Using the conditional probability formula, we have that,
\begin{equation} \label{eq:firstfactor}
\begin{aligned}
\P\big(m_{\{o,\boldsymbol{e}_1\}} \geq \frac{n_o}{d}  \geq kN \big) 
%& = \sum_{j=kdN}^\infty \P\big(m_{\{o,\boldsymbol{e}_1\}} \geq \frac{j}{d}, \, n_o=j \big) \\
& = \sum_{j=kdN}^\infty \P\big(m_{\{o,\boldsymbol{e}_1\}} \geq \frac{j}{d} \, | \, n_o=j \big) \, \P\big(n_o=j \big) \geq \frac{1}{2d} \, \P\big(n_o \geq kdN \big),
\end{aligned}
\end{equation}
where we used the rotational symmetry of the torus in the last step. 
For the second probability in \eqref{eq:G(0,e1)} we apply Lemma \ref{lemma: sameloop} with $\varepsilon=\frac{1}{d}$ and obtain that,
\begin{equation} \label{eq:secondfactor}
\begin{aligned}
& \P(\text{not all links on } \{o, \boldsymbol{e}_1\} \text{ of same colour } | \, m_{\{o,\boldsymbol{e}_1\}} \geq \frac{n_o}{d} \geq kN) \\
& \geq \frac{N-1}{N} \, \P\big(\text{at least two closed paths at } \{o,\boldsymbol{e}_1\} \, | \, m_{\{o,\boldsymbol{e}_1\}} \geq \frac{n_o}{d} \geq kN \big) \\
& \geq \frac{N-1}{N} \big(1- e^{-\frac{1}{d}(\frac{1}{4}-\frac{1}{2kN})} \big) \geq \frac{1}{8d} \frac{N-1}{N},
\end{aligned}
\end{equation}
where in the last step we used that $e^{-x} \leq 1-\frac{1}{2}x$ for any $x \in (0,\frac{1}{3})$. Plugging the lower bounds \eqref{eq:firstfactor} and \eqref{eq:secondfactor} in  \eqref{eq:G(0,e1)} and applying Lemma \ref{lemma:pairingcardinalityset} we obtain that for any weight function $U$ with range $R \geq kdN$, where $k:=16d^2N^2M+a+1$, there exists $\lambda_0=\lambda_0(d,N,U,v, M,a)<\infty$, such that for any $\lambda \geq \lambda_0$ and for any $L \in 2\N$, 
$$
\begin{aligned}
\E\big(m_{\{o,\boldsymbol{e}_1\}}^{(1)} m_{\{o,\boldsymbol{e}_1\}}^{(2)}\big) -a \, \E\big(m_{\{o,\boldsymbol{e}_1\}}^{(2)}\big)  \geq \frac{k-a}{16d^2 N^2} \, \P\big(n_o \geq kd \big) \geq M.
\end{aligned}
$$
This concludes the proof.
\end{proof}

\subsection{Two-point function and expected number of closed paths connecting two vertices} \label{sec:twopointfctloops}
In this section we derive the important relation between the two-point function $\G(x,y)$ defined in \eqref{eq: twopointfunction} and the expected number of closed paths connecting $x$ and $y$ using  reformulation \eqref{eq: twopointfunction2} of the two-point function. 
Recall from the definition in \eqref{eq:tildeNcal} that we denote by $\tilde{\Ncal}_{x,y}: \tilde{\Wcal} \to \R$ the function which counts the number of closed paths (cycles) visiting $x$ and $y$. 

\begin{proposition} \label{lemma:expectednumberloopsrelation}
There exists $c > 0$ such that for all $L \in 2\N$ and for all $x,y \in \T_L$ with $x \neq y$, 
\begin{equation} \label{eq: relationnumberloops}
\E\big(\tilde{\Ncal}_{x,y}\big) \geq %\frac{\lambda^4}{4 M^2 e^{2\lambda C N d}} 
c \, \G(x,y)^4.
\end{equation}
\end{proposition}
\begin{proof} Consider $x,y \in \T_L$ with $x \neq y$. 
By the reformulation \eqref{eq: twopointfunction2} of the two-point function, we have that
\begin{equation} \label{eq:twopointandloops}
\begin{aligned}
\G(x,y)  
& = \frac{2 \, \lambda^2}{N} \, \E\Big(\sum\limits_{\chi \in \zeta(w)} n_x^{(\chi)} n_y^{(\chi)}\, \frac{U(n_x+1)}{U(n_x)} \, \frac{U(n_y+1)}{U(n_y)} \, e^{-\tilde{v}_{x,y}}\Big)  \leq \frac{2\lambda^2 M^2}{N} \, \E\Big(\tilde{\Ncal}_{x,y} ^{\frac{1}{4}} \, e^{2 \, \bar{v}\, (n_x+n_y)} \,  e^{-\tilde{v}_{x,y}}\Big),
\end{aligned}
\end{equation}
where we used the upper bound $\sum_{\chi \in \zeta(w)}n_x^{(\chi )}n_y^{(\chi )} \leq \tilde{\Ncal}_{x,y}^{\frac{1}{4}} \, n_x \, n_y$ for any $w \in \tilde{\Wcal}$ and then applied the definition  in \eqref{eq:rapidlydecayingdef} of a good weight function $U$. Using now twice the Cauchy-Schwarz inequality, we obtain from \eqref{eq:twopointandloops} that 
$$
\begin{aligned}
\G(x,y) & \leq \frac{2\lambda^2 M^2}{N} \, \E\big(\tilde{\Ncal}_{x,y}\big)^{\frac{1}{4}} \, \E\big(e^{8 \bar{v}(n_x+n_y)}\big)^{\frac{1}{4}} \, \E\big(e^{-2\tilde{v}_{x,y}(w)}\big)^{\frac{1}{2}}.
\end{aligned}
$$
For the proof of \eqref{eq: relationnumberloops} it only remains to show that the two expectations $\E\big(e^{8 \, \bar{v} \, (n_x+n_y)}\big)$ and $\E(e^{-2\tilde{v}_{x,y}(w)})$ are finite uniformly in $L \in 2\N$. 
The existence of a finite $c$ such that for any $L \in 2 \mathbb{N}$,
 $$
 \E\big(e^{8 \bar{v}(n_x+n_y)}\big) \leq c,
 $$
follows from Lemma \ref{lemma:Chessboardapplicationno} and  Proposition \ref{prop:chessboardabstract}. We now derive a uniform upper bound for $\E(e^{-2\tilde{v}_{x,y}(w)})$. For any $x,y \in \T_L$, we have that
\begin{equation}
\label{eq:identitytwopointfcts}
\mu^\ell(e^{-2\tilde{v}_{x,y}(w)}) = e^{4\big(v_L(o,o)+v_L(x,y)\big)} \, \mu\bigg(\frac{U(n_x-2)}{U(n_x)} \, \frac{U(n_y-2)}{U(n_y)} \, \mathbbm{1}_{\{g_x=2=g_y\}} \mathbbm{1}_{\{g_z=0 \, \forall z \in \T_L \setminus \{x,y\} \}} \mathbbm{1}_{\{u_z=0\, \forall z \in \T_L\}}\bigg),
\end{equation}
where we use the convention $\frac{0}{0}=1$. The identity \eqref{eq:identitytwopointfcts} can be derived by defining a map on $\tilde{\Wcal}$ that adds precisely two ghost pairings at $x$ and at $y$.  
From \eqref{eq:identitytwopointfcts}, Proposition \ref{prop:chessboardabstract} and Remark \ref{remark:chessboard} we obtain that
\begin{equation} \label{eq:finitestep3}
\E(e^{-2\tilde{v}_{x,y}(w)}) \leq e^{4 \alpha} \, \bigg(\frac{1}{\Z^\ell} \, \mu\Big( \prod_{x \in \T_L} \frac{U(n_x-2)}{U(n_x)} \, \mathbbm{1}_{\{g_x=2, u_x=0\}} \Big)\bigg)^{\frac{2}{L^d}},
\end{equation}
where $\alpha:= \sum_{x \in \Z^d} |v(x)| < \infty$ since $v$ is tempered. To prove that the term on the right-hand side of \eqref{eq:finitestep3} is finite, we use that $\Z^\ell \geq 1$ and we  follow the same steps as in the proof of Lemma \ref{lemma: finitepartitionfunction}. The only difference here is that in the interaction term $V(w)$ at each vertex precisely two ghost pairings are present (in addition to the pairings of the links). However, as can be seen in \eqref{eq:finitestep1}, the presence of such two ghost pairings at each vertex decreases the weight of the whole expression. 
From these considerations and from \eqref{eq:finitestep3} and \eqref{eq:finitestep1} we thus deduce that there exists $c \in (0,\infty)$ such that for any $L \in 2\N$ it holds that, 
$
\E(e^{-2\tilde{v}_{x,y}(w)}) \leq c.
$ 
This concludes the proof of the proposition.
\end{proof}

\section{Derivation of the Key Inequality}
\label{sect:derivationofthekeyinequality}
The main goal of this section is to derive Theorem \ref{theo: keyinequality} below, which states a Key Inequality for the two-point function defined in \eqref{eq: twopointfunction}. The derivation of such a Key Inequality involves a central quantity, defined in \eqref{eq:Zh} below, and the chessboard estimate, which is given in Section \ref{sect:Chessboardscheme}. 

Throughout the section,  we fix parameters $d>2$, $N \in \N_{\geq 2}$, $\lambda \in \R^+$, $L \in 2\N$ and functions $v:\Z^d \rightarrow \R$ and $U:\N_0 \rightarrow \R_0^+$ such that $v$ is tempered and separable and such that $U$ is good and has range $R \geq 2$. Recall from Section \ref{sect:reflectionpositivityandchessboardestimate} that $(\T_L,\E_L)$ and $(\Tcal_L,\Ecal_L)$ refer to the original and to the extended torus. We write the sub-script $\Tcal_L$ or $\T_L$ when considering the RPM in the extended or in the original torus and we omit all the remaining sub-scripts for lighter notation.  

We now introduce the \textit{central quantity} $\mathscr{Z}( \boldsymbol{h})$ for a real-valued vector $\boldsymbol{h}=(h_x)_{x \in \Tcal_L}$. 
In Section \ref{sec:polynomialexpansion} we will expand $\mathscr{Z}(\varphi \boldsymbol{h})$, with $\varphi \in \R$, as a polynomial in $\varphi$. When considering its limit $\varphi \to 0$, we will see its relation to the partition functions defined right after Definition \ref{def:measure} and in \eqref{eq: definition_partition_function}. 
\paragraph{Central quantity.} 
We denote by $\Wcal_{\Tcal_L}^{vert} \subset \Wcal_{\Tcal_L}$ the set of configurations $w \in \Wcal_{\Tcal_L}$ such that the following properties hold at the same time:
\begin{enumerate}[(i)]
   \setlength{\itemsep}{0pt}
    \setlength{\topsep}{0pt}
\item The number of $1$-links on $\{x,y\}$ that are unpaired at $x$ is equal to the number of $2$-links on $\{x,y\}$ that are unpaired at $x$ for any $\{x,y\} \in \Ecal_L$, 
\item There exists no link of colour $i \geq 3$ that is unpaired, 
\item There exists no ghost pairing,
\item At virtual vertices all links are unpaired. 
\end{enumerate}
Condition (iv) implies that any closed path of $w$ lies entirely in the original torus. Recall from Section \ref{sect:randompathmodel} that $u_x(w)$ denotes the number of links touching $x$ that are unpaired at $x$ for any $x \in \mathcal{T}_L$. Given $w \in \Wcal_{\Tcal_L}$, we denote by $\alpha_x(w)$ the number of links on the edge $\{x,x+\boldsymbol{e}_{d+1}\}$ that are unpaired at $x \in \T_L$.

\begin{figure}
  \centering
\begin{subfigure}{.4\textwidth}
    \includegraphics[width=\textwidth]{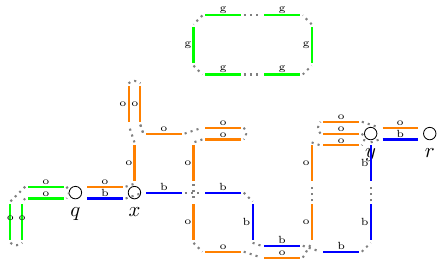}
    \caption{}
\end{subfigure}
\hspace{4em}
\begin{subfigure}{.4\textwidth}
    \includegraphics[width=\textwidth]{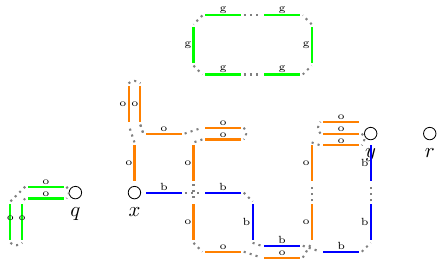}
    \caption{}
\end{subfigure}
\caption{(a) A configuration $w \in  \Acal((x,q),(y,r)) \subset \Wcal_{\Tcal_L}^{vert}$
 (we assume for graphical reasons that all the links lie on the original torus).
There exist precisely two open paths connecting $x$ and $y$, one of colour $1:=$ o and one of colour $2:=$ b. There also exist two closed paths of colour $3:=$ g. (b) The configuration $F_5(w) \in \Wcal_{x,y}$. The extremal links on $\{x,q\}$ and $\{y,r\}$ have been removed.}
\label{fig:twoopenpaths2}
\end{figure}
\begin{definition}[Central quantity]\label{def:Zh}
For any vector of real numbers $\boldsymbol{h} = (h_x)_{x \in \mathcal{T}_L}$, we define
\begin{equation}\label{eq:Zh}
\mathscr{Z}( \boldsymbol{h} ) : = \mu_{\Tcal_L}\Big( \prod_{x \in \mathcal{T}_L} 
h_x^{\frac{u_x}{2}} \prod_{x \in \T_L} \Big(\frac{1}{2}\Big)^{\frac{\alpha_x}{2}} \mathbbm{1}_{\Wcal_{\Tcal_L}^{\text{vert}}}
\Big  ).
\end{equation}
\end{definition}
In other words,   in  \eqref{eq:Zh}  a multiplicative factor $h_x$ is assigned to each pair of unpaired links of different colour at $x$,
while  a multiplicative factor $\frac{1}{\sqrt{2}}$ 
is assigned
to each link on an edge touching the virtual torus which is unpaired at both its end-points.

\subsection{Polynomial expansion} \label{sec:polynomialexpansion}

For any vector of real numbers $\boldsymbol{h}=(h_x)_{x \in \Tcal_L}$, we define the discrete Laplacian of $\boldsymbol{h}$ on the extended torus,
$$
(\triangle^*h)_x := \sum_{\substack{y \in \Tcal_L: \\ \{x,y\} \in \Ecal_L}} h_y.
$$

\begin{proposition}[Polynomial expansion]
\label{prop: central quantity expansion}
For any vector of real numbers $\boldsymbol{h}=(h_x)_{x \in \Tcal_L}$ and any $\varphi \in \R$, we have that,
\begin{equation}
\mathscr{Z}(\varphi \boldsymbol{h})= \Z_{\T_L}^\ell + \varphi^2 Z^{(2)}(\boldsymbol{h})+o(\varphi^2),
\end{equation}
in the limit $\varphi \to 0$, where 
\begin{equation}
\label{eq: termorderphi2}
\begin{aligned}
Z^{(2)}(\boldsymbol{h}) & := \lambda^2 \, \Z_{\T_L}^\ell \, \Big(\sum_{\{x,y\} \in \E_L} h_x h_y + \frac{1}{2} \sum_{x \in \T_L} h_x h_{x+\boldsymbol{e}_{d+1}}\Big) +2 \, \lambda^2 \mu_{\T_L}^\ell\big(m_{\{o,\boldsymbol{e}_1\}}^{(1)}\big)\sum_{\{x,y\} \in \E_L} h_x h_y  \\
& \qquad \qquad  + \frac{\lambda^4}{2} \sum_{x,y \in \T_L} \Z_{\T_L}(x,y) (\triangle^*h)_x(\triangle^*h)_y - \frac{\lambda^4}{4} \sum_{x \in \T_L} \Z_{\T_L}(x,x) \sum_{\substack{q \in \Tcal_L: \\ \{x,q\} \in \Ecal_L}} h_q^2.
\end{aligned}
\end{equation}
\end{proposition}
Reflection positivity will allow us to derive an upper bound on the term \eqref{eq: termorderphi2} for a particular choice of $\boldsymbol{h}$, which will then lead to the Key Inequality \eqref{eq: Key Inequality}. The remainder of the current subsection is devoted to the proof of Proposition \ref{prop: central quantity expansion}, Section \ref{sect:proofinequailty} can be read independently from what follows below in the current subsection. 

From now on we will call a link \textit{extremal} if at least one of its endpoints is unpaired.

The term $Z^{(2)}(\boldsymbol{h})$ of order two, defined in \eqref{eq: termorderphi2}, is the contribution to $\mathscr{Z}(\boldsymbol{h})$ of all configurations $w \in \Wcal_{\Tcal_L}^{vert}$ which are such that $\sum_{x \in \Tcal_L} u_x(w) =4$. By definition of $\Wcal_{\Tcal_L}^{vert}$ a total number of four unpaired endpoints of links can exist in three different situations depending on the number of extremal links, see also Figure \ref{fig:extremallinks}. The first one is that there exist precisely two extremal links which are unpaired at both its endpoints. Such two extremal links have to be on the same edge. The second situation is that there exist precisely three extremal links, one of them is unpaired at both its endpoints and the other two are unpaired at precisely one of its endpoints. Also such three extremal links have to be on the same edge. The third and last situation is that there exist precisely four extremal links, which are all unpaired at precisely one of its endpoints. Such four extremal links are either all on the same edge or there exist two edges with precisely two extremal links each.

%All the definitions below are functional to the proof of Proposition \ref{prop: central quantity expansion}. Subsection \ref{sect:proofinequailty}, which contains the proof of Theorem \ref{theo: keyinequality}, can be read independently from  what follows below in the current subsection.

\begin{figure}
\centering
\begin{subfigure}{.25\textwidth}
  \centering
  \includegraphics[width=\textwidth]{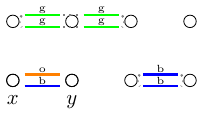}
  \caption{}
  %\label{fig:right}
\end{subfigure}
\hspace{3em}
\begin{subfigure}{.25\textwidth}
  \centering
  \includegraphics[width=\textwidth]{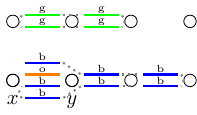}
  \caption{}
  %\label{fig:left}
\end{subfigure}
\hspace{3em}
\begin{subfigure}{.25\textwidth}
  \centering
  \includegraphics[width=\textwidth]{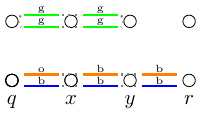}
  \caption{}
  %\label{fig:left}
\end{subfigure}
\caption{(a) A configuration $w \in \Acal(\{x,y\})$. (b) A configuration $w \in \tilde{\Acal}^1(\{x,y\})$. (c) A configuration $w \in \Acal((x,q),(y,r))$. The colours are defined by $1:=b$, $2:=o$ and $3:=g$. In all three configurations there exist precisely four unpaired endpoints of links.} 
    \label{fig:extremallinks}
\end{figure}

\paragraph{Subsets of $\Wcal_{\Tcal_L}^{vert}$.} 
We now define several subsets of $\Wcal_{\Tcal_L}^{vert} \subset \Wcal_{\Tcal_L}$ according to the considerations above. A glance at Figure \ref{fig:extremallinks} might be helpful. Recall that $\{x,y\} \in \Ecal_L$ represents an undirected edge. In the following definition, with slight abuse of notation, we represent by $(x,y) \in \Ecal_L$ an edge directed from $x$ to $y$. 
\begin{enumerate}[(i)]
\item For any $\{x,y\} \in \Ecal_L$, we let $\Acal(\{x,y\})$ be the set of realisations $w \in \Wcal_{\Tcal_L}^{vert}$ with precisely two extremal links, such extremal links are on the edge $\{x,y\}$, one of them having colour $1$ and one of them having colour $2$ and both extremal links are unpaired at both its endpoints. It follows from this definition that in any $w \in \Acal(\{x,y\})$, there exist only closed paths except for one open path of colour $1$ and one open path of colour $2$ which both have length $1$.
\item For any $\{x,y\} \in \E_L$, we let $\tilde{\Acal}^1(\{x,y\})$ be the set of realisations $w \in \Wcal_{\Tcal_L}^{vert}$ with precisely three extremal links, such extremal links are all on the edge $\{x,y\}$, two of them being $1$-links and one being a $2$-link. One of the extremal $1$-links is paired at $x$ and unpaired at $y$ and the other one is paired at $y$ and unpaired at $x$. The extremal $2$-link is unpaired at both its endpoints. Any configuration $w \in \tilde{\Acal}^1(\{x,y\})$ is such that there exist only closed paths except for one open path of colour $2$ which has length $1$ and one open path of colour $1$ which has length $>1$. Let $\tilde{\Acal}^2(\{x,y\})$ be defined as $\tilde{\Acal}^1(\{x,y\})$, but with the properties of the extremal $1$- and $2$-links on $\{x,y\}$ interchanged.
%Similarly, we let $\tilde{\Acal}^2(\{x,y\})$ be the set of realisations $w \in \Wcal^{vert}$ with precisely three extremal links, such extremal links are all on the edge $\{x,y\}$, two of them being $2$-links and one being a $1$-link. One of the extremal $2$-links is paired at $x$ and unpaired at $y$ and the other one is paired at $y$ and unpaired at $x$. The extremal $1$-link is unpaired at both its endpoints. 
\item For any pair of (directed) distinct edges $(x,q)$, $(y,r) \in \Ecal_L$, let $\Acal((x,q),(y,r))$ be the set of realisations $w \in \Wcal_{\Tcal_L}^{vert}$ with precisely four extremal links, two of them on $\{x,q\}$ and two of them on $\{y,r\}$. One of the extremal links on $\{x,q\}$ and $\{y,r\}$ is a $1$-link and the other one is a $2$-link. The extremal links on $\{x,q\}$ are both paired at $x$ and unpaired at $q$ and the extremal links on $\{y,r\}$ are both paired at $y$ and unpaired at $r$. It follows from this definition that any configuration $w \in \Acal((x,q),(y,r))$ is such that there exist only closed paths except for two open paths, one of them having colour $1$ and one of them having colour $2$, which both have length $>1$. Further, we let $\Acal((x,q),(x,q))$ be the set of realisations $w \in \Wcal_{\Tcal_L}^{vert}$ with precisely four extremal links, such extremal links are all on $\{x,q\}$, two of them being $1$-links and two being $2$-links and all extremal links are paired at $x$, but unpaired at $q$. 
\end{enumerate}

%The next lemma provides explicit formulas for the measure of the events defined above.
\begin{lemma} \label{lemma: polynomial expansion}
For any $(x,q),(y,r), \{u,b\} \in \Ecal_L$, any $\{z,w\} \in \E_L$ and any $i \in \{1,2\}$, we have that
\begin{align}
\label{eq: polynomialexpansion1}
\mu_{\Tcal_L}\big(\Acal(\{u,b\})\big)  & = \lambda^2 \, \Z^\ell_{\T_L}, \\
\label{eq: polynomialexpansion3}
\mu_{\Tcal_L}\big(\tilde{\Acal}^i(\{z,w\})\big) 
& = \lambda^2 \mu_{\T_L}^\ell(m_{\{z,w\}}^{(i)}),\\
\label{eq: polynomialexpansion2}
\mu_{\Tcal_L}\big(\Acal((x,q),(y,r))\big) 
& = \begin{cases} \lambda^4 \, \Z_{\T_L}(x,y) & \text{if } (x,q) \neq (y,r), \\
 \frac{\lambda^4}{4} \, \Z_{\T_L}(x,x) & \text{if } (x,q) = (y,r). \end{cases} 
\end{align}
\end{lemma}
\begin{proof}
We begin with the proof of \eqref{eq: polynomialexpansion1}. Fix an arbitrary undirected edge $\{u,b\} \in \Ecal_L$. Recall the definition of the set $\tilde{\Wcal}_{\Tcal_L}^o$ in \eqref{eq:extendedtorusclosedpaths}. We define a map $F_3: \Acal(\{u,b\})  \to \tilde{\Wcal}_{\Tcal_L}^o$, which acts by removing the two extremal links on the edge $\{u,b\}$. For any $w \in \tilde{\Wcal}_{\Tcal_L}^o$, $F_3^{-1}(w)$ corresponds to the set of configurations which are obtained from $w$ by inserting a $1$-link and a $2$-link on $\{u,b\}$ and by leaving the links unpaired at both its endpoints. 
We thus have that
\begin{equation} \label{eq: firstcase1}
\big| F_3^{-1}(w)\big| = 2 \, \binom{m_{\{u,b\}}(w)+2}{2},
\end{equation}
which corresponds to the number of ways a $1$-link and a $2$-link can be inserted on the edge $\{u,b\}$ among the ones already present in $w$. 

For any $\tilde{w} \in \tilde{\Wcal}_{\Tcal_L}^o$, for any $w \in \Acal(\{u,b\})$ such that $F_3(w)=\tilde{w}$, we have that
\begin{equation} \label{eq: firstcase2}
\mu_{\Tcal_L}(w)= \lambda^2 \mu_{\Tcal_L}(\tilde{w}) \frac{m_{\{u,b\}}(\tilde{w})!}{(m_{\{u,b\}}(\tilde{w})+2)!}.
\end{equation}
From \eqref{eq: firstcase1} and \eqref{eq: firstcase2}, we deduce that 
$$
\begin{aligned}
\mu_{\Tcal_L}\big(\Acal(\{u,b\})\big) 
& = \sum_{\tilde{w} \in \tilde{\Wcal}_{\Tcal_L}^o} \mu_{\Tcal_L}\big(\Acal(\{u,b\})\cap \{F_3(w)=\tilde{w}\}\big)  = \lambda^2 \sum_{\tilde{w} \in \tilde{\Wcal}_{\Tcal_L}^o} \mu_{\Tcal_L}(\tilde{w})  = \lambda^2 \, \Z_{\T_L}^\ell.
\end{aligned}
$$
This concludes the proof of \eqref{eq: polynomialexpansion1}.

We continue with the proof of \eqref{eq: polynomialexpansion3}. Fix an arbitrary edge $\{x, y\} \in \E_L$. We define a map $F_4: \tilde{\Acal}^1(\{x,y\}) \to \tilde{\Wcal}_{\Tcal_L}^o$, which acts by removing the extremal $2$-link and the extremal $1$-link on $\{x,y\}$ that is unpaired at $x$ and by pairing at $y$ the remaining extremal $1$-link on $\{x,y\}$ with the $1$-link that has received an unpaired endpoint. For any $w \in \tilde{\Wcal}_{\Tcal_L}^o$, $F_4^{-1}(w)$ corresponds to the set of configurations which are obtained from $w$ 
\begin{enumerate}[(i)]
   \setlength{\itemsep}{0pt}
    \setlength{\topsep}{0pt}
\item by inserting one $2$-link on $\{x,y\}$ and by leaving the inserted link unpaired at both its endpoints,
\item by choosing one $1$-link, $(\{x,y\},p)$ with $p \in [m_{\{x,y\}}]$, on $\{x,y\}$ and by removing its pairing at $y$,
\item by inserting one $1$-link on $\{x,y\}$, by leaving the inserted link unpaired at $x$ and by pairing the inserted link at $y$ to the link to which the link $(\{x,y\},p)$ has been paired before the removal of its pairing at $y$.
\end{enumerate}
For any $w \in \tilde{\Wcal}_{\Tcal_L}^o$, we thus have that
\begin{equation} \label{eq:cardinalityploy2}
|F_4^{-1}(w)|= \binom{m_{\{x,y\}}(w)+2}{2} \, 2 \, m_{\{x,y\}}^{(1)}(w),
\end{equation}
where the binomial coefficient correspond to the number of ways two links can be inserted on $\{x,y\}$ among the ones already present in $w$. The factor $2$ accounts for the number of possibilities of choosing which of the inserted links has colour $1$ and which one has colour $2$. The factor $m_{\{x,y\}}^{(1)}$ corresponds to the number of possibilities of choosing a $1$-link on $\{x,y\}$ whose pairing gets removed at $y$. For any $w \in \tilde{\Wcal}_{\Tcal_L}^o$ and any $\tilde{w} \in \Acal^1(\{x,y\})$ such that $F_4(\tilde{w})=w$, we have that
\begin{equation} \label{eq:measurepoly2}
\mu_{\Tcal_L}(\tilde{w})= \lambda^2 \mu_{\Tcal_L}(w) \frac{m_{\{x,y\}}(w)!}{(m_{\{x,y\}}(w)+2)!}.
\end{equation}
Combining \eqref{eq:cardinalityploy2} and \eqref{eq:measurepoly2} we thus obtain that
$$
\mu_{\Tcal_L}\big(\Acal^1(\{x,y\})\big)= \sum_{w \in \tilde{\Wcal}_{\Tcal_L}^o} \mu_{\Tcal_L}\big(\Acal^1(\{x,y\} \cap \{F_4(\tilde{w})=w\}\big) = \lambda^2 \mu_{\T_L}^\ell\big(m_{\{x,y\}}^{(1)}\big).
$$
This concludes the proof of \eqref{eq: polynomialexpansion3}.

We continue with the proof of \eqref{eq: polynomialexpansion2}. Consider $x,y \in \T_L$ (possibly $x=y$) and $q,r \in \Tcal_L$ such that $q \sim x$ and $r \sim y$. We denote by $\Wcal_{x,y}^o$ the extension of $\Wcal_{x,y}$, defined in \eqref{eq:Wcalxy} and \eqref{eq:Wcalxx}, to the extended torus by replacing $\T_L$ by $\Tcal_L$ and by adding the condition that $n_z=0$ for any $z \in \Tcal_L \setminus \T_L$. 
We define a map $F_5: \Acal((x,q),(y,r)) \to \Wcal_{x,y}^o$, which acts by removing all extremal links in the configuration, by leaving unpaired the links to which these removed links were paired and by adding $2$ ghost pairings at $x$ and $y$ in case $x \neq y$ and by adding $4$ ghost pairings at $x$ in case $x=y$, see also Figure \ref{fig:twoopenpaths2}. For any $w \in \Wcal_{x,y}^o$, $F_5^{-1}(w)$ corresponds to the set of configurations which are obtained from $w$ 
\begin{enumerate}[(i)]
   \setlength{\itemsep}{0pt}
    \setlength{\topsep}{0pt}
\item by removing all ghost pairings in the configuration,
\item by inserting $\frac{g_x}{2}$ $1$-links and $\frac{g_x}{2}$ $2$-links on $\{x,q\}$ and by inserting $\frac{g_y}{2}$ $1$-links and $\frac{g_y}{2}$ $2$-links on $\{y,r\}$, where $g_x=g_y=2$ in case that $x \neq y$ and $g_x=4$ in case $x=y$,
\item by leaving the inserted links unpaired at $q$ and $r$ and by pairing them at $x$ and $y$ to the links of same colour that have an unpaired endpoint in some arbitrary manner.
\end{enumerate} 
For any $w \in \Wcal_{x,y}^o$, we thus have that,
\begin{equation}
\label{eq:cardinalities_set_four_links}
\begin{aligned}
\big|F_5^{-1}(w)\big|  = \begin{cases}
4\,  \binom{m_{\{x,q\}}(w)+2}{2} \, \binom{m_{\{y,r\}}(w)+2}{2} & \text{if } x \neq y \text{ and } \{x,q\} \neq \{y,r\}, \\
4 \, \binom{m_{\{x,q\}}(w)+4}{4} \, \binom{4}{2} & \text{if } (x,q) = (r,y), \\
 2^{\mathbbm{1}_{\{u_x^1=2\}}+\mathbbm{1}_{\{u_x^2=2\}}} \, 4 \, \binom{m_{\{x,q\}}(w)+2}{2} \, \binom{m_{\{x,r\}}(w)+2}{2} & \text{if } x=y \text{ and } q \neq r, \\
 2^{\mathbbm{1}_{\{u_x^1=2\}}+\mathbbm{1}_{\{u_x^2=2\}}} \, \binom{m_{\{x,q\}}(w)+4}{4} \binom{4}{2}& \text{if } (x,q) = (y,r).
\end{cases} 
\end{aligned}
\end{equation}
We now explain \eqref{eq:cardinalities_set_four_links}. In all four cases, the binomial coefficients involving the number of links on the edges correspond to the number of ways the extremal links can be inserted on $\{x,q\}$ and $\{y,r\}$ among the ones already present in $w$. If $\{x,q\} \neq \{y,r\}$ we then have precisely two possibilities on each edge to colour the inserted links such that on each edge one of the inserted links has colour $1$ and the other one has colour $2$ giving us a factor of $4$. If $\{x,q\}=\{y,r\}$ we have $\binom{4}{2}$ possibilities of colouring the inserted links such that two of them have colour $1$ and two of them have colour $2$.  The remaining factors in \eqref{eq:cardinalities_set_four_links} correspond to the number of pairing possibilities of the inserted links. If $x \neq y$ and $\{x,q\} \neq \{y,r\}$, we have precisely one possibility of pairing the inserted links at $x$ and $y$ with the extremal links of same colour. If $(x,q)=(r,y)$ we have two possibilities of choosing which of the inserted $1$-links is unpaired at $q$ and two possibilities of choosing which of the inserted $2$-links is unpaired at $q$. After having chosen these links, there exists precisely one pairing possibility, thus giving us a factor of $4$. 
If $x=y$, we have to distinguish between three situations. If $u_x(w)=0$, then the inserted $1$-links and $2$-links have to be paired with each other at $x$. This implies that there exists precisely one possibility of pairing the inserted links. If $u_x^1=2$ and $u_x^2=0$, then the inserted $2$-links are paired with each other at $x$, but the inserted $1$-links are not paired with each other at $x$. We thus have precisely one possibility of pairing the inserted $2$-links and two possibilities of pairing the inserted $1$-links with the extremal $1$-links at $x$. The same considerations hold true in the cases $u_x^1=0$ and $u_x^2=1$, and $u_x^1=2=u_x^2$. 

For any $\tilde{w} \in \Wcal_{x,y}^o$, for any $w \in \Acal((x,q),(y,r))$ such that $F_5(w)=\tilde{w} $, we have that $n_z(\tilde{w} )=n_z(w)$ for all $z \in \Tcal_L$ due to the insertion of the ghost pairings at $x$ and $y$. Together with the considerations above, we deduce that, 
\begin{equation} \label{eq: weightfourlinks}
\mu_{\Tcal_L}(w) 
= \begin{cases} \lambda^4 \mu_{\Tcal_L}(\tilde{w} ) \frac{m_{\{x,q\}}(\tilde{w} )!}{(m_{\{x,q\}}(\tilde{w} )+2)!} \frac{m_{\{y,r\}}(\tilde{w} )!}{(m_{\{y,r\}}(\tilde{w} )+2)!} & \text{if } \{x,q\} \neq \{y,r\},  \\
\lambda^4 \mu_{\Tcal_L}(\tilde{w} ) \frac{m_{\{x,q\}}(\tilde{w} )!}{(m_{\{x,q\}}(\tilde{w} )+4)!}  & \text{if } \{x,q\} = \{y,r\}.
\end{cases}
\end{equation}
Recalling the definition of the partition function in \eqref{eq: definition_partition_function} and using \eqref{eq:cardinalities_set_four_links} and \eqref{eq: weightfourlinks}, we obtain for any pair of directed edges $(x,q)$, $(y,r) \in \Ecal_L$ that
$$
\begin{aligned}
\mu_{\Tcal_L}\big(\Acal((x,q),(y,r))\big) 
& = \sum_{\tilde{w} \in \Wcal_{x,y}^o} \mu_{\Tcal_L}\big(\Acal((x,q),(y,r)) \cap \{F_5(w)=\tilde{w} \} \big) = \begin{cases} \lambda^4 \, \Z_{\T_L}(x,y) & \text{if } (x,q) \neq (y,r), \\
 \frac{\lambda^4}{4} \, \Z_{\T_L}(x,y) & \text{if } (x,q) = (y,r). 
%0 & \text{if } \{x,y\} \cap (\Tcal_L \setminus \T_L) \neq \emptyset. 
\end{cases}
\end{aligned}
$$
This concludes the proof of the lemma.
\end{proof}

We now have all ingredients to prove Proposition \ref{prop: central quantity expansion}.

\begin{proof}[Proof of Proposition \ref{prop: central quantity expansion}]
Fix a vector of real numbers $\boldsymbol{h}=(h_x)_{x \in \Tcal_L}$ and $\varphi \in \R$. For any $i \in \N_0$, let
$$
\Ccal^{(i)}(\boldsymbol{h}) := \mu_{\Tcal_L}\Big(\mathbbm{1}_{\{U=i\}} \prod_{z \in \Tcal_L} h_z^{\frac{u_z}{2}} \prod_{x \in \T_L } \Big(\frac{1}{2}\Big)^{\frac{\alpha_x}{2}} \mathbbm{1}_{\Wcal_{\Tcal_L}^{vert}}\Big),
$$
where $U(w):= \frac{1}{2}\sum_{z \in \Tcal_L}u_z(w)$ is the total number of unpaired endpoints of links in $w \in \Wcal_{\Tcal_L}^{vert}$ divided by two. Using that $
\Ccal^{(0)}(\boldsymbol{h}) = \mu_{\Tcal_L}(\tilde{\Wcal}_{\Tcal_L}^o)=\Z_{\T_L}^\ell
$ and that $\Ccal^{(i)}(\boldsymbol{h})=0$ for any $i \in 2\N_0+1$, we have that $$
\mathscr{Z}(\varphi \boldsymbol{h}) = \Z_{\T_L}^\ell + \varphi^2 \, \Ccal^{(2)}(\boldsymbol{h}) + o(\varphi^2).$$ 
%Note that $\Ccal^{(i)}(\boldsymbol{h})=0$ for all $i \in 2\N_0 +1$ since by definition the number of unpaired endpoints of links of colours $1$ and $2$ are equal in any configuration $w \in \Wcal_{\Tcal_L}^{vert}$ and each path is either closed or open, i.e. it has either zero or precisely two unpaired endpoints. Thus, $\sum_{z \in \Tcal_L} u_z(w) \in 4\N$ implying that $U(w)$ is even for any $w \in \Wcal_{\Tcal_L}^{vert}$.
Further,
\begin{equation} \label{eq:secondorderterm}
\begin{aligned}
\Ccal^{(2)}(\boldsymbol{h}) & = \sum_{\{x,y\} \in \E_L} \mu_{\Tcal_L}\Big(\Acal(\{x,y\})\Big) h_x h_y +\frac{1}{2} \sum_{x \in \T_L} \mu_{\Tcal_L}\Big(\Acal(\{x,x+\boldsymbol{e}_{d+1}\})\Big) h_x h_{x+\boldsymbol{e}_{d+1}} \\
&  \qquad + 2 \sum_{\{x,y\} \in \E_L} \mu_{\Tcal_L}\Big(\tilde{\Acal}^1(\{x,y\}\Big) h_xh_y + \sum_{\{(x,q),(y,r)\} \subset \Ecal_L} \mu_{\Tcal_L}\Big(\Acal((x,q),(y,r))\Big) h_q h_r.
\end{aligned}
\end{equation}
Note that the fourth sum in the right hand-side of \eqref{eq:secondorderterm} is over all \textit{unordered} pairs of (not necessarily distinct) \textit{directed} edges. Applying Lemma \ref{lemma: polynomial expansion} this sum can be rewritten as follows,
$$
\begin{aligned}
& \sum_{\{(x,q),(y,r)\} \subset \Ecal_L} \mu_{\Tcal_L}\Big(\Acal((x,q),(y,r))\Big) \, h_q h_r \\
& =\frac{1}{2} \sum_{\substack{x,y \in \T_L: \\ x \neq y}} \sum_{\substack{q,r \in \Tcal_L: \\ \{x,q\},\{y,r\} \in \Ecal_L}} \mu_{\Tcal_L}\Big(A((x,q),(y,r))\Big) \, h_q h_r + \frac{1}{2} \sum_{x \in \T_L} \sum_{\substack{q,r \in \Tcal_L: \\ \{x,q\},\{y,r\} \in \Ecal_L, q \neq r}} \mu_{\Tcal_L}\Big(A((x,q),(x,r))\Big) \,  h_q h_r\\
& \qquad \qquad \qquad \qquad + \sum_{x \in \T_L} \sum_{\substack{q \in \Tcal_L: \\ \{x,q\} \in \Ecal_L}}  \mu_{\Tcal_L}\Big(A((x,q),(x,q))\Big) \, h_q^2 \\
& = \frac{\lambda^4}{2} \sum_{x,y \in \T_L} \Z_{\T_L}(x,y) (\triangle^*h)_x \, (\triangle^*h)_y - \frac{\lambda^4}{4} \sum_{x \in \T_L} \Z_{\T_L}(x,x) \sum_{\substack{q \in \Tcal_L: \\ \{x,q\} \in \Ecal_L}} h_q^2.
\end{aligned}
$$

Applying Lemma \ref{lemma: polynomial expansion} also to the first three terms of \eqref{eq:secondorderterm}, we obtain that $\Ccal^{(2)}(\boldsymbol{h})=Z^{(2)}(\boldsymbol{h})$, where $Z^{(2)}(\boldsymbol{h})$ is defined in \eqref{eq:secondorderterm}. This concludes the proof of the proposition.
\end{proof}

\subsection{Key Inequality}
\label{sect:proofinequailty}
Before stating the Key Inequality we state a proposition which is a direct application of Proposition 
\ref{prop:chessboardabstract}.

\begin{proposition}\label{prop:chessboardscheme}
Let $\boldsymbol{h} = (h_z)_{z \in \mathcal{T}_L}$ be a real-valued vector
such that $|h_z| \leq 1$ for every $z \in \mathcal{T}_L$. For any $x \in \T_L$ define the new real-valued vector
$\boldsymbol{h}^{x} = ( h_z^{x})_{z \in \mathcal{T}_L}$ by 
$$
\forall z \in \mathcal{T}_L \quad \quad h_z^{x} :
= 
\begin{cases}
h_x  &\mbox{ if $z \in \T_L$}, \\
h_{x+\boldsymbol{e}_{d+1}}  & \mbox{ if $z \in \T_L^{(2)}$}.
\end{cases}
$$
We have that
\begin{equation}\label{eq:weakGD}
\mathscr{Z}( \boldsymbol{h} )
\leq 
\Big ( \, \, \, \prod_{x \in \T_L}   \mathscr{Z}
\big  ( \boldsymbol{h}^x  \big ) \, \, \, \Big )^{\frac{1}{|\T_L|}}.
\end{equation}
\end{proposition}
The above proposition follows immediately from Proposition \ref{prop:chessboardabstract}
after observing that we can write, 
$$
\mathscr{Z}( \boldsymbol{h} ) = \mu_{\Tcal_L}\bigg(\prod_{x \in \T_L} f_{\boldsymbol{h},x}^{[x]}\bigg),
$$
for a sequence of real-valued functions $(f_{\boldsymbol{h},x})_{x \in \T_L}$ that all have support $\{o\}$ and where $f_{\boldsymbol{h},x}^{[x]}$ is defined as the reflection of $f_{\boldsymbol{h},x}$ to the site $x$, see the definition in Section \ref{sect:Chessboardscheme}.

%The proof is then analogous to  Proposition 4.7 in \cite{T}. Indeed, we observe that
%\begin{proof}
%Define
%\begin{equation}
%\begin{aligned}
%\forall x \in \T_L \qquad f_{\boldsymbol{h},x} & := (h_x)^{\frac{u_o}{2}} \, (h_{x+\boldsymbol{e}_{d+1}})^{\frac{u_{o+\boldsymbol{e}_{d+1}}}{2}} \, 2^{-\alpha_o} \, \mathbbm{1}_{\{g_o=g_{o+\boldsymbol{e}_{d+1}}=0, \, u_o^i=u_{o+\boldsymbol{e}_{d+1}}^i =0 \, \forall i > 2\}} \\ 
%& \qquad \times \mathbbm{1}_{\{u_{o+\boldsymbol{e}_{d+1}}=n_{o+\boldsymbol{e}_{d+1}}, u_{o+\boldsymbol{e}_{d+1}}^1=u_{o+\boldsymbol{e}_{d+1}}^2\}} \, \prod_{\substack{y \in \Tcal_L: \\ y \sim o}}\mathbbm{1}_{\big\{\substack{\# 1-\text{links on } \{o,y\} \text{ unpaired at o } \\ = \# 2-\text{links on } \{o,y\} \text{ unpaired at o }}\big\}}
%\end{aligned}
%\end{equation}
%\end{proof}

We can now state the Key Inequality. For any vector of real numbers, $\boldsymbol{v}=(v_z)_{z \in \T_L}$, we define the \textit{discrete Laplacian} of $\boldsymbol{v}$ on the original torus,
$$
\forall x \in \T_L \qquad (\triangle v)_x := \sum_{\substack{y \in \T_L: \\ y \sim x}} (v_y-v_x).
$$

\begin{theorem}[Key Inequality]
\label{theo: keyinequality}
For any real-valued vector $\boldsymbol{v}=(v_x)_{x \in \T_L}$, we have that,
\begin{equation}
\label{eq: Key Inequality}
\sum_{x,y \in \T_L} \mathbb{G}_{\T_L}(x,y) (\triangle v)_x (\triangle v)_y \leq \big(1+2 \, \E\big(m_{\{o,\boldsymbol{e}_1\}}^{(1)}\big)\big) \,\sum_{\{x,y\} \in \E_L} (v_y-v_x)^2.
\end{equation}
\end{theorem}

\begin{proof}
Fix a real-valued vector $\boldsymbol{v} = (v_x)_{x \in \T_L}$, and let 
$\boldsymbol{h}^{\boldsymbol{v}} = ( h_x^{\boldsymbol{v}}   )_{x \in \mathcal{T}_L}$
be obtained from $\boldsymbol{v}$ as follows:
\begin{equation}\label{eq:hasfunctionofv}
\forall x \in \mathcal{T}_L  \quad \quad
h^{\boldsymbol{v}} _x : = 
\begin{cases}
v_x & \mbox{ if $x \in \T_L$,} \\
-2\, d \, v_{x- \boldsymbol{e}_{d+1}} & \mbox{ if $x \in \T_L^{(2)}$.} 
\end{cases}
\end{equation}
We will use the following identities:
\begin{enumerate}[(i)] 
\item $\sum\limits_{  \{x,y\} \in \E_L  }
h^{\boldsymbol{v}}_x h^{\boldsymbol{v}}_y 
+ 
\frac{1}{2}  \sum\limits_{ x \in \T_L }
h^{\boldsymbol{v}}_x h^{\boldsymbol{v}}_{x+ \boldsymbol{e}_{d+1}  } 
=
- \frac{1}{2} \sum\limits_{  \{ x,y \} \in \E_L  } 
\big (  
v_x - v_y
\big )^2$, 
\item  $(\triangle^* h^{\boldsymbol{v}})_x =  ( \triangle \boldsymbol{v} )_{x}$,
\item $\sum_{x \in \T_L} \sum_{\substack{q \in \Tcal_L : \\ \{x,q\} \in \Ecal_L}} (h^{\boldsymbol{v}}_q)^2 %=\sum_{x \in \T_L} \Big(\sum_{\substack{q \in \T_L : \\ \{x,q\} \in \E_L}} v_x^2 + 4d^2v_x^2\Big) 
=2d \, (1+2d) \, \sum_{x \in \T_L} v_x^2.$
\end{enumerate}
Recall the definition of $Z^{(2)}(\boldsymbol{h})$ in \eqref{eq: termorderphi2} for a vector $\boldsymbol{h}=(h_x)_{x \in \Tcal_L} \in \R^{\Tcal_L}$.
From (i), (ii), (iii) we deduce that,
\begin{equation}\label{eq:beforehomogenized}
\begin{aligned}
Z^{(2)}(\boldsymbol{h}^{\boldsymbol{v}})
& = - \frac{\lambda^2}{2} \Z_{\T_L}^\ell\sum\limits_{ \{x,y\} \in  \E_L} (v_y-v_x)^2 + 2\lambda^2 \mu_{\T_L}^\ell\big(m_{\{o,\boldsymbol{e}_1\}}^{(1)}\big) \sum\limits_{ \{x,y\} \in  \E_L} v_x v_y +
 \frac{\lambda^4}{2} \sum\limits_{x, y \in \T_L} \Z_{\T_L}(x,y) 
(\triangle \boldsymbol{v})_x \,
(\triangle \boldsymbol{v})_y \\
& \qquad \qquad \qquad - \frac{\lambda^4}{2} \Z_{\T_L}(o,o) \, d \, (1+2d) \sum_{x \in \T_L} v_x^2.
\end{aligned}
\end{equation}
Moreover, recall that, as defined in Section \ref{sect:Chessboardscheme},
for any original vertex $x \in \T_L$, 
$( \boldsymbol{h}^{\boldsymbol{v}})^x$ is defined as the vector 
which is obtained from $\boldsymbol{h}^{\boldsymbol{v}}$
by copying the value $h^{\boldsymbol{v} }_x = v_x$ at any original vertex and the value
$h^{\boldsymbol{v} }_{ x+ \boldsymbol{e}_{d+1}} = -2d v_x$
at any virtual vertex and deduce from this and from (\ref{eq:beforehomogenized})
that, 
\begin{equation}\label{eq:homogenized}
\forall \boldsymbol{v} = (v_z)_{z \in \T_L}  \quad 
\forall x \in \T_L, \quad  \quad Z^{(2)} \big (( \boldsymbol{h}^{\boldsymbol{v}})^{x} \big )
= \Big(2 \, \lambda^2 \, \mu_{\T_L}^\ell\big(m_{\{o,\boldsymbol{e}_1\}}^{(1)}\big) \, |\E_L| - \frac{\lambda^4}{2} \Z_{\T_L}(o,o) \, d \, (1+2d) \, |\T_L|\Big) \, v_x^2.
\end{equation}
Using Propositions \ref{prop: central quantity expansion} and \ref{prop:chessboardscheme}, a Taylor expansion and \eqref{eq:homogenized}, we have that, in the limit as $\varphi \rightarrow 0$,
\begin{align*}
\mathscr{Z}(\varphi \boldsymbol{h}^{\boldsymbol{v}}) & = \Z_{\T_L}^{\ell}
+ \varphi^2 Z^{(2)}(\boldsymbol{h}^{\boldsymbol{v}}) + o(\varphi^2) \leq \Big ( \prod_{x \in \T_L}
\mathscr{Z} \big ( (\varphi \boldsymbol{h}^{\boldsymbol{v}})^x \big )
 \Big )^{\frac{1}{|\T_L|}} \\
% & = \Big ( \prod_{x \in \T_L} \big ( \Z^{\ell} + \varphi^2 \, \big(2 \, \lambda^2 \, \mu^\ell\big(m_{\{o,\boldsymbol{e}_1\}}^{(1)}\big) \, |\E_L| - \frac{\lambda^4}{4} \Z(o,o) \, 2d \, (1+2d) \, |\T_L|\big) \, v_x^2+ o(\varphi^2) \big )    \Big )^{\frac{1}{|\T_L|}} \\
& = \Z_{\T_L}^{\ell} + \varphi^2 \Big(2 \, \lambda^2 \, \mu_{\T_L}^\ell\big(m_{\{o,\boldsymbol{e}_1\}}^{(1)}\big) \, d - \frac{\lambda^4}{2} \Z_{\T_L}(o,o) \, d \, (1+2d)\Big) \,  \sum_{x \in \T_L} v_x^2 + o(\varphi^2).
\end{align*}
Thus we proved that, in the limit as $\varphi \rightarrow 0$, 
$$
 \Z_{\T_L}^{\ell}
+ \varphi^2 Z^{(2)}(\boldsymbol{h}^{\boldsymbol{v}}) + o(\varphi^2) \leq   \Z_{\T_L}^{\ell} + \varphi^2 \, \Big(2 \, \lambda^2 \, \mu_{\T_L}^\ell\big(m_{\{o,\boldsymbol{e}_1\}}^{(1)}\big) \, d - \frac{\lambda^4}{2} \Z_{\T_L}(o,o) \, d \, (1+2d)\Big) \,  \sum_{x \in \T_L} v_x^2 + o(\varphi^2),
$$
where $\boldsymbol{h}^{\boldsymbol{v}}$ was defined in (\ref{eq:hasfunctionofv}) as a function of $\boldsymbol{v}$,
and this can only hold true if 
\begin{equation}\label{eq:2lessthan0}
Z^{(2)}(\boldsymbol{h}^{ \boldsymbol{v}})  \leq \Big(2 \, \lambda^2 \, \mu_{\T_L}^\ell\big(m_{\{o,\boldsymbol{e}_1\}}^{(1)}\big) \, d - \frac{\lambda^4}{2} \Z_{\T_L}(o,o) \, d \, (1+2d)\Big) \,  \sum_{x \in \T_L} v_x^2 .
\end{equation}
By replacing  (\ref{eq:beforehomogenized}) in the left-hand side of 
(\ref{eq:2lessthan0}), dividing the whole expression by $\frac{ \lambda^2}{2} \Z_{\T_L}^\ell$ and plugging in 
(\ref{eq: twopointfunction}), we deduce that, for any finite strictly positive $\lambda$,
$$
\begin{aligned}
\sum\limits_{x,y \in \T_L} \mathbb{G}_{\T_L}(x,y) (\triangle v )_x \, 
(\triangle v)_{y} &   
\leq  \sum\limits_{\{x,y\} \in \E_L} \big (v_y - v_x \big )^2+ 4 \, \E_{\T_L}\big(m_{\{o,\boldsymbol{e}_1\}}^{(1)}\big) \, \big(d \,  \sum_{x \in \T_L} v_x^2-\sum\limits_{\{x,y\} \in \E_L} v_xv_y\big) \\ 
& = \big(1+2 \, \E_{\T_L}\big(m_{\{o,\boldsymbol{e}_1\}}^{(1)}\big)\big) \, \sum\limits_{\{x,y\} \in \E_L} \big (v_y - v_x \big )^2.
\end{aligned}
$$
This concludes the proof of the theorem.
\end{proof}

\begin{remark}
The bound (\ref{eq:weakGD}) is reminiscent of the Gaussian Domination bound, which appears  in the framework of spin systems with continuous symmetry    \cite{Frohlich}, where a quantity which is analogous to our quantity $\mathscr{Z}( \boldsymbol{h} )$ is defined.  In that framework Gaussian Domination is referred
to the inequality $\mathscr{Z}( \boldsymbol{h} )  \leq \mathscr{Z}( \boldsymbol{0})$.
Our case differs from \cite{Frohlich}  since, as one can deduce from our Polynomial expansion,  Proposition \ref{prop: central quantity expansion}, 
no  Gaussian Domination can hold.
\end{remark}

\section{Upper bound on the Fourier sum}
\label{sect:upperboundonthefouriersum}
In this section we provide a uniform lower bound for the Ces\`{a}ro sum of the two-point function (recall the  definition in \eqref{eq: twopointfunction}). The result is presented in Theorem \ref{theo: Infrared bound} below. It is derived using Fourier transforms and the Key Inequality for the two-point function stated in Theorem \ref{theo: keyinequality}.  The lower bound depends on the numerical value of $\G(o,\boldsymbol{e}_1)$ and on the expected number of $1$-links on $\{o,\boldsymbol{e}_1\}$. For a comparison between these two quantities we refer the reader to Proposition \ref{proposition: uniformpositivity}. In the whole section $d>2$, $N \in \N_{>1}$, $\lambda \in \R^+$, $L \in 2\N$, $U: \N_0 \to \R_0^+$ and $v: \Z^d \to \R$ are fixed and satisfy the assumptions of Theorem \ref{theo: keyinequality}. We omit all sub-scripts here.

For $x \in \Z^d$, we denote by $N_x:= \sum_{n=0}^\infty \mathbbm{1}_{\{S_n=x\}}$ the number of visits at $x$ of a simple random walk $S=(S_n)_{n \in \N_0}$ in $\Z^d$ starting at the origin. We denote its probability measure and its expectation by $P^{\text{r.w.}}$ and $E^{\text{r.w.}}$. %, respectively.

\begin{theorem}
\label{theo: Infrared bound}
We have that
\begin{equation}
\frac{1}{|\T_L|} \sum_{x \in \T_L} \G(o,x) \geq \frac{\G(o,\boldsymbol{e}_1)+\G(o,o)}{2} -\big( 1+ 2\,\E(m_{\{o,\boldsymbol{e}_1\}}^{(1)})\big) \, C_L(d),
\end{equation}
where $(C_L(d))_{L \in \N}$ is a sequence of non-negative real numbers, which is defined in \eqref{eq: definitionsequence} below, whose limit $L \to \infty$ exists and satisfies 
\begin{equation} \label{eq: limitsequence}
 \lim_{L \to \infty} C_L(d) = \frac{2 E^{\text{r.w.}}(N_o)-1}{4d}.
\end{equation}
\end{theorem}
%\subsection{Fourier transforms}
To begin, we define the Fourier \textit{dual torus},
$$
\T_L^* := \big  \{\frac{2 \pi}{L} (n_1, \dots, n_d) \in \R^d : n_i \in (-\frac{L}{2}, \frac{L}{2}] \cap \mathbb{Z}  \, \, \forall i \in [d] \big\}.
$$
We denote the elements of $\T_L^*$ by $k=(k_1,\dots,k_d)$ and we keep using the notation $o$ for $(0,\dots,0) \in \T_L$ or $(0,\dots,0) \in \T_L^*$. 
Given a function $f \in l^2(\T_L)$, we define its Fourier transform,
\begin{equation}
\label{eq: fouriertransform}
\forall k \in \T_L^*, \quad \hat{f}(k):= \sum_{x \in \T_L} e^{-ik\cdot x}f(x).
\end{equation}
It follows from this definition that ,
\begin{equation}
\label{eq: inversefouriertransform}
\forall x \in \T_L, \quad f(x)=\frac{1}{|\T_L|} \sum_{k \in \T_L^*}e^{ik \cdot x}\hat{f}(k).
\end{equation}
We use the notation $\G(x)=\G(o,x)$ and we denote by $\hat{\G}(k)$ the Fourier transform of $\G(x)$. 
We start from the following equality which is an immediate consequence of \eqref{eq: fouriertransform} and of \eqref{eq: inversefouriertransform}.
\begin{lemma}
We have that,
\begin{equation}
\label{eq: startingpoint}
\frac{2}{|\T_L|} \sum_{x \in \T_L} \G(x) = \G(o)+\G(\boldsymbol{e}_1)-\frac{2}{|\T_L|} \sum_{k \in \T_L^* \setminus \{o\}} \cos^2\big(\frac{k_1}{2}\big) \,  \hat{\G}(k).
\end{equation}
\end{lemma}
\begin{proof}
To begin note that it follows from \eqref{eq: inversefouriertransform} that,
\begin{equation} \label{eq: fouriertransforme1}
\begin{aligned}
\G(\boldsymbol{e}_1) & = \frac{1}{|\T_L|} \sum_{k \in \T_L^*} e^{i k \cdot \boldsymbol{e}_1} \hat{\G}(k) 
& = \frac{1}{|\T_L|} \hat{\G}(o)+\frac{1}{|\T_L|} \sum_{k \in \T_L^* \setminus \{o\}} e^{i k \cdot \boldsymbol{e}_1} \hat{\G}(k)
\end{aligned}
\end{equation}
and that
\begin{equation} \label{eq: fouriertransformo}
\begin{aligned}
\G(o) & = \frac{1}{|\T_L|} \sum_{k \in \T_L^*} \hat{\G}(k) 
& = \frac{1}{|\T_L|} \hat{\G}(o)+\frac{1}{|\T_L|} \sum_{k \in \T_L^* \setminus \{o\}} \hat{\G}(k).
\end{aligned}
\end{equation}
By definition \eqref{eq: fouriertransform}, it holds that $\frac{1}{|\T_L|} \hat{\G}(o)=\frac{1}{|\T_L|} \sum_{x \in \T_L} \G(x)$.
Thus, by summing \eqref{eq: fouriertransforme1} and \eqref{eq: fouriertransformo}, we obtain that
\begin{equation}
\label{eq: zerofouriermode}
\begin{aligned}
\frac{2}{|\T_L|} \sum_{x \in \T_L} \G(x) & = \G(o)+\G(\boldsymbol{e}_1)- \frac{1}{|\T_L|} \sum_{k \in \T_L^* \setminus \{o\}}\big(1+e^{ik\cdot \boldsymbol{e}_1}\big)\, \hat{\G}(k).
\end{aligned}
\end{equation}
Since the term in the left-hand side of \eqref{eq: zerofouriermode} is real and since $\hat{\G}(k)$ is real, it follows that 
$$
\frac{1}{|\T_L|} \sum_{k \in \T_L^* \setminus \{o\}}\big(1+e^{ik\cdot e_1}\big)\, \hat{\G}(k) = \frac{2}{|\T_L|} \sum_{k \in \T_L^* \setminus \{o\}}\cos^2\big(\frac{k_1}{2}\big)\, \hat{\G}(k),
$$
where we used the equality $1+\cos(k_1)=2 \cos^2(\frac{k_1}{2})$. This concludes the proof.
\end{proof}
The goal is to bound away from zero uniformly in $L$ the quantity on the left-hand side of \eqref{eq: startingpoint}. %This quantity corresponds to the zero Fourier mode of the two point function.
The next proposition is an application of Theorem \ref{theo: keyinequality}.

\begin{proposition}[High frequency upper bound]
\label{prop: infrared bound}
We have that
\begin{equation}
\label{eq: infraredbound}
 \forall k \in \T_L^* \setminus \{o\} \quad \quad  \cos^2\big(\frac{k_1}{2}\big)\,\hat{\G}(k) \leq  \frac{\cos^2(\frac{k_1}{2})}{\varepsilon(k)} \, \big(1+2 \, \E(m_{\{o,\boldsymbol{e}_1\}}^{(1)})\big),
\end{equation}
where for any $k \in \T_L^*$,
$$
 \varepsilon(k) := 
2 \sum\limits_{j=1}^{d} \big ( \,  1 - \cos (k_j) \,  \big ).
$$
\end{proposition}
\begin{proof}
To begin, we  fix an arbitrary $k \in \T^*_L \setminus \{o\}$ and define the vector $\boldsymbol{v} = (v_x)_{x \in \T_L} \in \R^{\T_L}$ by,
$$
\forall x \in \T_L \quad \quad v_x  := \cos(\frac{k_1}{2})\cos( k \cdot x).
$$
We note that under this choice the following facts hold true,
\vspace{-0.3cm}
\begin{enumerate}[(i)]
   \setlength{\itemsep}{0pt}
    \setlength{\topsep}{0pt}
\item For any $x \in \T_L$,
$({\triangle v})_x  = - \varepsilon(k) \, \, v_x,$
\item $\sum_{\{x,y\} \in \E_L   } (v_y - v_x)^2 =  \varepsilon(k) \, \,  \sum_{x \in \T_L} v_x^2,$
\item %Let $F_{x} \in \ell^2(\T_L)$ be a function such that, for any $x \in \T_L$, $F_x = F_{-x}$.
$
\sum_{x, y \in \T_L} \, v_x  \, v_y  \,\mathbb{G}(x,y) = \hat{\mathbb{G}}(k) \, \sum_{x \in \T_L}  v_x^2.
$
\end{enumerate}
These computations are classical and a proof can be found for example in the appendix of \cite{T}. %In particular, it is used that for any $k \in \T_L^*$, the Fourier transform $\hat{\G}_{L,N,\lambda,U}(k)$ is real which follows from the definition of the Fourier transform and the symmetries of $\Z^d / L \Z^d$.
We now apply  (i) 
to the left-hand side  of \eqref{eq: Key Inequality} and (ii)  to the right-hand side of 
\eqref{eq: Key Inequality},  thus obtaining that 
$$
\varepsilon^2(k) \,   \sum\limits_{x,y \in \T_L} \, v_x \, v_y  \, \mathbb{G}(x,y)\,  \leq \, \varepsilon(k)\, \big(1+2 \, \E(m_{\{o,\boldsymbol{e}_1\}}^{(1)})\big) \,
 \sum\limits_{x \in \T_L} \,  v_x^2.
$$
Applying (iii) to the left-hand side of the previous inequality and dividing everything by $\varepsilon^2(k) \sum_{x \in \T_L} \cos^2(k \cdot x)$ gives \eqref{eq: infraredbound} and concludes the proof.
\end{proof}
We now have all ingredients for proving Theorem \ref{theo: Infrared bound}.
\begin{proof}[Proof of Theorem \ref{theo: Infrared bound}]
%To begin, we define the set of vectors $\Ncal := \{\pm \boldsymbol{e}_1, \dots, \pm \boldsymbol{e}_d\}$ and the function
%$$
%J(k):= \frac{1}{d} \sum_{j=1}^d \cos(k_j) = \frac{1}{2d} \sum_{\boldsymbol{e} \in \Ncal} e^{i \boldsymbol{e} \cdot k}.
%$$
To begin, we deduce from \eqref{eq: startingpoint} and \eqref{eq: infraredbound} that
$$
\begin{aligned}
\frac{1}{|\T_L|} \sum_{x \in \T_L} \G(x) %& = \frac{\G(o)+\G(\boldsymbol{e}_1)}{2} - \frac{1}{|\T_L|}  \sum_{k \in \T_L^* \setminus \{o\}} \cos^2\big(\frac{k_1}{2}\big) \, \hat{\G}(k) \\
& \geq   \frac{ \G(o)+\G(\boldsymbol{e}_1)}{2} -\big( 1+ 2 \, \E(m_{\{o,\boldsymbol{e}_1\}}^{(1)})\big) \, \frac{1}{|\T_L|}  \sum_{k \in \T_L^* \setminus \{o\}} \frac{\cos^2(\frac{k_1}{2})}{\varepsilon(k)}.
\end{aligned}
$$
We set
\begin{equation} \label{eq: definitionsequence}
C_L(d):= \frac{1}{|\T_L|}  \sum_{k \in \T_L^* \setminus \{o\}} \frac{\cos^2(\frac{k_1}{2})}{\varepsilon(k)}.
\end{equation}
Using the fact that the sum in \eqref{eq: definitionsequence} is Riemann and using the identity $\cos^2(\frac{k_1}{2})=\frac{1}{2}(1+\cos(k_1))$, we obtain
\begin{equation} \label{eq:laststepderivation}
\lim_{L \to \infty} C_L(d)=\frac{E^{\text{r.w.}}(N_{\boldsymbol{e}_1})+E^{\text{r.w.}}(N_o)}{4d}=\frac{2E^{\text{r.w.}}(N_o)-1}{4d},
\end{equation}
where in the last step we used that $E^{\text{r.w.}}(N_{\boldsymbol{e}_1})=E^{\text{r.w.}}(N_o)-1$. For a detailed derivation of \eqref{eq:laststepderivation} see also \cite[Section 5]{T}.
This proves \eqref{eq: limitsequence} and concludes the proof.
\end{proof}

\section{Proof of Theorem \ref{theo:maintheorem}}
\label{sect:proofoftheorem1.1}
In this section we present the proof of Theorem \ref{theo:maintheorem}. The idea of the proof is to first use  our results from Sections \ref{sect:reflectionpositivityandchessboardestimate} to \ref{sect:upperboundonthefouriersum} to derive statements that are similar to the ones in Theorem \ref{theo:maintheorem}, but in the RPM. The main results for the RWLS will then follow from the equivalence theorem, Theorem \ref{theo: equivalence}.
Recall our convention on the constant $c$, which is always positive and finite and may differ from line to line. 

\begin{proof}[Proof of Theorem \ref{theo:maintheorem}] Let $d, N \in \N$ such that $d \geq 3$ and $N \geq 2$ and suppose that $v: \Z^d \to \R$ is tempered and separable. Recall from Section \ref{sect:upperboundonthefouriersum} that we denote by $E^{r.w.}(N_o)$ the expected number of visits at the origin of a simple random in $\Z^d$ starting at the origin. We let $M>\frac{2 E^{\text{r.w.}}(N_o)-1}{4d}$ and we set $\tilde{M}:=M - \frac{2 E^{\text{r.w.}}(N_o)-1}{4d}>0$. Let $U:\N_0 \to \R_0^+$ be good with  large enough range $R$ and let $\lambda_0 \in (0,\infty)$ be such that for any $\lambda \geq \lambda_0$, it holds that
\begin{equation} \label{eq:uniformpositivecesaro}
\begin{aligned}
\liminf_{\substack{L \to \infty \\ L \in 2\N}} \frac{1}{|\T_L|} \sum_{x \in \T_L} \G(o,x) & \geq \tilde{M}>0.
\end{aligned}
\end{equation}
% For any potential v there exists a weight function v with infinite range, see the examples in the introdcution.
The existence of such a weight function $U$ and parameter $\lambda_0$ follows from Theorem \ref{theo: Infrared bound} and from the lower bound \eqref{eq:lowerboundG(o,e1)} for the term $\G(o,\boldsymbol{e}_1)$. We now fix $\lambda > \lambda_0$.
By site-monotonicity, Proposition \ref{prop:monotonicity},  and by Lemma \ref{lemma:Chessboardapplicationno} there exists $c >0$ such that for any $L \in 2\N$ and $x,y \in \T_L$, we have that 
\begin{equation} \label{eq:twopointbounded}
\G(x,y) \leq \G(o,\boldsymbol{e}_1) \leq \E(n_o^2) \leq c.
\end{equation} 
Recall the definition of $\tilde{\Ncal}_{x,y}: \tilde{\Wcal} \to \N_0$ in \eqref{eq:tildeNcal}.
From the monotonicity result \cite[Theorem 2.2]{L-T-quantum} and from \eqref{eq:twopointbounded}
we
 deduce that there exist $m \in (0, 1)$ and $c>0$ such that, for any $L \in 2 \mathbb{N}$ and any $x,y \in \mathbb{T}_L$  satisfying $|x_i-y_i| \in 2\N+1$ for all $i \in [d]$ and such that 
  $|x-y| \leq m L $,   we have that $\G(x,y) \geq c$, from which we deduce 
using  \eqref{eq: relationnumberloops} 
  that there exists $c^\prime \in (0, \infty)$ such that 
  for each such pair of sites $x, y \in \mathbb{T}_L$,
 \begin{equation}
 \label{eq: positiveG2}
\E(\tilde{\Ncal}_{x,y}) >c \, \G(x,y)^4 \geq c^\prime.
 \end{equation}
The Cauchy-Schwarz inequality now implies that 
\begin{equation} \label{eq:positiveprobkappa}
\P(\tilde{\Ncal}_{x,y} >0) \geq \frac{\E(\tilde{\Ncal}_{x,y})^2}{\E(\tilde{\Ncal}_{x,y}^2)}>c,
\end{equation}
where in the last step we used \eqref{eq: positiveG2} and that $\E(\tilde{\Ncal}_{x,y}^2) \leq \E(n_x^2n_y^2) < c$ by Lemma \ref{lemma:Chessboardapplicationno}. Recall from Lemma \ref{lemma:example} that $\Ncal_{x,y}:\Omega \to \N_0$, defined in \eqref{eq:Ncal}, is such that $\Ncal_{x,y} \sim \tilde{\Ncal}_{x,y}$. From \eqref{eq: positiveG2}, \eqref{eq:positiveprobkappa}
 and from Theorem \ref{theo: equivalence} we can thus deduce that
 \begin{equation} \label{eq:expgreaterthanconstant}
 E(\Ncal_{x,y})>c,
 \end{equation}
and
\begin{equation} 
\mathcal{P} \big ( \exists n \in \{1,  \ldots, |\omega| \}  \, : \, x, y \in \Gamma^{(n)}  \big   ) = \Pcal(\Ncal_{x,y}>0) >c,
 \end{equation}
where we recall that the expectation in \eqref{eq:expgreaterthanconstant} refers to the expectation in the RWLS. This proves \eqref{eq:twopointBEC2}. It remains to prove \eqref{eq:twopointBEC1}. Let $x \in \T_L$. For any $\omega \in \Omega$, let $K_x(\omega)$ be the total length of r-o-loops in $\omega$ visiting $x$, namely, for any $\omega=(\ell_1,\dots,\ell_{|\omega|}) \in \Omega$, we let
$$
K_x(\omega):= \sum_{j=1}^{|\omega|} \mathbbm{1}_{\{x \in \ell_j\}}(\omega),
$$
be the local time at $x$ for the RWLS.

For any $\omega=(\ell_1,\dots,\ell_{|\omega|})  \in \Omega$, 
recall from (\ref{eq:firstloopx}) that  $\Gamma_x(\omega)$  denotes the first loop of $\omega$ that contains $x$.
Then, for any $k \in \N$, we have that
\begin{equation} \label{eq:Gamma1}
\begin{aligned}
\sum_{y \in \T_L} E(\Ncal_{x,y} \mathbbm{1}_{\{n_x \leq k\}}) & \leq  E(K_x \mathbbm{1}_{\{n_x \leq k\}})   \leq \sum_{n=1}^k E\big(K_x \mathbbm{1}_{\{\text{precisely } n \text{ loops contain } x \}}\big) 
% = \sum_{n=1}^k \sum_{i=1}^n  E\big(|\Gamma_x^i| \mathbbm{1}_{\{\text{precisely } n \text{ loops contain } x \}}\big) 
%& = \sum_{n=1}^k n E\big(|\Gamma_x^1| \mathbbm{1}_{\{\text{precisely } n \text{ loops contain } x \}}\big) \\
 \leq E\big(|\Gamma_x|\big) \, \frac{k(k+1)}{2}.
\end{aligned}
\end{equation}
Further, there exists $c \in \R^+$ such that for any $y \in \T_L$ such that $|x-y| \leq mL$, it holds that
\begin{equation} \label{eq:verify}
\begin{aligned}
E(\Ncal_{x,y} \mathbbm{1}_{\{n_x \leq k\}}) & = E(\Ncal_{x,y}) - E(\Ncal_{x,y} \mathbbm{1}_{\{n_x > k\}}) > c - E(\Ncal_{x,y} \mathbbm{1}_{\{n_x > k\}}),
\end{aligned}
\end{equation}
where we used \eqref{eq:expgreaterthanconstant} in the last step. The second term in \eqref{eq:verify} converges to $0$ as $k \to \infty$, which can be seen by applying the equivalence theorem, Theorem \ref{theo: equivalence}, the Cauchy-Schwarz inequality, the Markov inequality and Lemma \ref{lemma:Chessboardapplicationno}. Indeed,
\begin{equation} \label{eq:limit}
\begin{aligned}
E(\Ncal_{x,y} \mathbbm{1}_{\{n_x > k\}})& =\E(\tilde{\Ncal}_{x,y} \mathbbm{1}_{\{n_x > k\}})  \leq \E(\tilde{\Ncal}_{x,y}^2)^{\frac{1}{2}} \, \P( n_x >k)^{\frac{1}{2}}  \leq \bigg(\E(n_x^2) \,\frac{\E( n_x)}{k}\bigg)^{\frac{1}{2}} =: c(k),
\end{aligned}
\end{equation}
where $\lim_{k \to \infty} c(k) =0$. 
From \eqref{eq:Gamma1}, \eqref{eq:verify} and \eqref{eq:limit} we thus deduce that for any $k \in \N$, 
$$
E(|\Gamma_x|) > \frac{2(c-c(k))}{k(k+1)} \, \tilde{m} \, L^d
$$
for some $\tilde{m} \in (0,\infty)$. 
We now choose $k_0 \in \N$ large enough such that $c(k_0) < c$ and the proof of the theorem is concluded.
\end{proof}
\appendix

\section{Appendix}
\label{sect:appendix}

\subsection{The interacting Bose gas and random loops}
\label{sect:extensionsBEC}
We introduce the partition function of the random loop model corresponding to the Bose gas in $\mathbb{T}_L$ with interaction potential $v_L: \mathbb{T}_L \times \mathbb{T}_L \rightarrow \mathbb{R}$.
Given $n \in \mathbb{N}$, the number of particles, and $\beta \in [0, \infty)$, the inverse temperature,  we define the partition function,
\begin{equation}
\label{eq:partitionbose}
Z^{bose}_{n, \beta} := \frac{1}{n!} \,   \sum\limits_{ x = (x_1, \ldots, x_n) \in {( \mathbb{T}_L)}^n   }
\sum\limits_{ \pi \in S_n  }
%\sum\limits_{  \substack{ (\ell_1, \ldots, \ell_n) \in \mathcal{L}^n : \\ \, \, \ell_i(0) = x_i  \forall i %\in [n]}    }
 \, \prod_{i=1}^{n} % \frac{1}{|\ell_i|}
  \int dW^\beta_{x_i, x_{ \pi(i)} } (\omega^i(t)) \, \, \exp \Big (    \sum\limits_{  1 \leq    i ,  j \leq n }\int_{ 0}^\beta v_L\big ( \omega^i(t),  \omega^j (t) \big )  dt \Big )
\end{equation}
where
$S_n$ is the group of permutations of  $n$ integers,
 $d  W^\beta_{x,y}$ is the measure of a continuous-time simple random walk in $\mathbb{T}_L$ of time-length $\beta$  which starts at $x$ and  ends at $y$,
and is defined in the measure space of continuous-time C\`adl\`ag trajectories
$\omega : [0, \beta] \rightarrow \mathbb{T}_L$ which start from $x$ and end at $y$,
the normalisation is such that 
$
\int dW^\beta_{x,y} 
$
equals the probability that a continuous time simple random walk in $\mathbb{T}_L$ of time $\beta$ starting from $x$ ends at $y$,
$v_L$ is the interaction potential which was defined in the introduction.
This  reformulation of the Bose gas can be derived from the  quantum analytic formulation
through the Feyman-Kac formula
 \cite{Ginibre}, see also 
\cite{UeltschiRelation2}
for the analogous derivation in continuous space. 
Since permutations naturally induce cycles, (\ref{eq:partitionbose}) can be viewed
as the partition function of a random loop model.
The grand canonical partition function  of the interacting Bose gas is then defined as,
\begin{equation}
Z^{bose,  g}_{\mu, \beta}  = \sum\limits_{n=0}^{\infty} e^{\mu n} Z^{bose}_{n, \beta},
\end{equation}
where $\mu \in \mathbb{R}$ is the chemical potential, an external parameter.
We formally obtain the RWLS considered in this paper
(in the special case $N= 2$ and weight function $U(n) = 1$ for any $n \in \mathbb{N}_0$)
if we replace the continuous time simple random walk of length $\beta$ by a single-step simple random walk trajectory.  More formally, 
replace $d {W}^\beta_{x,y}$  in (\ref{eq:partitionbose}) by $d \hat{W}^\beta_{x,y}$,
where  $d \hat{W}^\beta_{x,y}$
 is defined as the measure which assigns weight $\frac{1}{2d}$ to the  trajectory such that,
$$
\forall t \in [0, \beta] \quad 
\omega(t) =  
\begin{cases}
 x \quad  \mbox{ if  $t \in [0,  \beta)$, } \\
 y \quad  \mbox{ if $t = \beta$}, \\
\end{cases}
$$
and weight zero to all the other trajectories.
In other words,  the particle spends time $\beta$  at the vertex $x$ and then jumps to a nearest neighbour, $y$, with uniform probability.
We now make this connection formal by setting  for simplicity $\beta=1$. By replacing 
$d W^1_{x,y}$
by $d\hat{W}^1_{x,y}$ in (\ref{eq:partitionbose}),  we thus obtain 
\begin{align*}
\hat{Z}^{bose}_{n} & := 
\frac{1}{n!} \,   \sum\limits_{ x = (x_1, \ldots, x_n) \in { \mathbb{T}_L}^n   }
\sum\limits_{ \pi \in S_n  }
 \, \prod_{i=1}^{n} % \frac{1}{|\ell_i|}
  \int d \hat{W}^\beta_{x_i, x_{ \pi(i)} } (\omega^i) \, \, \exp \big (    \sum\limits_{  1 \leq    i ,  j \leq n }\int_{ 0}^1 v_L( \omega^i(t),  \omega^j (t) )  dt \big ) \\
&
= \frac{1}{n!} \,   \sum\limits_{ x = (x_1, \ldots, x_n) \in { \mathbb{T}_L}^n   }
\sum\limits_{ \pi \in S_n  }
(\frac{1}{2d})^n  \, \mathbbm{1}_{ \{ |x_{\pi(i)} - x_i|_1 = 1 \forall i \in [n]   \} }
\exp \big (\sum\limits_{  1 \leq  i , j \leq n  }   v_L ( x_i , x_j  )  \big )   \\ 
& = 
\sum\limits_{k=1}^{n}  \frac{1}{k!}  \sum\limits_{  (\ell_1, \ldots, \ell_k) \in \mathcal{L}^k   }
  \mathbbm{1}_{ \big  \{  \sum\limits_{i=1}^{k} | \ell_i| = n    \big \}  } \prod_{  i=1 }^{k}
 \frac{{{(2d)}}^{-|\ell_i|}}{|\ell_i|}  
 \exp \big ( \mathcal{V}(\ell_1, \ldots, \ell_k) \big ),
\end{align*}
where $\mathcal{V}$ was defined in (\ref{eq:definitioninteraction}).
For the last identity, for any pair $(x, \pi)$ appearing in the sum
 we summed over all the loop configurations $(\ell_1, \ldots, \ell_k) \in \mathcal{L}^k$ with the same loop structure as in 
$(x, \pi)$ and divided by the number of such configurations. 
After simplifying the resulting factorials and summing over all configurations $(x, \pi)$
we obtain the last expression in the previous display.
Finally,  by considering the grand canonical ensemble and setting 
$\lambda =\frac{1}{2d}  e^{ \mu}$  we obtain,
\begin{align}
\label{eq:bosegasref}
\hat{Z}^{bose,  g}_{\mu}  &:  = \sum\limits_{n=0}^{\infty}  e^{\mu n} \hat{Z}^{bose}_{n} \\ & = 
\sum\limits_{k=0}^{\infty} \frac{1}{k!}   \sum\limits_{  (\ell_1, \ldots, \ell_k) \in \mathcal{L}^k   } 
 \prod_{  i=1 }^{n}
 \frac{ \lambda^{|\ell_i|} }{|\ell_i|}
\exp \big (   \mathcal{V}( \ell_1, \ldots, \ell_k   )  \big ) \\
& =  \mathcal{Z}_{L, U, v, N, \lambda},
\end{align}
where $  \mathcal{Z}_{L, U, v, N, \lambda}$ was defined in the introduction, here $N=2$, $U(n) = 1$ for any $n \in \mathbb{N}_0$
and $v : \mathbb{Z}^d  \rightarrow \mathbb{R}$ is related to $v_L$ by (\ref{eq:periodicpotential}).

\subsubsection{Extensions}
\label{sect:extensions}
Our proof technique would allow us to prove the occurrence of BEC for random loop models which are even closer to the Bose gas, (\ref{eq:partitionbose}), than (\ref{eq:bosegasref}).
Indeed, the  only  limitation of our technique is that the distance between two consecutive particles in the same loop is at most one.  
This is a necessary condition for reflection positivity, which is employed in a crucial way.
For example, one may consider the following model,
\begin{multline}
\label{eq:partitionbosemodified}
{\tilde{Z}}^{bose}_{n, \beta} := \frac{1}{n!} \,   \sum\limits_{ x = (x_1, \ldots, x_n) \in {( \mathbb{T}_L)}^n   }
\sum\limits_{ \pi \in S_n  }
%\sum\limits_{  \substack{ (\ell_1, \ldots, \ell_n) \in \mathcal{L}^n : \\ \, \, \ell_i(0) = x_i  \forall i %\in [n]}    }
\mathbbm{1}_{\big \{  
  \forall i \in [n ] \, \,  | x_{\pi(i)} - x_i |_1 \leq 1   \big \} } \\
\times  \, \prod_{i=1}^{n} % \frac{1}{|\ell_i|}
  \int dW^\beta_{x_i, x_{ \pi(i)} } (\omega^i) \, \, \exp \Big (    \sum\limits_{  1 \leq    i \leq  j \leq n }\int_{ 0}^\beta v_L \big ( \omega^i(t),  \omega^j(t) \big   )  dt \Big ),
\end{multline}
which is a minor modification of the Bose gas, (\ref{eq:partitionbose}). The modification consists  in the introduction of the indicator
that the distance between $x_i$ and $x_{\pi(i)}$
is at most one for each particle $i$.
Our proof method would allow us to prove the occurrence
BEC for this model in the grand canonical ensemble,  
namely the occurrence of macroscopic loops
 for any $\beta \in [0, \infty)$ as long as the  chemical potential $\mu$. 
 Now if $\beta$ is `sufficiently' small, then
   (\ref{eq:partitionbosemodified})  is extremely `close' to the interacting Bose gas,  (\ref{eq:partitionbose}).
 Indeed, since $\beta$ is small,   it rarely happens in   (\ref{eq:partitionbose}) that $x_{\pi(i)}$ (corresponding to the position of the last step of a continuous time simple random walk of time $\beta$ starting from $x_i$) is at a distance greater than $1$ from $x_i$.
 Being the two models so close, the hope is that in the future it may be possible to deduce BEC for (\ref{eq:partitionbose})
 by a comparison with    (\ref{eq:partitionbosemodified}).

\subsection{Implications for the Spin O(N) model}
\label{sect:spinresults}
Our result has new implications on the Spin O(N) model,
which is `equivalent' to our random loop model 
when the weight function is chosen as in (\ref{eq:weightfunctionspin})
and the potential $v$ is zero. 

To begin, we define the Spin $O(N)$ model precisely. 
Let $\mathcal{G} = (\mathcal{V}, \mathcal{E})$ be a finite undirected graph.
Fix an integer $N \in \mathbb{N}_{>0}$.
We denote by  ${\varphi} \in {(\mathbb{S}^{N-1})}^{\mathcal{V}}$
the spin configurations,
where $\mathbb{S}^{N-1}\subset\R^N$ is the unit sphere of dimension $N-1$. For example, $\mathbb{S}^{0} = \{-1,1\}$ and 
$\mathbb{S}^1 \subset \mathbb{R}^2$ is the unit circle. We will often write spin configurations as $\varphi=(\varphi_x)_{x\in\Vcal}$ where $\varphi_x = (\varphi_x^1, \ldots, \varphi_x^N) \in \mathbb{R}^N$ is the value of $\varphi$ at the vertex $x\in\Vcal$.
The hamiltonian of the spin $O(N)$ model is defined as 
\begin{equation}\label{eq:hamiltonian}
H_{\Gcal,N}(\varphi) = - \sum\limits_{ \{x,y\} \in \mathcal{E}}  \,    \varphi_x \cdot \varphi_y,
\end{equation}
where $\varphi_x \cdot \varphi_y$ denotes the usual inner product of two $N$-component vectors. For a parameter $\beta\geq 0$ known as the \emph{inverse temperature} the partition function at inverse temperature $\beta$ is given by 
\begin{equation}\label{eq:partition function}
Z_{\Gcal,N, \beta}^{spin}  = 
\Big  (  \,  \prod_{x \in \mathcal{V}} 
\int_{\mathbb{S}^{N-1}} d \varphi_x    \,   \Big  ) 
e^{- \beta H_{\Gcal,N}(\varphi)},
\end{equation}
where $d \varphi_x$ denotes the uniform %(Haar)
probability measure on $\mathbb{S}^{N-1}$, that is,
$\int_{\mathbb{S}^{N-1}} d \varphi_x = 1$.
We introduce an expectation operator $\langle \cdot \rangle_{\mathcal{G}, N, \beta}$ on functions 
$(\mathbb{S}^{N-1})^{\mathcal{V}} \rightarrow \mathbb{R}$, that assigns the value
\begin{equation}\label{eq:deffunctional}
{\langle f \rangle}_{ \Gcal, N, \beta} = 
\frac{1}{Z_{ \Gcal,N, \beta}^{spin}} \,  
\Big  (  \,  \prod_{x \in \mathcal{V}} 
\int_{\mathbb{S}^{N-1}} d \varphi_x    \, \Big ) 
f \big ( \, (\varphi_x)_{x \in \mathcal{V}} \,  \big ) \, e^{- \beta H_{\Gcal,N}(\varphi)}.
\end{equation}
We write ${\langle f \rangle}_{ L, N, \beta}$
when $\Gcal$ is the $d$-dimensional torus of side length $L$.

When translated back into the language of spins, our general result about macroscopic loops implies that the correlation functions of the form
$\langle \varphi_x^1 \varphi_x^2 \varphi_y^1 \varphi_y^2 \rangle _{L, N, \beta}$
are uniformly positive as long as $N \geq 2$, $d \geq 3$ and $\beta$ is large enough.
The transition from a regime of exponential decay of such correlation functions
(when $\beta$ is small) to a regime of uniform positivity of such correlation functions
(when $\beta$ is large) in dimensions $d \geq 3$ can be seen
as an alternative characterisation of the phase transition in the Spin O(N) model with $N \geq 2$.
As far as we know,   only in the special case $N=2$ 
such a point-wise positivity result can be deduced from the point-wise 
positivity of the correlations of the form $\langle \varphi_x^1  \varphi_y^1 \rangle _{L, N, \beta}$,
which was proved in  \cite{Frohlich}
 (see  \cite{Be-U} for a derivation, which uses Pfister's theorem).
Hence, our result on the Spin O(N) model, which is stated in the next theorem, is new when   $N>2$.
\begin{theorem}\label{theo:spintheorem}
Suppose that $d \geq 3$,  $N \geq 2$,  and $\beta$ is large enough,  then
there exists $c_1, c_2>0$ such that for any $L \in 2 \mathbb{N}$
and for any $x, y \in \mathbb{T}_L$ such that $\| x - y \|_\infty \leq c_2 \, L$,
\begin{equation}\label{eq:spinresult}
\langle \varphi_x^{1} \varphi_x^{2} \varphi_y^{1} \varphi_y^{2} \rangle_{L, N, \beta} \geq c_1.
\end{equation}
\end{theorem}
\begin{proof}
Using the same expansion as in 
 \cite[Proposition 2.3]{L-T-first}
   we deduce that,  under the same choice of the weight function,  $U(n) = \frac{\Gamma(\frac{N}{2})}{\Gamma( \frac{N}{2} + n    )}$ for any $n \in \mathbb{N}_0$,
   and in the absence of long range interactions,  i.e, $v(x) = 0$ for any $x \in \mathbb{Z}^d$, 
$$
\G_{  \mathbb{T}_L, U, v, N, \beta   } (x,y) = \beta^2  \langle    \varphi_x^{1} \varphi_x^{2} \varphi_y^{1} \varphi_y^{2}  \rangle_{L, N, \beta}.
$$
From
(\ref{eq: positiveG2}) 
we then deduce the desired result.
\end{proof}

\subsection{Examples of tempered and separable potentials}
\label{sect:appendixpotentials}
In this section we provide examples of potentials $v:\Z^d \to \R$ that are tempered and separable.

\begin{lemma} \label{lemma:examplesseparable}
Let $v^1, v^2: \Z^d \to \R$ be defined by 
$$
v^1(x)=\alpha_1 \mathbbm{1}_{\{x=0\}} - \beta_1 e^{-\iota |x|_1} \mathbbm{1}_{\{x \neq 0\}}, \qquad \qquad v^2(x)= \alpha_2 \mathbbm{1}_{\{x=0\}}  -\beta_2 |x|_1^{-s} \mathbbm{1}_{\{x \neq o\}},
$$ 
for some constants $\alpha_1, \alpha_2, \beta_1,\beta_2, \iota, s >0$ such that $\alpha_1 > \beta_1 \sum_{y \in \Z^d, y \neq o}e^{-\iota|y|_1}$,  $\alpha_2 > \beta_2 \sum_{y \in \Z^d, y \neq 0} |y|_1^{-s}$, and $s>d$. Then $v^1$ and $v^2$ are  tempered and separable.
\end{lemma}
\begin{proof}
It is easy to see that $v^1$ and $v^2$ are tempered. We now prove separability of $v^1$. Let $i \in [d]$ and let $x,y \in \mathbb{T}_L$ with $y_i < 0 < x_i$. We have that
$$
-v^1_L(x,y) 
%= \sum_{z \in \mathbb{Z}^d} v^1(y+Lz-x) = \beta_1 \sum_{z \in \mathbb{Z}^d} \prod_{j=1}^d e^{-\iota |x_j-y_j-L z_j|}
= \beta_1 \prod_{j=1}^d \sum_{z \in \mathbb{Z}} e^{-\iota |x_j-y_j -Lz|}.
$$
We can write the factor involving the $i$-th coordinate by
\begin{equation} \label{eq: first coordinate}
\begin{aligned}
\sum_{z \in \mathbb{Z}} e^{-\iota |x_i-y_i-Lz|} %& = \sum_{n \in \mathbb{N}} e^{-\iota|x_i-y_i-Ln|} + \sum_{n \in \mathbb{N}} e^{-\iota|x_i-y_i+Ln|} + e^{-\iota |x_i-y_i|} \\
%& = \sum_{n \in \mathbb{N}} e^{-\iota (Ln- |x_i-y_i|)} + \sum_{n \in \mathbb{N}} e^{-\iota (Ln+ |x_i-y_i|)} + e^{-\iota |x_i-y_i|} \\
%& = \sum_{n \in \mathbb{N}} e^{-\iota (Ln- |x_i-y_i|)} + \sum_{n \in \mathbb{N}_0} e^{-\iota (Ln+ |x_i-y_i|)} \\
%& = \sum_{n \in \mathbb{N}} e^{-\iota L n} e^{\iota|x_i|}e^{\iota|y_i|} + \sum_{n \in \mathbb{N}_0} e^{-\iota L n}e^{-\iota|x_i|}e^{-\iota|y_i|} \\
& = \frac{1}{e^{L \iota}-1} e^{\iota|x_i|}e^{\iota|y_i|} + \frac{e^{L \iota}}{e^{L \iota}-1} e^{-\iota|x_i|}e^{-\iota|y_i|} .
\end{aligned}
\end{equation}
For $j \in [d] \setminus \{i\}$,  we use complex Fourier transforms (see e.g. \cite{Tolstov}) and obtain that 
\begin{equation} \label{eq: remaining coordinates}
\sum_{z \in \mathbb{Z}} e^{-\iota |x_j-y_j -Lz|} = \sum_{z \in \mathbb{Z}} e^{-\iota \, |(|x_j-y_j| \mod L) -Lz|} %=s^L(x_j-y_j)
=\sum_{n = - \infty}^{\infty} \hat{f}(n) e^{i \frac{2 \pi n}{L}x_j} e^{-i \frac{2 \pi n}{L}y_j},
\end{equation}
where $\hat{f}$ is the Fourier transform of $f: \mathbb{R} \to [0,\infty)$, $f(z)=e^{-\iota L|z|}$ given by $\hat{f}(\xi)=\frac{2 \iota L}{\iota^2 L^2 + 4 \pi^2 \xi^2} \geq 0$.
Combining \eqref{eq: first coordinate} and \eqref{eq: remaining coordinates} we can write 
$$
\begin{aligned}
-v^1_L(x,y)  %& = -\beta \, \big(\sum_{z \in \mathbb{Z}} e^{-\iota |x_1-y_1 +Lz|} \big) \cdots \big(\sum_{z \in \mathbb{Z}} e^{-\iota |x_d-y_d +Lz|}\big) \\
%& = - \beta \, (c_1 e^{\iota|x_i|}e^{\iota|y_i|} + c_2 e^{-\iota|x_i|}e^{-\iota|y_i|})
%\prod_{j \neq i} \Big( \sum_{n = - \infty}^{\infty} \hat{f}(n) e^{i \frac{2 \pi n}{L}x_j }e^{-i \frac{2 \pi n}{L}y_j} \Big) \\
%& = -\beta \, c_1 \sum_{z \in \mathbb{Z}^{d-1}} \hat{f}(z_1) \cdots \hat{f}(z_{d-1}) e^{\iota|x_i|} e^{i \frac{2 \pi}{L} (x_1, \dots, x_{i-1}, x_{i+1}, \dots, x_d) \cdot z } e^{\iota|y_i|} e^{-i \frac{2 \pi}{L}(y_1, \dots, y_{i-1}, y_{i+1}, \dots, y_d) \cdot z} \\
%& \quad - \beta \, c_2 \sum_{z \in \mathbb{Z}^{d-1}} \hat{f}(z_1) \cdots \hat{f}(z_{d-1}) e^{-\iota|x_i|} e^{i \frac{2 \pi}{L} (x_1, \dots, x_{i-1}, x_{i+1}, \dots, x_d) \cdot z } e^{-\iota|y_i|} e^{-i \frac{2 \pi}{L}(y_1, \dots, y_{i-1}, y_{i+1}, \dots, y_d) \cdot z} \\
 = \int_{\mathbb{R}^{d-1}} \alpha(t,x) \cdot \tilde{\alpha}(t,y) d \nu(t),
\end{aligned}
$$
where $\alpha: \mathbb{R}^{d-1} \times \mathbb{T}_L \to \mathbb{C}^2$ is defined by 
$
\alpha(t,x)= (\alpha_j(t,x))_{j \in \{1,2\}}$
with 
$$
\alpha_1(t,x) = c_1 e^{\iota |x_i|} e^{i \frac{2 \pi }{L}(x_1, \dots, x_{i-1}, x_{i+1}, \dots, x_d) \cdot \lfloor t \rfloor}; \qquad \alpha_2(t,x) = c_1 e^{-\iota |x_i|} e^{i \frac{2 \pi}{L} (x_1, \dots, x_{i-1}, x_{i+1}, \dots, x_d) \cdot \lfloor t \rfloor}
$$
for some $c_1=c_1(\beta,\iota,L) \in \R^+$. The function $\tilde{\alpha}: \mathbb{R}^{d-1} \times \mathbb{T}_L \to \mathbb{C}^2$ is defined as in \eqref{eq: conjugate}. 
Here we use the notation $\lfloor t \rfloor =(\lfloor t_1 \rfloor, \dots, \lfloor t_{d-1} \rfloor) \in \mathbb{Z}^{d-1}$ for a vector $t=(t_1, \dots, t_{d-1}) \in \mathbb{R}^{d-1}$.
The finite, non-negative measure $d \nu(t)$ is given by
$$
d \nu(t) = \hat{f}(\lfloor t_1 \rfloor ) \cdots \hat{f}(\lfloor t_{d-1} \rfloor) dt,
$$
where $dt$ denotes the Lebesgue measure on $T=\mathbb{R}^{d-1}$. 
Hence, $v^1_L$ yields a representation of the form \eqref{eq: interaction}.

Separability of $v^2$ can be deduced from separability of $v^1$ since for $x, y \in \mathbb{T}_L$ it holds that
$$
\begin{aligned}
v^2_L(x,y) & = -\beta_2 \sum_{z \in \mathbb{Z}^d} \frac{1}{| x-y -Lz|_1^s} = -\beta_2 \sum_{z \in \mathbb{Z}^d} \frac{1}{\Gamma(s)} \int_0^\infty \mu^{s-1} e^{- \mu | x-y -Lz|_1} d\mu. %\\
%& = \frac{1}{\Gamma(s)} \int_0^\infty \mu^{s-1} v^L_{\mu, C}(x,y)  d\mu,
\end{aligned}
$$
\end{proof}

\section*{Acknowledgements} The authors thank the German Research Foundation (project number 444084038, priority program SPP2265) for financial support. AQ additionally
thanks the German Research Foundation through IRTG 2544 and the German Academic Exchange Service (grant number 57556281) for financial support.
The authors also thank the anonymous referee for carefully reviewing the paper.

%\section*{Declarations}
%\subsection*{Data availability}
%Data sharing not applicable to this article as no datasets were generated or analysed during the current study.

%\subsection*{Conflict of interest}
%The authors have no competing interests to declare that are relevant to the content of this article.

%%      ---------------------------------------------------------------------
%%      --------------------------- BIBLIOGRAPHY ----------------------------
%%      ---------------------------------------------------------------------

\end{document}